\documentclass[reqno]{amsart}

\usepackage{amsmath,amsthm,amssymb,bm}
\usepackage{hyperref}
\usepackage{a4wide}
\usepackage{cleveref}
\usepackage{graphicx,color}
\usepackage{subcaption}
\usepackage{tikz}
\usepackage{ytableau}
\numberwithin{equation}{section}

\newtheorem{thm}{Theorem}[section]
\newtheorem{lem}[thm]{Lemma}
\newtheorem{prop}[thm]{Proposition}
\newtheorem{cor}[thm]{Corollary}
\theoremstyle{definition}
\newtheorem{defn}[thm]{Definition}

\newtheorem{remark}[thm]{Remark}
\newtheorem{algo}{Algorithm}

\crefname{lem}{Lemma}{Lemmas}
\crefname{thm}{Theorem}{Theorems}
\crefname{prop}{Proposition}{Propositions}
\crefname{question}{Question}{Questions}
\crefname{defn}{Definition}{Definitions}
\crefname{conj}{Conjecture}{Conjectures}
\crefname{figure}{Figure}{Figures}
\crefname{cor}{Corollary}{Corollaries} 
\crefformat{equation}{(#2#1#3)}
\Crefformat{equation}{Equation #2(#1)#3}

\AddToHook{env/lem/begin}{\crefalias{thm}{lem}}
\AddToHook{env/prop/begin}{\crefalias{thm}{prop}}
\AddToHook{env/cor/begin}{\crefalias{thm}{cor}}
\AddToHook{env/defn/begin}{\crefalias{thm}{defn}}
\AddToHook{env/conj/begin}{\crefalias{thm}{conj}}
\AddToHook{env/question/begin}{\crefalias{thm}{question}}
\AddToHook{env/remark/begin}{\crefalias{thm}{remark}}

\newcommand\inv{\operatorname{inv}}
\newcommand\vi{\vec{i}}
\newcommand\va{\vec{a}}

\newcommand\vT{\vec{T}}
\newcommand\vS{\vec{S}}
\newcommand\SE{\operatorname{SE}}
\newcommand\NE{\operatorname{NE}}

\newcommand\tn{\tilde{\nu}}
\newcommand\ts{\tilde{\sigma}}

\newcommand\QQ{\mathbb{Q}}

\newcommand\TT{\mathcal{T}}

\newcommand{\ZZ}{\mathbb{Z}}

\newcommand\LL{\mathcal{L}}
\newcommand\vq{\mathbf{q}}
\newcommand\sym{\mathfrak{S}}

\renewcommand\vec[1]{\mathbf{#1}}
\newcommand\flr[1]{\left\lfloor #1\right\rfloor}
\newcommand\ceil[1]{\left\lceil #1\right\rceil}
\newcommand\norm[1]{\lVert #1\rVert}
\newcommand\Qbinom[3]{\genfrac{[}{]}{0pt}{}{#1}{#2}_{#3}}
\newcommand\qbinom[2]{\Qbinom{#1}{#2}{q}}

\newcommand\qHyper[5]{
  {}_{#1}\phi_{#2} \left(
    \begin{matrix}
      #3\\
      #4\\
    \end{matrix}
    ; #5
  \right)
}

\newcommand\TRIL[2]{
  \def\X{#1} \def\Y{#2}
  \foreach \i in {0,...,\X}
  {
    \draw[gray,very thin] (\i,0) -- (\i,\i);
    \draw[gray,very thin] (\i,\i) -- (\X,\i);
    \node at (\i,-.3) {\i};
  }
}

\newcommand\LHLL[2]{
  \def\X{#1} \def\Y{#2}
  \foreach \i in {0,...,\X}
  {
\pgfmathsetmacro{\m}{\Y};
    \draw[gray,very thin] (\i,0) -- (\i,\m);
  }
\pgfmathsetmacro{\m}{\Y-1};
  \foreach \x in {1,...,\X}
{ \foreach \j in {0,...,\m}
  {
\draw[gray,very thin] (\x-1,\j) -- (\x,\j);
}}
\pgfmathsetmacro{\m}{\Y-1};
 \foreach \i in {2,...,\X}
 \foreach \j in {0,...,\m} 
{\foreach \y in {2,...,\i}
{
\pgfmathsetmacro{\w}{\j+(\y-1)/\i};
\pgfmathsetmacro{\z}{\j+(\y-1)/(\i+1)};
\draw[gray,very thin] (\i-1,\w) -- (\i,\z);
}}
}

\newcommand\LHlabel[2]{
  \def\X{#1} \def\Y{#2}
  \foreach \i in {0,...,\X}
  {
\pgfmathsetmacro{\m}{\Y};
    \node at (\i,-.3) {\i};
  }
\pgfmathsetmacro{\m}{\Y+.5};
\pgfmathsetmacro{\m}{\Y};
  \foreach \j in {0,...,\m}
  {
    \node at (-.3,\j) {\j};
  }
}

\newcommand\LHLLL[2]{
  \LHLL{#1}{#2} \LHlabel{#1}{#2}
  \draw[gray,very thin] (0,#2) -- (#1,#2);
}

\title{Lecture hall graphs and the Askey scheme}
 
\author{Sylvie Corteel}
\address{(Sylvie Corteel) CNRS, IRIF Université Paris Cité, Paris, France and Department of Mathematics, University of California Berkeley, USA}
\email{corteel@irif.fr}

\author{Bhargavi Jonnadula}
\address{(Bhargavi Jonnadula) Mathematical Institute, University of Oxford, Oxford OX2 6GG, UK}
\curraddr{Nu Quantum Ltd., Cambridge, UK}
\email{bhargavi.jonnadula@nu-quantum.com}

\author{Jonathan P. Keating}
\address{(Jonathan P. Keating) Mathematical Institute, University of Oxford, Oxford OX2 6GG, UK}
\email{keating@maths.ox.ac.uk}

\author{Jang Soo Kim}
\address{(Jang Soo Kim) Department of Mathematics,
Sungkyunkwan University (SKKU), Suwon, South Korea}
\email{jangsookim@skku.edu}

\begin{document}

\begin{abstract}
  We establish, for every family of orthogonal polynomials in the
  \( q \)-Askey scheme and the Askey scheme, a combinatorial model for
  mixed moments and coefficients in terms of paths on the lecture hall
  graph. This generalizes the previous results of Corteel and Kim for
  the little \( q \)-Jacobi polynomials. We build these combinatorial
  models by bootstrapping, beginning with polynomials at the bottom
  and working towards Askey--Wilson polynomials which sit at the top
  of the \( q \)-Askey scheme. As an application of the theory, we provide
  the first combinatorial proof of the symmetries in the parameters of
  the Askey--Wilson polynomials.
\end{abstract}


\maketitle


\section{Introduction}

Orthogonal polynomials are classical objects that play a central role
in a wide range of areas of mathematics. Since the pioneering work of
Flajolet \cite{Flajolet1980} and Viennot \cite{ViennotLN, ViennotOP}
in the 1980s, researchers have found them to possess many interesting
combinatorial properties. See \cite{CKS, ViennotOP, Zeng2021} and
references therein for more details.

Orthogonal polynomials can be defined as a sequence
\( \{p_n(x)\}_{n\ge0} \) of polynomials with \( \deg p_n(x) =n \) such
that there is a linear functional \( \LL \) satisfying
\( \LL(p_n(x)p_m(x)) = K_n \delta_{n,m} \), where \( K_n\ne 0 \). Their 
\emph{moments \( \{\sigma_n\}_{n\ge0} \)} are defined by
\[
  \sigma_n = \frac{\LL(x^n)}{\LL(1)}.
\]
More generally, the \emph{mixed moments
  \( \{\sigma_{n,k}\}_{n,k\ge 0} \)} are defined by
\[
  \sigma_{n,k} = \frac{\LL(x^np_k(x))}{\LL(p_k(x)^2)}.
\]
By the orthogonality, we have
\[
  x^n = \sum_{k=0}^{n} \sigma_{n,k} p_k(x),
\]
which can be taken as the definition of the mixed moments. We also
define the \emph{coefficients} (or \emph{dual mixed moments})
\( \{\nu_{n,k}\}_{n,k\ge 0} \) by
\[
  p_n(x) = \sum_{k=0}^{n} \nu_{n,k} x^k.
\]

Viennot \cite{ViennotLN, ViennotOP} found a combinatorial
interpretation for $\sigma_{n,k}$ using Motzkin paths. To illustrate,
suppose that \( \{ p_n(x) \}_{n\ge 0} \) is a sequence of monic
orthogonal polynomials satisfying the 3-term recurrence
\[
p_{n+1}(x) = (x-b_n)p_n(x)-\lambda_{n}p_{n-1}(x).
\]
A \emph{Motzkin path} is a path on $\mathbb{Z}^2$ from $(a, b)$ to
$(c,d)$ consisting of up steps \( (1,1) \), horizontal steps
\( (1,0) \), and down steps \( (1,-1) \) that never go below the line
$y = 0$. The weight \( w(p) \) of a Motzkin path \( p \) is the
product of the weights of the steps in \( p \), where the weight of an
up step is always \( 1 \), the weight of a horizontal step at height
$k$ is $b_k$, and the weight of a down step starting at height $k$ is
$\lambda_k$. Viennot showed that $\sigma_{n,k}$ is the sum of
\( w(p) \) for all Motzkin paths \( p \) from \( (0,0) \) to
\( (n,k) \). For example,
\[
  \sigma_{3,1} = b_0^2+b_0b_1+b_1^2+\lambda_1+\lambda_2
\]
is the generating function for all Motzkin paths from \( (0,0) \) to
\( (3,1) \) as shown in \Cref{fig:motzkin}. Using this combinatorial
interpretation for \( \sigma_{n,k} \) one can show that the moments of
Hermite, Charlier, and Laguerre polynomials are generating functions
for perfect matchings, set partitions, and permutations, respectively.
Viennot also found a combinatorial model for $\nu_{n,k}$ using lattice
paths called Favard paths. 

\begin{figure}
  \hspace*{\fill}
    \begin{subfigure}[b]{0.19\textwidth}
\centering
\begin{tikzpicture}[scale = 0.65]
\draw(-0.25,-0.25) grid (3.5,2.5);
\foreach \x in {0,1,...,3} { \node [anchor=north] at (\x,-0.3) {\x}; }
\foreach \y in {0,1,2} { \node [anchor=east] at (-0.3,\y) {\y}; }
\draw [thick] (0,0) -- (1,0) -- (2,0)-- (3,1);
\filldraw[black](0,0)circle[radius=2pt];
\filldraw[black](1,0)circle[radius=2pt];
\filldraw[black](2,0)circle[radius=2pt];
\filldraw[black](3,1)circle[radius=2pt];
\node [anchor=north] at (0.5,0.75) {$b_0$};
\node [anchor=north] at (1.5,0.75) {$b_0$};
\end{tikzpicture}
\captionsetup{labelformat=empty}
\caption{$w(p) =b_0^2$}
\end{subfigure} 
   \begin{subfigure}[b]{0.19\textwidth}
\centering
\begin{tikzpicture}[scale = 0.65]
\draw(-0.25,-0.25) grid (3.5,2.5);
\foreach \x in {0,1,...,3} { \node [anchor=north] at (\x,-0.3) {\x}; }
\draw [thick] (0,0) -- (1,0) -- (2,1)-- (3,1);
\filldraw[black](0,0)circle[radius=2pt];
\filldraw[black](1,0)circle[radius=2pt];
\filldraw[black](2,1)circle[radius=2pt];
\filldraw[black](3,1)circle[radius=2pt];
\node [anchor=north] at (0.5,0.75) {$b_0$};
\node [anchor=north] at (2.5,1.75) {$b_1$};
\end{tikzpicture}
\captionsetup{labelformat=empty}
\caption{$w(p) =b_0 b_1$}
\end{subfigure} 
  \begin{subfigure}[b]{0.19\textwidth}
\centering
\begin{tikzpicture}[scale = 0.65]
\draw(-0.25,-0.25) grid (3.5,2.5);
\foreach \x in {0,1,...,3} { \node [anchor=north] at (\x,-0.3) {\x}; }
\draw [thick] (0,0) -- (1,1) -- (2,1)-- (3,1);
\filldraw[black](0,0)circle[radius=2pt];
\filldraw[black](1,1)circle[radius=2pt];
\filldraw[black](2,1)circle[radius=2pt];
\filldraw[black](3,1)circle[radius=2pt];
\node [anchor=north] at (1.5,1.75) {$b_1$};
\node [anchor=north] at (2.5,1.75) {$b_1$};
\end{tikzpicture}
\captionsetup{labelformat=empty}
\caption{$w(p) = b_1^2$}
\end{subfigure} 
\begin{subfigure}[b]{0.19\textwidth}
\centering
\begin{tikzpicture}[scale = 0.65]
\draw(-0.25,-0.25) grid (3.5,2.5);
\foreach \x in {0,1,...,3} { \node [anchor=north] at (\x,-0.3) {\x}; }
\draw [thick] (0,0) -- (1,1) -- (2,0)-- (3,1);
\filldraw[black](0,0)circle[radius=2pt];
\filldraw[black](1,1)circle[radius=2pt];
\filldraw[black](2,0)circle[radius=2pt];
\filldraw[black](3,1)circle[radius=2pt];
\node [anchor=north] at (1.75,1.1) {$\lambda_1$};
\end{tikzpicture}
\captionsetup{labelformat=empty}
\caption{$w(p) = \lambda_1$}
\end{subfigure} 
\begin{subfigure}[b]{0.19\textwidth}
\centering
\begin{tikzpicture}[scale = 0.65]
\draw(-0.25,-0.25) grid (3.5,2.5);
\foreach \x in {0,1,...,3} { \node [anchor=north] at (\x,-0.3) {\x}; }
\draw [thick] (0,0) -- (2,2) -- (3,1);
\filldraw[black](0,0)circle[radius=2pt];
\filldraw[black](1,1)circle[radius=2pt];
\filldraw[black](2,2)circle[radius=2pt];
\filldraw[black](3,1)circle[radius=2pt];
\node [anchor=north] at (2.75,2.1) {$\lambda_2$};
\end{tikzpicture}
\captionsetup{labelformat=empty}
\caption{$w(p) = \lambda_2$}
\end{subfigure}\hspace{0.0cm} 
  \hspace*{\fill}
  \caption{ The Motzkin paths from \( (0,0) \) to \( (3,1) \) and
    their weights. The weights of horizontal and down steps are
    indicated.}
\label{fig:motzkin}
\end{figure}

There is a natural way to extend univariate orthogonal polynomials to
multivariate orthogonal polynomials using bialternant formulas. In the
expansion of the multivariate orthogonal polynomials (resp.~Schur
polynomials) in terms of Schur polynomials (resp.~multivariate
orthogonal polynomials), the coefficients are given as determinants of
\( \nu_{n,k} \) (resp.~\( \sigma_{n,k} \)). Therefore one may ask whether
the Lindstr\"om--Gessel--Viennot (LGV) lemma \cite{LGV} (see also
\cite[Chapter~32]{Aigner2018}) can be used to find an interpretation
for the Schur coefficients using nonintersecting paths. However,
neither Motzkin paths nor Favard paths give such a model, since their
underlying graphs are not planar. One of the motivations of this paper
is to find a planar graph that provides a lattice path interpretation
for the mixed moments, so that it can be extended naturally to
multivariate orthogonal polynomials. 

Corteel and Kim \cite{LHT} introduced
the lecture hall graph in which paths are in natural bijection with
lecture hall partitions and anti-lecture hall compositions. Lecture
hall partitions with $n$ nonnegative parts were originally defined in
the enumeration of certain affine permutations coming from the affine
Coxeter group $\tilde{C}_n$ \cite{BME1}. They also have many
connections with other combinatorial objects; see \cite{LHPSavage} for
example.

Corteel and Kim \cite{LHT} found a path model in the lecture hall
graph for both mixed moments and coefficients of the little
\( q \)-Jacobi polynomials. A nice feature of their combinatorial
model is that it can be extended to give a tableau model for
multivariate little \( q \)-Jacobi polynomials. Their discovery of the
combinatorial model for little \( q \)-Jacobi polynomials, however,
was somewhat by accident, because the coefficient \( \nu_{n,k} \) of
the little \( q \)-Jacobi polynomial happens to be equal to a formula,
previously established in \cite{trunc_LHP}, for a generating function
for truncated lecture hall partitions.

The goal of this paper is to develop a general interpretation of the
mixed moments \( \sigma_{n,k} \) and the coefficients \( \nu_{n,k} \)
of orthogonal polynomials as generating functions for paths in the
lecture hall graph. We show that it is not a coincidence that paths on
the lecture hall graph and little $q$-Jacobi polynomials are related.
In fact, we prove that the relation is much stronger and that the
orthogonal polynomials in the whole \( q \)-Askey scheme \cite{KLS}
can be studied this way. 

We give a simple algorithm to construct a possible candidate giving a
lecture hall path interpretation for both coefficients and mixed
moments of orthogonal polynomials. This algorithm allows us to
rediscover the lecture hall path interpretations for the little
\( q \)-Jacobi polynomials without any prior knowledge of the
generating function for truncated lecture hall partitions. 

Our main result may be summarized in the following ``meta'' theorem.
This is deliberately quite vague for now. The purpose of our
manuscript is to develop the tools and results necessary to make it
precise. 

\begin{thm}
  For every family of polynomials in the $q$-Askey scheme and in the
  Askey scheme, the mixed moments and coefficients have combinatorial
  models on the lecture hall graph.
\end{thm}

For example, \Cref{thm:AW-wt-full} gives a combinatorial model for the
mixed moments of the Askey--Wilson polynomials, and
\Cref{pro:h-e-dual} together with this provides a model for their
coefficients.

For certain polynomials in the $q$-Askey scheme, we can give different
models: some will be just on the lecture hall graph of finite height;
others will be on the graph of infinite height. For the polynomials in
the \( q \)-Askey scheme, the models with finite height have the
benefit that we can take the limit $q\rightarrow 1$ and those with
infinite rows give a combinatorial model for the formal power series
expansion of the mixed moments and the coefficients.

Here are some advantages of our lecture hall graph models. The first
advantage of this method is that it gives a combinatorial model for
the mixed moments and the coefficients with \emph{the same graph
  model}, where southeast paths are used for mixed moments and
northeast paths without consecutive east steps are used for
coefficients. Note that Viennot's combinatorial models do not have
this uniform property because Motzkin paths and Favard paths have
different underlying graphs.

The second advantage is that it has a nice relationship with the
\( q \)-Askey scheme. We start with the Stieltjes--Wigert polynomials
and we establish simple lemmas that let us go up in the scheme. The
only family that needs to be treated as a special case is that of the
continuous \( q \)-Hermite polynomials.

The third advantage is that, since the lecture hall graph is planar,
our combinatorial models can be naturally extended to study
multivariate orthogonal polynomials using the
Lindstr\"om--Gessel--Viennot lemma, which is not possible for
Viennot's model. This was done in \cite{LHT} for little $q$-Jacobi
polynomials and we can now generalize this to the whole \( q \)-Askey
scheme. In a forthcoming paper, we will give applications of our
combinatorial models to multivariate versions of orthogonal
polynomials, total positivity, and random matrix theory.

The last but not least advantage is that it gives a combinatorial
model for the mixed moments of Askey--Wilson polynomials. Using this
model we give the first combinatorial proof of the symmetry of the
parameters \( a,b,c,d \) in the Askey--Wilson polynomials, which until
now has been understood only analytically. We note that there is a
combinatorial model, called staircase tableaux, for moments of
Askey--Wilson polynomials with some reparametrization of \( a,b,c,d \)
to \( \alpha,\beta,\gamma,\delta \); see \cite{Corteel2011} and
\cite{Corteel2012}. It is known that the new parameters
\( \alpha,\beta,\gamma,\delta \) are also symmetric, but finding a
combinatorial proof is still open.

The paper is organized as follows. In \Cref{sec:preliminaries} we
define the lecture hall graph and we introduce mixed moments relative
to different bases. In \Cref{sec:guess-prov-techn} we prove useful
properties of the lecture hall graph, which will be used throughout
the paper. In \Cref{sec:bootstr-meth-from} we find weight systems for
the mixed moments of Stieltjes--Wigert polynomials, \( q \)-Bessel
polynomials, little \( q \)-Jacobi polynomials, big \( q \)-Jacobi
polynomials, and Askey--Wilson polynomials, using a bootstrapping
method. In \Cref{sec:anoth-bootstr-meth} we give another bootstrapping
method to find a lecture hall graph model for mixed moments of
Askey--Wilson polynomials relative to continuous \( q \)-Hermite
polynomials. In \Cref{sec:comb-prop-askey} we give another
combinatorial model for mixed moments of Askey--Wilson polynomials.
Using this model we give the first combinatorial proof of the symmetry
of \( a,b,c,d \) in Askey--Wilson polynomials.

Finally, we remark that our method can also be applied to all orthogonal
polynomials in the \( q \)-Askey scheme and in the Askey scheme. The
remaining orthogonal polynomials will be covered in a forthcoming
paper.

\section{Preliminaries}
\label{sec:preliminaries}

In this section we give basic definitions and lemmas on the lecture
hall graph. We then define mixed moments and coefficients of
orthogonal polynomials relative to other bases. We will follow the
standard notation of basic hypergeometric series; see for example
\cite{GR}.

\subsection{Lecture hall graphs}
\label{sec:lecture-hall-graph}

The lecture hall graph was first introduced by Corteel and
Kim~\cite{LHT} in their study of little \( q \)-Jacobi polynomials.
It has further been studied in \cite{Corteel2020}
and \cite{Corteel2021a}.

For nonnegative integers \( t,i,j \) with \( j\le i \),
we denote
\[
 v^t_{i,j} = \left(i,t+ \frac{j}{i+1}\right) \in \ZZ\times\QQ.
\]

\begin{defn}
  The \emph{lecture hall graph} is the (undirected) graph
  \( \mathcal{G} = (V,E) \), where
\begin{align*}
  V &= \{v^t_{i,j}: t,i,j\in \ZZ_{\ge0}, 0\le j\le i \},\\
  E &= \{(v^t_{i,j}, v^t_{i,j+1}) : t,i,j\in\ZZ_{\ge0}, 0\le j\le i \} 
        \cup \{(v^t_{i,j}, v^t_{i+1,j}) : t,i,j\in\ZZ_{\ge0}, 0\le j\le i \}.
\end{align*}
Here, \( v^t_{i,i+1} \) is defined to be \( v^{t+1}_{i,0} \); see
\Cref{fig:LHL}.
\end{defn}

\begin{figure}
  \centering
  \begin{tikzpicture}[scale=2]
    \LHLL{4}2
    \begin{scope}
      \node at (0,0) [circle,fill,inner sep=1pt]{};
      \node at (1,0) [circle,fill,inner sep=1pt]{};
      \node at (1,1/2) [circle,fill,inner sep=1pt]{};
      \node at (2,0) [circle,fill,inner sep=1pt]{};
      \node at (2,1/3) [circle,fill,inner sep=1pt]{};
      \node at (2,2/3) [circle,fill,inner sep=1pt]{};
      \node at (3,0) [circle,fill,inner sep=1pt]{};
      \node at (3,1/4) [circle,fill,inner sep=1pt]{};
      \node at (3,2/4) [circle,fill,inner sep=1pt]{};
      \node at (3,3/4) [circle,fill,inner sep=1pt]{};
      \node at (0,1) [circle,fill,inner sep=1pt]{};
      \node at (1,1) [circle,fill,inner sep=1pt]{};
      \node at (1,3/2) [circle,fill,inner sep=1pt]{};
      \node at (2,1) [circle,fill,inner sep=1pt]{};
      \node at (2,4/3) [circle,fill,inner sep=1pt]{};
      \node at (2,5/3) [circle,fill,inner sep=1pt]{};
      \node at (3,1) [circle,fill,inner sep=1pt]{};
      \node at (3,5/4) [circle,fill,inner sep=1pt]{};
      \node at (3,6/4) [circle,fill,inner sep=1pt]{};
      \node at (3,7/4) [circle,fill,inner sep=1pt]{};
    \end{scope}
    \begin{scope}[shift={(0.2,0.1)}]
      \small
      \node at (0,0) {\( v^0_{0,0} \)};
      \node at (1,0) {\( v^0_{1,0} \)};
      \node at (1,1/2) {\( v^0_{1,1} \)};
      \node at (2,0) {\( v^0_{2,0} \)};
      \node at (2,1/3) {\( v^0_{2,1} \)};
      \node at (2,2/3) {\( v^0_{2,2} \)};
      \node at (3,0) {\( v^0_{3,0} \)};
      \node at (3,1/4) {\( v^0_{3,1} \)};
      \node at (3,2/4) {\( v^0_{3,2} \)};
      \node at (3,3/4) {\( v^0_{3,3} \)};
      \node at (0,1) {\( v^1_{0,0} \)};
      \node at (1,1) {\( v^1_{1,0} \)};
      \node at (1,3/2) {\( v^1_{1,1} \)};
      \node at (2,1) {\( v^1_{2,0} \)};
      \node at (2,4/3) {\( v^1_{2,1} \)};
      \node at (2,5/3) {\( v^1_{2,2} \)};
      \node at (3,1) {\( v^1_{3,0} \)};
      \node at (3,5/4) {\( v^1_{3,1} \)};
      \node at (3,6/4) {\( v^1_{3,2} \)};
      \node at (3,7/4) {\( v^1_{3,3} \)};
    \end{scope}
    \node at (0.5,2.3) {\( \vdots \)};
    \node at (1.5,2.3) {\( \vdots \)};
    \node at (2.5,2.3) {\( \vdots \)};
    \node at (3.5,2.3) {\( \vdots \)};
    \node at (4.5,0.5) {\( \cdots \)};
    \node at (4.5,1.5) {\( \cdots \)};
  \end{tikzpicture}
  \caption{The lecture hall graph \( \mathcal{G} \).}
  \label{fig:LHL}
\end{figure}

We consider three kinds of steps:
\begin{itemize}
\item an \emph{east step} is a directed edge of the form
  \( (v^t_{i,j}, v^t_{i+1,j}) \) in \( \mathcal{G} \),
\item a \emph{north step} is a directed edge of the form
  \( (v^t_{i,j}, v^t_{i,j+1}) \) or \( (v^t_{i,i}, v^{t+1}_{i,0}) \),
\item a \emph{south step} is a directed edge of the form
  \( (v^t_{i,j}, v^t_{i,j-1}) \) or \( (v^t_{i,0}, v^{t-1}_{i,i}) \).
\end{itemize}

For two vertices \( u \) and \( v \) in \( \mathcal{G} \), a \emph{path} from
\( u \) to \( v \) is a sequence \( (v_0,v_1,\dots,v_n) \) of vertices
of \( \mathcal{G} \) such that \( v_0=u \), \( v_n=v \), and each
\( (v_{i}, v_{i+1}) \) is an east, north, or south step. We also
consider a path from \( u\in \mathcal{G} \) to \( v=(a,\infty) \), which is a
sequence \( (v_0,v_1,v_2,\ldots) \) of vertices of \( \mathcal{G} \) such that
\( v_0=u \), \( \lim_{n\to \infty}v_n=v \), and each
\( (v_{i}, v_{i+1}) \) is an east, north, or south step. Similarly, a
path from \( u=(a,\infty) \) to \( v\in \mathcal{G} \) is a sequence
\( (\ldots, v_{2},v_{1},v_0) \) of vertices of \( \mathcal{G} \) such that
\( v_0=v \), \( \lim_{n\to \infty}v_{n}=u \), and each
\( (v_{i+1}, v_{i}) \) is an east, north, or south step.

Let \( \SE(u\to v) \) denote the set of paths from \( u \) to \( v \)
in \( \mathcal{G} \) consisting of south and east steps. We also
define \( \NE^*(u\to v) \) to be the set of paths from \( u \) to
\( v \) in \( \mathcal{G} \) consisting of north and east steps with
\emph{no consecutive east steps}. Here, consecutive east steps means a
pair \( (s_1,s_2) \) of steps of the form
\( s_1=(v^t_{i,j},v^t_{i+1,j}) \) and
\( s_2=(v^t_{i+1,j},v^t_{i+2,j}) \). For example, in \Cref{fig:1}, the
path \( p_1\in \SE((1,\infty)\to (5,0)) \) has consecutive east steps
\( (v^2_{2,2}, v^2_{3,2}) \) and \( (v^2_{3,2}, v^2_{4,2}) \) but the
path \( p_2\in \NE^*((1,0)\to (5,\infty)) \) does not have any
consecutive east steps.

\begin{figure}
  \centering
  \begin{tikzpicture}[scale=1.7]
    \LHLL53 \LHlabel52
    \draw [red, very thick] (5,0) -- (5,3/6) -- (4,3/5)
     -- (4,7/5) -- (3,6/4) -- (2,5/3) -- (2,2) -- (1,2) -- (1,3);
     \draw [blue, very thick] (1,0) -- (2,0) -- (2,1/3) -- (3,1/4)
     -- (3,7/4) -- (4,8/5) -- (4,12/5) -- (5,14/6) -- (5,3);
     \node at (.8,2.8) {\( p_1 \)};
     \node at (.8,.2) {\( p_2 \)};
   \begin{scope}[shift={(-.5,.1)}]
   \end{scope}
   \end{tikzpicture}
   \caption{A path \( p_1 \) in \( \SE((1,\infty)\to (5,0)) \)
     and a path \( p_2 \) in \( \NE^*((1,0)\to (5,\infty)) \).}
\label{fig:1}
\end{figure}

Paths in \( \SE(u\to v) \) or
\( \NE^*(u\to v) \) are closely related to lecture hall
partitions. To describe their connection we need some definitions.

A \emph{lecture hall partition}
is a nonnegative integer sequence
  \( \lambda=(\lambda_{1},\dots,\lambda_n) \) such that
  \[
    \frac{\lambda_{1}}{n} \ge \frac{\lambda_{2}}{n-1} \ge
     \cdots \ge \frac{\lambda_{n}}{1}.
  \]
An \emph{anti-lecture hall composition}
is a nonnegative integer sequence
  \( \alpha=(\alpha_{1},\dots,\alpha_n) \) such that
  \[
    \frac{\alpha_{1}}{1} \ge \frac{\alpha_{2}}{2} \ge
     \cdots \ge \frac{\alpha_{n}}{n}.
  \]
  As variations of the above objects, we consider the set
  \( AL_{n,k} \) of \emph{truncated anti-lecture hall compositions}
  and the set \( RL^*_{n,k} \) of \emph{truncated strict reverse
    lecture hall partitions} defined by
  \begin{align*}
    AL_{n,k}
    &= \left\{ (\alpha_{k+1},\dots,\alpha_n)
      :    \frac{\alpha_{k+1}}{k+1} \ge \frac{\alpha_{k+2}}{k+2} \ge
     \cdots \ge \frac{\alpha_{n}}{n} \right\}, \\
    RL^*_{n,k}
    &= \left\{ (\rho_{k+1},\dots,\rho_n):
          \frac{\rho_{k+1}}{k+1} < \frac{\rho_{k+2}}{k+2} <
     \cdots < \frac{\rho_{n}}{n} \right\}.
\end{align*}

Now we define a map \( \phi \), which sends a path to a sequence of
integers. Suppose that \( p \) is a path from \( (k,a) \) to
\( (n,b) \) for some \( k,n\in \ZZ_{\ge0} \) and
\( a,b\in \ZZ_{\ge0}\cup\{\infty\} \). For \( i=k+1,\dots,n \),
suppose that the east step of \( p \) between \( x=i-1 \) and
\( x=i \) is the \( \alpha_i \)th east step among all east steps
between \( x=i-1 \) and \( x=i \) from the bottom. In other words, if
the east step of \( p \) between \( x=i-1 \) and \( x=i \) is
\( (v^t_{i-1,j},v^t_{i,j}) \) in \( \mathcal{G} \), then
\( \alpha_i = ti + j \). We define
\( \phi(p) = (\alpha_{k+1},\dots,\alpha_n) \). For example, if
\( p_1 \) and \( p_2 \) are the paths in \Cref{fig:1}, then
\( \phi(p_1) = (4,5,6,3) \) and \( \phi(p_2) = (0,1,7,12) \). This map
gives natural bijections from paths in a lecture hall graph to
truncated anti-lecture hall compositions and truncated reverse lecture
hall partitions.

\begin{prop}\label{pro:3bij}
  The map \( \phi \) induces the following two bijections:
  \begin{align*}
   \phi &: \SE((k,\infty)\to (n,0)) \to AL_{n,k}, \\
   \phi &: \NE^*((k,0)\to (n,\infty)) \to RL^*_{n,k}.
  \end{align*}
\end{prop}

\begin{proof}
  Let \( p\in \SE((k,\infty)\to (n,0)) \) and
  \( \phi(p) = (\alpha_{k+1},\dots,\alpha_n) \). By the construction
  of \( \phi \), for \( i=k+1,\dots,n \), if the east step of \( p \)
  between \( x=i-1 \) and \( x=i \) is \( (v^t_{i-1,j},v^t_{i,j}) \),
  then \( \alpha_i = ti + j \). Since \( \alpha_i/i = t+j/i \) is the
  \( y \)-coordinate of the starting point \( v^t_{i-1,j} \) of this
  east step, we have
  \( \frac{\alpha_{k+1}}{k+1} \ge \frac{\alpha_{k+2}}{k+2} \ge \cdots
  \ge \frac{\alpha_{n}}{n} \). Thus
  \( (\alpha_{k+1},\dots,\alpha_n)\in AL_{n,k} \). Conversely, for
  \( (\alpha_{k+1},\dots,\alpha_n)\in AL_{n,k} \), we can reconstruct
  \( p\in \SE((k,\infty)\to (n,0)) \) using the relations
  \( t = \flr{\alpha_i/i} \) and
  \( j = \alpha_i - i\flr{\alpha_i/i} \). This shows the first
  bijection. The second bijection can be proved similarly.
\end{proof}

Our main goal is to find combinatorial models for mixed moments
\( \sigma_{n,k} \) and coefficients \( \nu_{n,k} \) of orthogonal
polynomials. We will see later in \Cref{pro:h-e-dual} that a
combinatorial model for \( \sigma_{n,k} \) using \( \SE(u\to v) \)
immediately gives a combinatorial model for \( \nu_{n,k} \) using
\( \NE^*(u\to v) \), and vice versa. Hence, in this paper we will mostly consider the
set \( \SE(u\to v) \) rather than \( \NE^*(u\to v) \). For brevity, we
will write \( p:u\to v \) to mean \( p\in \SE(u\to v) \).
  
\medskip

A \emph{weight system} is a function \( w \) that assigns a weight
\( w(s) \) to each east step \( s=(v^t_{i,j}, v^t_{i+1,j}) \) in
\( \mathcal{G} \). We denote
\[
  w(t;i,j) := w(v^t_{i,j},v^t_{i+1,j}).
\]
In other words, \( w(t;i,j) \) is the weight of the \( j \)th east
step from the bottom, where the bottommost one is the \( 0 \)th step,
among the east steps in the region
\( \{(x,y): i\le x\le i+1, t\le y<t+1 \} \); see \Cref{fig:LHL-wt}.
Given a weight system \( w \), the weight \( w(p) \) of a path \( p \)
is defined to be the product of the weights of all east steps in
\( p \).

\begin{figure}
  \centering
  \begin{tikzpicture}[scale=2]
    \LHLL{4}2
    \LHlabel{4}2
    \begin{scope}
      \node at (0,0) [circle,fill,inner sep=1pt]{};
      \node at (1,0) [circle,fill,inner sep=1pt]{};
      \node at (1,1/2) [circle,fill,inner sep=1pt]{};
      \node at (2,0) [circle,fill,inner sep=1pt]{};
      \node at (2,1/3) [circle,fill,inner sep=1pt]{};
      \node at (2,2/3) [circle,fill,inner sep=1pt]{};
      \node at (3,0) [circle,fill,inner sep=1pt]{};
      \node at (3,1/4) [circle,fill,inner sep=1pt]{};
      \node at (3,2/4) [circle,fill,inner sep=1pt]{};
      \node at (3,3/4) [circle,fill,inner sep=1pt]{};
      \node at (0,1) [circle,fill,inner sep=1pt]{};
      \node at (1,1) [circle,fill,inner sep=1pt]{};
      \node at (1,3/2) [circle,fill,inner sep=1pt]{};
      \node at (2,1) [circle,fill,inner sep=1pt]{};
      \node at (2,4/3) [circle,fill,inner sep=1pt]{};
      \node at (2,5/3) [circle,fill,inner sep=1pt]{};
      \node at (3,1) [circle,fill,inner sep=1pt]{};
      \node at (3,5/4) [circle,fill,inner sep=1pt]{};
      \node at (3,6/4) [circle,fill,inner sep=1pt]{};
      \node at (3,7/4) [circle,fill,inner sep=1pt]{};
    \end{scope}
    \begin{scope}[shift={(0.5,0.1)}]
      \small
      \node at (0,0) {\( w(0;0,0) \)};
      \node at (1,0) {\( w(0;1,0) \)};
      \node at (1,1/2) {\( w(0;1,1) \)};
      \node at (2,0) {\( w(0;2,0) \)};
      \node at (2,1/3) {\( w(0;2,1) \)};
      \node at (2,2/3) {\( w(0;2,2) \)};
      \node at (3,0) {\( w(0;3,0) \)};
      \node at (3,1/4) {\( w(0;3,1) \)};
      \node at (3,2/4) {\( w(0;3,2) \)};
      \node at (3,3/4) {\( w(0;3,3) \)};
      \node at (0,1) {\( w(1;0,0) \)};
      \node at (1,1) {\( w(1;1,0) \)};
      \node at (1,3/2) {\( w(1;1,1) \)};
      \node at (2,1) {\( w(1;2,0) \)};
      \node at (2,4/3) {\( w(1;2,1) \)};
      \node at (2,5/3) {\( w(1;2,2) \)};
      \node at (3,1) {\( w(1;3,0) \)};
      \node at (3,5/4) {\( w(1;3,1) \)};
      \node at (3,6/4) {\( w(1;3,2) \)};
      \node at (3,7/4) {\( w(1;3,3) \)};
    \end{scope}
    \node at (0.5,2.3) {\( \vdots \)};
    \node at (1.5,2.3) {\( \vdots \)};
    \node at (2.5,2.3) {\( \vdots \)};
    \node at (3.5,2.3) {\( \vdots \)};
    \node at (4.5,0.5) {\( \cdots \)};
    \node at (4.5,1.5) {\( \cdots \)};
  \end{tikzpicture}
  \caption{An illustration of \( w(t;i,j) \).}
  \label{fig:LHL-wt}
\end{figure}

The following definition will be used throughout this paper.
\begin{defn} \label{def:h-e} For a weight system \( w \) and two
  nonnegative integers \( n \) and \( k \), we define
  \begin{align*}
  h^w_{n,k} &= \sum_{p\in \SE((k,\infty)\to (n,0))} w(p),  \\
  e^w_{n,k} &= \sum_{p\in \NE^*((k,0)\to (n,\infty))} w(p).
\end{align*}
\end{defn}

By definition, we have \( h^w_{n,k}=e^w_{n,k}=0 \) if \( n<k \) and
\( h^w_{n,n}=e^w_{n,n}=1 \). Note that \( h^w_{n,k} \) and
\( e^w_{n,k} \) generalize the \emph{homogeneous symmetric function}
\[
  h_n(x_0,x_1,\ldots) = \sum_{0\le i_0\le i_1 \le \cdots \le i_{n-1}}
  x_{i_0} x_{i_1} \cdots x_{i_{n-1}}
\]
 and the \emph{elementary symmetric function} 
\[
  e_n(x_0,x_1,\ldots) = \sum_{0\le i_0< i_1 < \cdots < i_{n-1}}
  x_{i_0} x_{i_1} \cdots x_{i_{n-1}}
\]
in the following sense: if \( w \) is the weight system defined by
\[
  w(t;i,j) =
  \begin{cases}
   x_t & \mbox{if \( j=0 \)},\\
   0 & \mbox{otherwise,}
  \end{cases}
\]
then we have
\[
  h_n(x_0,x_1,\ldots) = h^w_{n,0}, \qquad 
  e_n(x_0,x_1,\ldots) = e^w_{n,0}.
\]

The quantities \( h^w_{n,k} \) and \( e^w_{n,k} \) are dual to each
other in the following sense.

\begin{lem}\label{lem:dual-wt}
  We have the matrix identity
  \[
    (h^w_{n,k})_{n,k=0}^\infty =
    \left( ((-1)^{n-k} e^w_{n,k})_{n,k=0}^\infty \right)^{-1}.
  \]
  Equivalently, for \( n\ge m \), we have
  \[
    \sum_{r=m}^n h^w_{n,r}(-1)^{r-m} e^w_{r,m} = \delta_{n,m}.
  \]
\end{lem}
\begin{proof}
  The equivalence of the two identities is immediate from the fact
  that \( (h^w_{n,k})_{n,k=0}^\infty \) is lower triangular. Thus it
  suffices to prove the second identity. Since this identity is
  clearly true for \( n=m \), we assume that \( n>m \). We follow the
  argument in the proof of \cite[Proposition~3.5]{LHT}.

  By \Cref{pro:3bij}, we can write
  \begin{equation}\label{eq:4}
    \sum_{r=m}^n h^w_{n,r}(-1)^{r-m} e^w_{r,m}
    =  \sum_{(\rho,\alpha)\in X} W(\rho,\alpha),
  \end{equation}
  where \( X \) is the set of pairs \( (\rho,\alpha) \) of nonnegative
  integer sequences
  \( \rho=(\rho_{m+1},\rho_{m+2},\dots,\rho_{r}) \) and
  \( \alpha=(\alpha_{r+1},\alpha_{r+2},\dots,\alpha_{n}) \), for some \( r \),
  such that
  \begin{equation}\label{eq:3}
    \frac{\rho_{m+1}}{m+1} < \frac{\rho_{m+2}}{m+2} < \cdots < \frac{\rho_{r}}{r},
    \qquad 
    \frac{\alpha_{r+1}}{r+1} \ge \frac{\alpha_{r+2}}{r+2} \ge \cdots \ge \frac{\alpha_{n}}{n}, 
  \end{equation}
  and \( W(\rho,\alpha) \) is defined by
  \[
    W(\rho,\alpha) = (-1)^{r-m}
    \prod_{i=m+1}^{r} w(\flr{\rho_i/i};i,\rho_i-i\flr{\rho_i/i})
    \prod_{i=r+1}^{n} w(\flr{\alpha_i/i};i,\alpha_i - i\flr{\alpha_i/i}).
  \]

  We define a sign-reversing involution on \( X \) as follows. Suppose
  that \( (\rho,\alpha)\in X \) is given as in \eqref{eq:3}. If
  \( \alpha_{r+1}/(r+1)\le \rho_r/r \), then let
  \( \alpha' = (\rho_r,\alpha_{r+1},\dots,\alpha_n) \) and
  \( \rho' = (\rho_{m+1},\dots,\rho_{r-1}) \), and otherwise let
  \( \alpha' = (\alpha_{r+2},\dots,\alpha_n) \) and
  \( \rho' = (\rho_{m+1}, \dots, \rho_r, \alpha_{r+1}) \).
  Then \( (\rho',\alpha')\in X \) and
  \( W(\rho',\alpha') = - W(\rho,\alpha) \). It is easy to check that
  the map \( (\rho,\alpha) \mapsto (\rho',\alpha') \) is a
  sign-reversing involution on \( X \) with no fixed point if
  \( n>m \). Thus (\ref{eq:4}) is equal to \( 0 \) and the proof is completed.
\end{proof}

The following simple lemma will be used frequently in this paper.

\begin{lem}\label{lem:wt-mult}
  Suppose that \( w \) and \( w' \) are weight systems such that
  \( w'(t;i,j) = C_i \cdot w(t;i,j) \) for all
  \( t,i,j\in\ZZ_{\ge0} \) with \( j\le i \). Then
  \[
    h^{w'}_{n,k} = C_{k}C_{k+1} \cdots C_{n-1} \cdot h^{w}_{n,k}.
  \]
\end{lem}

\begin{proof}
This follows from the observation that
\begin{align*}
  h^{w'}_{n,k} &= \sum_{p\in \SE((k,\infty)\to (n,0))} w'(p)\\
  &= \sum_{p\in \SE((k,\infty)\to (n,0))} C_{k}C_{k+1} \cdots C_{n-1} w(p)
  = C_{k}C_{k+1} \cdots C_{n-1} \cdot h^{w}_{n,k},
\end{align*}
because every \( p\in \SE((k,\infty)\to (n,0)) \) has exactly one east
step between \( x=i \) and \( x=i+1 \) for \( i= k,k+1,\dots,n-1 \).
\end{proof}

\subsection{Mixed moments relative to other bases}

Recall that the mixed moments \( \{\sigma_{n,k}\}_{n,k\ge 0} \) of
orthogonal polynomials \( \{p_n(x)\}_{n\ge0} \) with respect to a linear
functional \( \LL \) are defined by
\[
  \sigma_{n,k} = \frac{\LL(x^np_k(x))}{\LL(p_k(x)^2)},
\]
or equivalently,
\[
  x^n = \sum_{k=0}^{n} \sigma_{n,k} p_k(x).
\]
By abuse of terminology, we extend the definition of the mixed moments
\( \sigma_{n,k} \) to polynomials that are not necessarily orthogonal
polynomials.

\begin{defn}
  Let \( \{p_n(x)\}_{n\ge0} \) be a sequence of polynomials with
  \( \deg p_n(x) = n \). The \emph{mixed moments}
  \( \{\sigma_{n,k}\}_{n,k\ge 0} \) and the \emph{coefficients}
  \( \{\nu_{n,k}\}_{n,k\ge 0} \) of \( \{p_n(x)\}_{n\ge0} \) are
  defined by
  \[
      x^n = \sum_{k=0}^{n} \sigma_{n,k} p_k(x), \qquad 
      p_n(x) = \sum_{k=0}^{n} \nu_{n,k} x^k.
  \]
\end{defn}

More generally, we also define mixed moments and coefficients by
choosing a basis other than \( \{x^n\}_{n\ge0} \).

\begin{defn}
  Let \( \{p_n(x)\}_{n\ge0} \) and \( \{q_n(x)\}_{n\ge0} \) be
  sequences of polynomials with \( \deg p_n(x) = \deg q_n(x) = n \).
  We define the \emph{mixed moments \( \{\sigma_{n,k}\}_{n,k\ge 0} \)}
  and the \emph{coefficients} \( \{\nu_{n,k}\}_{n,k\ge 0} \) of
  \( \{p_n(x)\}_{n\ge0} \) \emph{relative to} \( \{q_n(x)\}_{n\ge0} \)
  by
\[
  q_n(x) = \sum_{k=0}^{n} \sigma_{n,k} p_k(x), \qquad
  p_n(x) = \sum_{k=0}^{n} \nu_{n,k} q_k(x).
\]
\end{defn}

Note that the original mixed moments \( \{\sigma_{n,k}\}_{n,k\ge 0} \)
of \( \{p_n(x)\}_{n\ge0} \) are the mixed moments of
\( \{p_n(x)\}_{n\ge0} \) relative to the standard basis polynomials
\( \{x^n\}_{n\ge0} \). Note also that, by definition, the mixed
moments \( \{\sigma_{n,k}\}_{n,k\ge 0} \) and the coefficients
\( \{\nu_{n,k}\}_{n,k\ge 0} \) of \( \{p_n(x)\}_{n\ge0} \) relative to
\( \{q_n(x)\}_{n\ge0} \) are the entries of the change-of-basis
matrices between the two bases \( \{p_n(x)\}_{n\ge0} \) and
\( \{q_n(x)\}_{n\ge0} \) of the space of polynomials. Therefore the
two matrices \( (\sigma_{n,k})_{n,k=0}^\infty \) and
\( (\nu_{n,k})_{n,k=0}^\infty \) are inverses of each other. By this
fact and \Cref{lem:dual-wt}, we obtain the following proposition.

\begin{prop}\label{pro:h-e-dual}
  Let \( \sigma_{n,k} \) and \( \nu_{n,k} \) be the mixed moments and
  coefficients of polynomials \( \{p_n(x)\}_{n\ge0} \) (with respect
  to some basis \( \{ q_n(x) \}_{n\ge 0} \)). If \( w \) is a weight
  system such that
  \begin{equation}\label{eq:26}
    \sigma_{n,k} = h_{n,k}^w, \qquad n\ge k\ge 0,
  \end{equation}
  then
  \begin{equation}\label{eq:27}
    \nu_{n,k} = (-1)^{n-k} e_{n,k}^{w}, \qquad n\ge k\ge 0.
  \end{equation}
  Conversely, if \( w \) is a weight
  system satisfying \eqref{eq:27}, then \eqref{eq:26} also holds.
\end{prop}

\Cref{pro:h-e-dual} implies that if we have a lecture hall graph
model for \( \sigma_{n,k} \) then the same weight system also gives a
lecture hall graph model for \( (-1)^{n-k}\nu_{n,k} \) and vice
versa. Therefore we will henceforth only focus on finding a lecture
hall graph model for the mixed moments \( \sigma_{n,k} \).

By the following
lemma with \( b_n(x) = x^n \), we can use the mixed moments
\( \{\ts_{n,k}\}_{n,k\ge 0} \) relative to \( \{q_n(x)\}_{n\ge0} \) as
an intermediate step to study the original mixed moments
\( \{\sigma_{n,k}\}_{n,k\ge 0} \).

\begin{lem}\label{lem:inter-mixed}
  Let \( \{p_n(x)\}_{n\ge0} \), \( \{q_n(x)\}_{n\ge0} \), and
  \( \{b_n(x)\}_{n\ge0} \) be sequences of polynomials with
  \( \deg p_n(x) = \deg q_n(x) = \deg b_n(x) = n \) and let
\begin{align*}
 b_n(x) &= \sum_{k=0}^{n} \sigma_{n,k} p_k(x),\\
  b_n(x) &= \sum_{k=0}^{n} \tau_{n,k} q_k(x),\\
  q_n(x) &= \sum_{k=0}^{n} \ts_{n,k} p_k(x).
\end{align*}
Then we have
  \[
    \sigma_{n,k} = \sum_{r=k}^{n} \tau_{n,r} \ts_{r,k}.
  \]
\end{lem}
\begin{proof}
Observe that
  \[
\sum_{k=0}^{n} \sigma_{n,k} p_k(x) = b_n(x) = \sum_{r=0}^{n} \tau_{n,r} q_r(x)
  = \sum_{r=0}^{n} \tau_{n,r} \sum_{k=0}^r\ts_{r,k} p_k(x).
  \]
  Since \( \{p_n(x)\}_{n\ge0} \) is a basis of the polynomial space,
  we obtain the lemma.
\end{proof}

We note that an equivalent statement of \Cref{lem:inter-mixed}, for
the case that \( \{p_n(x)\}_{n\ge0} \) and \( \{q_n(x)\}_{n\ge0} \)
are orthogonal polynomials and \( b_n(x) = x^n \), has been proved in
\cite[Proposition~2.2]{KS15}. \Cref{lem:inter-mixed} allows us to find
a lecture hall graph model for \( \sigma_{n,k} \) using those for
\( \tau_{n,k} \) and \( \ts_{n,k} \). To give a precise statement we
introduce some definitions.

\begin{defn}
  The \emph{height} of a weight system \( w \) is the smallest integer
  \( \ell \) such that \( w(t;i,j)=0 \) for all \( t\ge \ell \). If
  there is no such integer, the height of \( w \) is defined to be
  \( \infty \). For a weight system \( w^{(1)} \) of height
  \( \ell<\infty \) and any weight system \( w^{(2)} \), we define
  \( w^{(1)}\sqcup w^{(2)} \) to be the weight system obtained by
  adding \( w^{(2)} \) on top of \( w^{(1)} \), that is,
\[
     (w^{(1)}\sqcup w^{(2)})(t;i,j) =
    \begin{cases}
     w^{(1)}(t;i,j) & \mbox{if \( t<\ell \)},\\
     w^{(2)}(t-\ell;i,j) & \mbox{if \( t\ge \ell \).}
    \end{cases}
\]
\end{defn}

\begin{lem}\label{lem:inter-mixed2}
  Let \( \sigma_{n,k}, \tau_{n,k} \), and \( \ts_{n,k} \) be given as
  in \Cref{lem:inter-mixed}. Suppose that \( w^{(1)} \) is a weight
  system of finite height with \( h^{w^{(1)}}_{n,k} = \tau_{n,k} \)
  and that \( w^{(2)} \) is a weight system with
  \( h^{w^{(2)}}_{n,k} = \ts_{n,k} \). Then
  \( h^{w^{(1)}\sqcup w^{(2)}}_{n,k} = \sigma_{n,k} \).
\end{lem}
\begin{proof}
  Let \( \ell \) be the height of \( w^{(1)} \). Since
  \begin{align*}
    h^{w^{(1)}\sqcup w^{(2)}}_{n,k}
    &= \sum_{p:(k,\infty)\to (n,0)} (w^{(1)}\sqcup w^{(2)})(p)\\
    &= \sum_{r=k}^{n} 
      \left( \sum_{p:(k,\infty)\to (r,\ell)} w^{(2)}(p) \right)
      \left( \sum_{p:(r,\ell)\to (n,0)} w^{(1)}(p) \right) 
      =    \sum_{r=k}^{n} \tau_{n,r} \ts_{r,k},
  \end{align*}
we have \( h^{w^{(1)}\sqcup w^{(2)}}_{n,k}=\sigma_{n,k} \).
\end{proof}

For a sequence \( d=(d_0,d_1,d_2,\ldots) \), define
\[
  (x|d)^n = (x-d_0)(x-d_1) \cdots (x-d_{n-1}).
\]
We will call the mixed moments \( \ts_{n,k} \) relative to
\( \{(x|d)^n\}_{n\ge0} \) the \emph{factorial mixed moments}. In most
cases we will consider the original mixed moments \( \sigma_{n,k} \),
the factorial mixed moments \( \ts_{n,k} \) for some \( (x|d)_n \),
and the mixed moments relative to the continuous \( q \)-Hermite
polynomials.

\section{Properties of weight systems}
\label{sec:guess-prov-techn}

In this section we prove useful properties of weight systems of height
\( 1 \) and weight systems of infinite height. Firstly, we consider
weight systems of height \( 1 \). In this case our lecture hall graph
model is equivalent to a special case of the planar network of Fomin
and Zelevinsky \cite{Fomin2000}. Using their result
\cite[Lemma~6]{Fomin2000}, given a lower unitriangular matrix
\( (\sigma_{n,k})_{n,k=0}^\infty \), one can deduce that there is a
unique weight system \( w \) of height \( 1 \) satisfying
\( h^w_{n,k} = \sigma_{n,k} \). We give an explicit formula for the
weight system \( w \) using minors of the matrix
\( (\sigma_{n,k})_{n,k=0}^\infty \).

In the case of weight systems of infinite height, there can be many
weight systems \( w \) satisfying \( h^w_{n,k} = \sigma_{n,k} \). We
provide a simple algorithm to discover a weight system of infinite
height. Although our algorithm does not always produce a correct
weight system, it does provide a valid one for all orthogonal
polynomials in the \( q \)-Askey scheme except for the case of the
continuous \( q \)-Hermite polynomials. In particular, this algorithm
allows us to discover systematically the weight system for the mixed
moments of little \( q \)-Jacobi polynomials in \cite{LHT}, which was
first discovered ``by accident'' when they happened to come across the
results in lecture hall partitions \cite{trunc_LHP}. Furthermore, this
algorithm can also be used to discover weight systems for big
\( q \)-Jacobi polynomials and Askey--Wilson polynomials in this
paper. We also provide various recurrences for weight systems which
can be used to prove that a given weight system \( w \) indeed
satisfies \( h^w_{n,k} = \sigma_{n,k} \).

\subsection{Weight systems of height 1}

In this subsection we will show that weight systems of height 1 are
particularly useful because their weights are uniquely determined by
\( h^w_{n,k} \) and they can be used to build other weight systems.
Observe that the lecture hall graph of height \( 1 \) can be
identified with a staircase grid as shown in \Cref{fig:image7}.
This implies that the total number of paths \( p:(k,1) \to (n,0) \) is
\( \binom{n}{k} \).

\begin{figure}
  \centering
  \begin{tikzpicture}[scale=1.7]
    \LHLL{5}1 \LHlabel51
    \draw[white, thick] (0,1) -- (5,1);
    \draw [red, very thick] (5,0) -- (5,1/6) -- (4,1/5)
     -- (4,2/5) -- (3,2/4) -- (2,2/3) -- (2,1);
    \node at (5.5,.5) {\(\Leftrightarrow\)};
   \end{tikzpicture}
  \begin{tikzpicture}[scale=0.7]
    \TRIL{5}5
    \draw [red, very thick] (5,0) -- (5,1) -- (4,1)
     -- (4,2) -- (3,2) -- (3,2) -- (2,2);
  \end{tikzpicture}
  \caption{A path in the lecture hall graph of height \( 1 \) and
    its corresponding path in the staircase lattice.}
  \label{fig:image7}
\end{figure}

In order to show the uniqueness for a weight system of height \( 1 \),
we introduce some notation. Given a matrix
\( A=(a_{i,j})_{i,j=0}^\infty \), let \( A_{r}(i,j) \) denote the
minor of \( A \) with row indices \( i,i+1,\ldots,i+r-1 \) and column
indices \( j,j+1,\ldots,j+r-1 \).

\begin{prop}\label{prop:MM/MM}
  Let \( \{a_{n,k}\}_{n\ge k\ge 0} \) be a triangular array of
  indeterminates and let \( A=(a_{n,k})_{n,k\ge0} \), where
  \( a_{n,k}=0 \) if \( n<k \). Then there is a unique weight system
  \( w \) of height \( 1 \) satisfying \( h^w_{n,k} = a_{n,k} \) for
  all \( n\ge k\ge0 \). Moreover, each \( w(0;i,j) \) is the rational
  function in the indeterminates \( a_{n,k} \) given by
  \begin{equation}\label{eq:MM/MM}
    w(0;i,j) = \frac{A_{j+1}(i-j+1,0) A_{j}(i-j,0)}{A_{j}(i-j+1,0) A_{j+1}(i-j,0)}.
  \end{equation}
\end{prop}
\begin{proof}
  Suppose that \( w \) is a weight system of height \( 1 \) satisfying
  \( h^w_{n,k} = a_{n,k} \) for all \( n\ge k\ge0 \). Then, by the LGV
  lemma,
\begin{equation}\label{eq:M(k,0,r)}
   A_{r}(k,0) = \prod_{a=0}^{r-1} \prod_{b=0}^{k-1} w(0;a+b,a),
 \end{equation}
 because there is only one nonintersecting family of paths from
 \( (0,1),(1,1),\dots,(r-1,1) \) to
 \( (k,0),(k+1,0),\dots,(k+r-1,0) \). More precisely, for
 \( 0\le a\le r-1 \), the unique path from \( (a,1) \) to
 \( (k+a,0) \) has the east steps of weights
 \( w(0;a,a), w(0;a+1,a),\dots,w(0;a+k-1,a) \). See
 \Cref{fig:image1} for an example.
\begin{figure}
  \centering
  \begin{tikzpicture}[scale=1.5]
    \LHLL{6}1
    \LHlabel{6}1
    \draw[white, thick] (0,1) -- (6,1);
    \draw [red, very thick] (4,0) -- (0,0) -- (0,1);
    \draw [red, very thick] (5,0) -- (5,1/6) -- (4,1/5) -- (3,1/4) -- (2,1/3)
    -- (1,1/2) -- (1,1);
    \draw [red, very thick] (6,0) -- (6,2/7) -- (5,2/6)
    -- (4,2/5) -- (3,2/4) -- (2,2/3) -- (2,1);
  \end{tikzpicture}
  \caption{The unique family of nonintersecting paths
  from \( (0,1), (1,1), (2,1) \) to \( (4,0),(5,0),(6,0) \).}
  \label{fig:image1}
\end{figure}
 By substituting \eqref{eq:M(k,0,r)}
 into the right-hand side of \eqref{eq:MM/MM}, we get the left-hand
 side of \eqref{eq:MM/MM}.
\end{proof}

\begin{remark}
Our path model also has a close connection with totally
positive matrices. Fomin and Zelevinsky \cite{Fomin2000} showed that a
square matrix \( A \) is totally positive if and only if every
``initial'' minor is positive. Moreover, they showed that in this case
there is a unique weight system on a certain planar graph that gives
a combinatorial meaning to the \( (i,j) \)-entry of \( A \). Our
path model of height \( 1 \) for the mixed moment
\( \sigma_{n,k} \) is the unique weight system of Fomin and Zelevinsky
for the lower unitriangular matrix
\( (\sigma_{n,k})_{n,k=0}^\infty \). As an application we obtain a
total positivity of the matrix of the mixed moments (and also for the
coefficients) of the big \( q \)-Jacobi polynomials with some
reparametrization. More details will be given in our forthcoming paper.
\end{remark}

\begin{remark}
Nakagawa et al.~\cite{Nakagawa2001}
used a lattice equivalent to a weight system of height \( 1 \)
in their study of Macdonald's ninth variation of Schur functions.
\end{remark}

The following proposition shows that given an array \( \{a_{n,k}\} \),
finding a height 1 weight system \( w \) with
\( h^w_{n,k} = a_{n,k} \) is equivalent to the same problem with
\( e^w_{n,k} = a_{n,k} \). Hence, by \Cref{prop:MM/MM}, there is also
a unique weight system \( w \) of height \( 1 \) such that
\( e^w_{n,k} = a_{n,k} \).

\begin{prop}\label{pro:w=overline w}
  Let \( w \) be a weight system of height \( 1 \). Define the weight
  system \( \overline{w} \) of height \( 1 \) by
  \( \overline{w}(0;i,j) = w(0;i,i-j) \) for \( 0\le j\le i \). Then
  \[
    h^w_{n,k} = e^{\overline{w}}_{n,k}, \qquad 
    e^w_{n,k} = h^{\overline{w}}_{n,k}.
  \]
\end{prop}
\begin{proof}
  By \Cref{pro:3bij}, we have
  \[
    h^w_{n,k} = \sum_{\lambda\in X} \prod_{i=k}^{n-1} w(0;i,\lambda_{i+1}),
  \]
  where \( X \) is the set of integer sequences
  \( \lambda=(\lambda_{k+1},\dots,\lambda_n) \) such that
  \[
    1> \frac{\lambda_{k+1}}{k+1} \ge \frac{\lambda_{k+2}}{k+2}
   \ge \cdots \ge \frac{\lambda_{n}}{n}\ge0.
  \]
  Similarly,
  \[
    e^{\overline{w}}_{n,k} = \sum_{\mu\in Y} \prod_{i=k}^{n-1} \overline{w}(0;i,\mu_{i+1}),
  \]
  where \( Y \) is the set of integer sequences
  \( \mu=(\mu_{k+1},\dots,\mu_n) \) such that
  \[
    0 \le \frac{\mu_{k+1}}{k+1} < \frac{\mu_{k+2}}{k+2}
   < \cdots < \frac{\mu_{n}}{n} <1.
  \]
  
  For \( \lambda\in X \), define \( \psi(\lambda)=\mu \)
  by \( \mu_i = i-1-\lambda_i \). It is straightforward to check that
  \( \psi:X\to Y \) is a weight-preserving bijection, which implies the desired identity.
\end{proof}

The following two lemmas give recurrences for \( h^w_{n,k} \), which
can be used to prove an identity of the form
\( \sigma_{n,k} = h^w_{n,k} \).

\begin{lem}\label{lem:rec}
  Let \( w \) be a weight system of height \( 1 \). We have
  \[
    h^{w}_{n,k} =  w(0;k,k) h^{w'}_{n-1,k} + h^{w'}_{n-1,k-1},
  \]
  where \( w'(0;i,j) = w(0;i+1,j) \) for \( i,j\in \ZZ_{\ge0} \).
\end{lem}
\begin{proof}
  Since \( w \) is of height \( 1 \), we have
  \[
    h^{w}_{n,k} = \sum_{p:(k,1)\to(n,0)} w(p).
  \]
  Consider a path \( p:(k,1)\to(n,0) \). If the first east step of
  \( p \) is the \( k \)th east step from the bottom, then \( p \)
  passes through \( v^0_{k+1,k} \). Otherwise \( p \) passes
  through \( v^0_{k,k-1} \) as shown in \Cref{fig:image3}.
\begin{figure}
  \centering
  \begin{tikzpicture}[scale=1.5]
    \LHLL{6}1
    \draw[white, thick] (0,1) -- (6,1);
    \draw [red, very thick] (4,3/5) -- (3,3/4) -- (3,1);
    \node at (3,1.2) {\( (k,1) \)};
    \node [right] at (6,0) {\( (n,0) \)};
    \node [right] at (4,3/5) {\( v^0_{k+1,k} \)};
    \node [left] at (3,3/4) {\( v^0_{k,k} \)};
    \node at (3.5,0.85) {\small \( w(0;k,k) \)};
  \end{tikzpicture}
  \begin{tikzpicture}[scale=1.5]
    \LHLL{6}1
    \draw[white, thick] (0,1) -- (6,1);
    \draw [red, very thick] (3,2/4) -- (3,1);
    \node at (3,1.2) {\( (k,1) \)};
    \node [right] at (6,0) {\( (n,0) \)};
    \node [right] at (3,1/2) {\( v^0_{k,k-1} \)};
  \end{tikzpicture}
  \caption{The beginning of a path \( p:(k,1)\to(n,0) \) is drawn as
    in the first diagram if it has an east step of maximum height,
    which has weight \( w(0;k,k) \). Otherwise, \( p \) is drawn as in
    the second diagram.}
  \label{fig:image3}
\end{figure}
Thus, 
  \begin{equation}\label{eq:1}
    h^{w}_{n,k} = w(0;k,k) \sum_{p:v^0_{k+1,k}\to(n,0)} w(p)
      + \sum_{p:v^0_{k,k-1}\to(n,0)} w(p).
  \end{equation}

  Now consider a path \( p:v^0_{k,k-1}\to(n,0) \). Note that
  \( p \) is determined by its east steps. Let
  \( p':v^0_{k-1,k-1}\to(n-1,0) \) be the path obtained from \( p \)
  by replacing each east step \( (v^0_{i,j},v^0_{i+1,j}) \) of
  \( p \) by an east step \( (v^0_{i-1,j},v^0_{i,j}) \) as shown
  in \Cref{fig:image4}. By construction, we have
  \( w(p) = w'(p') \). Since the map \( p\mapsto p' \)
  is a bijection from the set of paths \( p:v^0_{k,k-1}\to(n,0) \)
  to the set of paths \( p':v^0_{k-1,k-1}\to(n-1,0) \), we have
  \begin{equation}\label{eq:2}
    \sum_{p:v^0_{k,k-1}\to(n,0)} w(p) = 
    \sum_{p':v^0_{k-1,k-1}\to(n-1,0)} w'(p')= h^{w'}_{n-1,k-1}.
  \end{equation}

  \begin{figure}
  \centering
  \begin{tikzpicture}[scale=1.5]
    \LHLL{6}1
    \draw[white, thick] (0,1) -- (6,1);
    \draw [red, very thick] (3,2/4) -- (4,2/5) -- (4,1/5) -- (5,1/6)
     -- (5,0) -- (6,0);
    \node at (3,1.2) {\( (k,1) \)};
    \node [right] at (6,0) {\( (n,0) \)};
    \node [left] at (3,1/2) {\( v^0_{k,k-1} \)};
    \begin{scope}[shift={(0,-1.7)}]
    \LHLL{6}1
    \draw[white, thick] (0,1) -- (7,1);
    \draw [red, very thick] (2,2/3) -- (3,2/4) -- (3,1/4) -- (4,1/5)
     -- (4,0) -- (5,0);
    \node at (2,1.2) {\( (k-1,1) \)};
    \node [below] at (5,0) {\( (n-1,0) \)};
    \node [left] at (2,2/3) {\( v^0_{k-1,k-1} \)};
    \end{scope}
  \end{tikzpicture}
  \caption{The correspondence between a path
    \( p:v^0_{k,k-1}\to(n,0) \) and a path
    \( p:v^0_{k-1,k-1}\to(n-1,0) \).}
  \label{fig:image4}
\end{figure}

The lemma then follows from \eqref{eq:1} and \eqref{eq:2}.
\end{proof}

\begin{lem}\label{lem:rec1+}
  Let \( w \) be a weight system of height 1. We have
  \[
    h^{w}_{n,k} = w(0;n-1,0) h^{w}_{n-1,k} +  h^{w^+}_{n-1,k-1},
  \]
  where \( w^+(0;i,j) = w(0;i+1,j+1) \) for \( i,j\in \ZZ_{\ge0} \).
\end{lem}
\begin{proof}
  This can be proved similarly as in the proof of \Cref{lem:rec}
  except that we consider the last east step of
  \( p:(k,1)\to (n,0) \) instead of the first east step.
\end{proof}

Using \Cref{lem:rec1+} we can prove the following proposition, which
states that a weight system for mixed moments can be obtained from a
weight system for factorial mixed moments by adding a weight system of
height \( 1 \) at the bottom.

\begin{prop}\label{pro:add-one-row}
  Fix a polynomial sequence \( \{p_n(x)\}_{n\ge0} \) with
  \( \deg(p_n(x)) = n \). Let \( \sigma_{n,k} \) and \( \ts_{n,k} \)
  be the mixed moments and the factorial mixed moments of
  \( \{p_n(x)\}_{n\ge0} \), i.e.,
  \[
    x^n = \sum_{k=0}^{n} \sigma_{n,k} p_k(x),
    \qquad
    (x|d)^n = \sum_{k=0}^{n} \ts_{n,k} p_k(x).
  \]
  Suppose that \( \ts_{n,k}=h^{\widetilde{w}}_{n,k} \)
  for a weight system \( \widetilde{w} \).
  Then \( \sigma_{n,k}=h^{w}_{n,k} \), where \( w \)
  is the weight system defined by
  \[
    w(t;i,j) =
    \begin{cases}
     d_{j} & \mbox{if \( t=0 \)},\\
     \widetilde{w}(t-1;i,j) & \mbox{if \( t\ge 1 \)}.
    \end{cases}
  \]
\end{prop}

\begin{proof}
  By \Cref{lem:inter-mixed2}, it suffices to show that
\begin{equation}\label{eq:x=x|d}
  x^n = \sum_{k=0}^{n} h^{w_1}_{n,k} (x|d)^k,
\end{equation}
where \( w_1 \) is the weight system of height \( 1 \) defined by
\( w_1(0;i,j) = d_{j} \). We will prove this by induction on \( n \).

  If \( n=0 \), then \( k=0 \) and both sides of \eqref{eq:x=x|d} are
  equal to \( 1 \). Let \( n\ge1 \) and suppose that \eqref{eq:x=x|d} holds for
  \( n-1 \). By \Cref{lem:rec1+}, 
  \[
    h^{w_1}_{n,k} = d_0 h^{w_1}_{n-1,k} + h^{w^+_1}_{n-1,k-1}.
  \]
  Then, by the induction hypothesis,
  \begin{align*}
    \sum_{k=0}^{n} h^{w_1}_{n,k} (x|d)^k
    &= d_0 \sum_{k=0}^{n-1} h^{w_1}_{n-1,k} (x|d)^k 
      +  (x-d_0)\sum_{k=1}^{n} h^{w^+_1}_{n-1,k-1} (x|d^+)^{k-1}\\
    &= d_0 x^{n-1} +(x-d_0)x^{n-1} = x^{n},
  \end{align*}
  where \( d^+=(d_1,d_2,\ldots) \). Thus \eqref{eq:x=x|d} also holds
  for \( n \).
\end{proof}

\subsection{Weight systems of infinite height}
\label{sec:weight-syst-infin}

In this subsection we consider weight systems of infinite height. To
see the difference between weight systems of infinite height and those
of height \( 1 \), suppose that given a triangular array
\( \{a_{n,k}\}_{n\ge k\ge0} \) we want to find a weight system \( w \)
such that \( h_{n,k}^w = a_{n,k} \). If \( w \) is a weight system of
height \( 1 \), then it is uniquely determined and we can use the
formula in \Cref{prop:MM/MM} to get the weight system. If, on the
other hand, \( w \) is of infinite height, the weight \( w \) is not
uniquely determined in general. How can we find such a weight system
then? We propose the following algorithm to construct a possible
candidate \( w \).

\begin{algo}[Finding a weight system of infinite height]
  \label{alg:1}
  Fix a monomial ordering for the monomials of the parameters in
  \( \{a_{n,k}\}_{n\ge k\ge0} \). Define a weight system \( w \) such
  that the weight of the \( j \)th east step from the bottom between
  \( x=n-1 \) and \( x=n \) to be the \( j \)th term of
  \( a_{n,n-1} \) according to the monomial ordering.
\end{algo}

It may seem too optimistic to hope that such a simple algorithm
provides a correct weight system. However, this algorithm gives a
valid weight system of infinite height for all orthogonal polynomials
in the \( q \)-Askey scheme except the continuous \( q \)-Hermite
polynomials. Once we have a candidate weight system \( w \) explicitly
constructed via this algorithm, we can prove the identity
\( h_{n,k}^w = a_{n,k} \) using recursions such as the one in the
following lemma.

\begin{lem}\label{lem:rec2}
  For any weight system \( w \) and a positive integer \( \ell \), we have
  \[
    h^w_{n,k} =  \sum_{r=k}^{n} h^{w_\ell}_{n,r} h^{w_\ell^+}_{r,k} ,
  \]
  where \( w_\ell \) is the weight system of height \( \ell \) defined
  by 
  \[
    w_\ell(t;i,j) =
    \begin{cases}
      w(t;i,j) & \mbox{if \( t<\ell \)},\\
     0 & \mbox{if \( t\ge\ell \),}
    \end{cases}
  \]
  and \( w_\ell^+ \) is the weight system defined by
  \( w_\ell^+(t;i,j) = w(t+\ell;i,j) \).
\end{lem}

\begin{proof}
  Consider a path \( p :(k,\infty) \to (n,0) \). Let \( r \) be the
  largest integer such that \( p \) passes through \( (r,\ell) \).
  Then we can decompose \( p \) as \( p=p_1p_2 \), where
  \( p_1:(k,\infty) \to (r,\ell) \) and \( p_2:(r,\ell) \to (n,0) \),
  and \( p_2 \) starts with a south step.
  Thus,
  \begin{align*}
    h^{w}_{n,k}
    &= \sum_{p:(k,\infty)\to(n,0)} w(p)\\
    &= \sum_{r=k}^{n} \sum_{p_1:(k,\infty) \to (r,\ell)} w(p_1)
      \sum_{\substack{p_2:(r,\ell) \to (n,0) \\ p_2 \text{ starts with a south step}}} w(p_2)\\
    &= \sum_{r=k}^{n} \sum_{p_1:(k,\infty) \to (r,0)} w_\ell^+(p_1) \sum_{p_2:(r,\ell) \to (n,0)} w_\ell(p_2)\\
    &= \sum_{r=k}^{n} h^{w_\ell}_{n,r} h^{w_\ell^+}_{r,k} . \qedhere
  \end{align*}
\end{proof}

If there is a weight system \( w_1 \) of height \( 1 \) and a weight
system \( w \) of infinite height with the ``shifting property'',
which is given in \eqref{eq:shifting} below, we can construct a weight
system of any height.

\begin{prop}\label{prop:any-ht}
  Suppose that \( \{\sigma_{n,k}\}_{n,k\ge0} \) satisfies
\[
  \sigma_{n,k} = h^{w}_{n,k} = h^{w_1}_{n,k}
\]
for a weight system \( w_1 \) of height \( 1 \) and a weight system
\( w \) of infinite height such that
  \begin{equation}\label{eq:shifting}
   w(t;i,j) = C_i^{t} w(0;i,j),
  \end{equation}
  for all \( t,i,j \in \ZZ_{\ge0} \) with \( j\le i \). Then, for any
  integer \( \ell\ge1 \), we have
  \[
    \sigma_{n,k} = h^{w^{(\ell)}}_{n,k},
  \]
  where \( w^{(\ell)} \) is the weight system of height \( \ell \) given by
  \[
    w^{(\ell)}(t;i,j)=
\begin{cases}
 w(t;i,j) & \mbox{if \( 0\le t< \ell-1 \)},\\
 C_i^{\ell-1} w_1(0;i,j) & \mbox{if \( t=\ell-1 \).}
\end{cases}
  \]
\end{prop}
\begin{proof}
  Using the notation in \Cref{lem:rec2} with \( \ell \) replaced by
  \( \ell-1 \), we have
  \[
    h^w_{n,k} =  \sum_{r=k}^{n} h^{w_{\ell-1}}_{n,r} h^{w_{\ell-1}^+}_{r,k}.
  \]
  By \eqref{eq:shifting}, we have
  \( w_{\ell-1}^+(t;i,j) = w(t+\ell-1;i,j) = C_i^{\ell-1} w(t;i,j) \).
  Thus, by \Cref{lem:wt-mult} and the assumption, we obtain
  \[
    h^{w_{\ell-1}^+}_{r,k} =
    C_k^{\ell-1} C_{k+1}^{\ell-1} \cdots C_{r-1}^{\ell-1} h^{w}_{r,k}
    = C_k^{\ell-1} C_{k+1}^{\ell-1} \cdots C_{r-1}^{\ell-1} h^{w_1}_{r,k}
    = h^{w'}_{r,k},
  \]
  where \( w' \) is the weight system of height \( 1 \) defined by
  \( w'(0;i,j) = C_i^{\ell-1} w_1(0;i,j) \). Hence,
  \[
    h^w_{n,k} =  \sum_{r=k}^{n} h^{w_{\ell-1}}_{n,r} h^{w'}_{r,k}.
  \]
  On the other hand, applying \Cref{lem:rec2} to the weight system
  \( w^{(\ell)} \), we also have
  \[
    h^{w^{(\ell)}}_{n,k} =  \sum_{r=k}^{n} h^{w_{\ell-1}}_{n,r} h^{w'}_{r,k}.
  \]
  Combining the results, we obtain
  \( \sigma_{n,k} = h^{w}_{n,k} = h^{w^{(\ell)}}_{n,k} \).
\end{proof}

\section{A bootstrapping method from Stieltjes--Wigert to Askey--Wilson}
\label{sec:bootstr-meth-from}

In this section, we construct lecture hall graph models for the mixed
moments of Stieltjes--Wigert polynomials, \( q \)-Bessel polynomials,
little \( q \)-Jacobi polynomials, big \( q \)-Jacobi polynomials, and
Askey--Wilson polynomials. We first find a lecture hall graph model of
height \( 1 \) for the mixed moments of Stieltjes--Wigert polynomials,
and models of height \( 1 \) and infinite height for \( q \)-Bessel
polynomials. We then prove ``expanding lemmas'' (\Cref{lem:two-to-one}
and \Cref{lem:two-to-one-gen}), which turn a weight system of height
\( 1 \) to a weight system of height \( 2 \). Using the expanding
lemmas, we obtain lecture hall graph models of height \( 1 \) and
infinite height for little \( q \)-Jacobi polynomials, big
\( q \)-Jacobi polynomials, and Askey--Wilson polynomials.

\subsection{Stieltjes--Wigert}

The monic \emph{Stieltjes--Wigert polynomials} are defined by
\[
  S_n(x;q) = (-1)^{n} q^{-n^2}
  \qHyper11{q^{-n}}{0}{q,-q^{n+1}x}.
\]
The mixed moments \( \sigma_{n,k} \) and the coefficients
\( \nu_{n,k} \) of Stieltjes--Wigert polynomials are given by
\begin{align*}
  \sigma_{n,k} &= q^{k^2-n^2+\binom{n-k}{2}} \qbinom{n}{k}, \\
  \nu_{n,k} &= (-1)^{n-k} q^{k^2-n^2} \qbinom{n}{k}.
\end{align*}
Hence the mixed moments of Stieltjes--Wigert polynomials are
\( q \)-binomial coefficients with some factors.
The \( q \)-binomial coefficients satisfy the following recurrences:
\begin{equation}\label{eq:rec-q-bin}
  \qbinom{n}{k} = q^k\qbinom{n-1}{k} + \qbinom{n-1}{k-1} =
  \qbinom{n-1}{k} + q^{n-k}\qbinom{n-1}{k-1}.
\end{equation}
We begin with two simple weight systems giving \( q \)-binomial
coefficients with some possible factors.

\begin{lem}\label{lem:q-binom-wt}
  Let \( w \) be the weight system of height \( 1 \) defined by
  \( w(0;i,j) = q^{j} \). Then
\[
  h^w_{n,k} = \qbinom{n}{k}.
\]
\end{lem}
\begin{proof}
  Note that the weight $w(0;i,j) = q^j$ is independent of $i$
  and \( w(0;k,k)=q^k \).
  Therefore, by applying \Cref{lem:rec}, we obtain the recurrence
  \[
    h_{n,k}^w= q^k h_{n-1,k}^w+h_{n-1,k-1}^w
  \]
  for $1\le k\le n-1$, which is the same recurrence as the first one
  in \eqref{eq:rec-q-bin}. Since \( h^w_{n,k} \) and
  \( \qbinom{n}{k} \) have the same initial conditions, namely,
  $h_{n,0}^w=h_{n,n}^w=1$, we conclude that
  $h_{n,k}^w=\qbinom{n}{k}$.
\end{proof}

\begin{lem}\label{lem:q-binom-wt2}
  Let \( w \) be the weight system of height \( 1 \) defined by
  \( w(0;i,j) = q^{i-j} \). Then
\[
  h^w_{n,k} = q^{\binom{n-k}{2}} \qbinom{n}{k}.
\]
\end{lem}
\begin{proof}
  Let \( w' \) be the weight system of height \( 1 \) defined by
  \( w'(0;i,j) = q^{-j} \). By \Cref{lem:q-binom-wt},
  \( h^{w'}_{n,k} =  \Qbinom{n}{k}{q^{-1}} = q^{-k(n-k)}\qbinom{n}{k} \).
  Since \( w(0;i,j) = q^{i}w'(0;i,j) \), by \Cref{lem:wt-mult},
  \[
    h^{w}_{n,k} = q^{k+(k+1) + \cdots + (n-1)}h^{w'}_{n,k} =
    q^{\binom{n-k}{2}} \qbinom{n}{k}. \qedhere
  \]
\end{proof}

We now give a weight system for the mixed moments of Stieltjes--Wigert
polynomials.

\begin{prop}\label{prop:SW-wt}
  Let \( w \) be the weight system of height \( 1 \) defined by
  \( w(0;i,j) = q^{-i-j-1} \). Then
\[
  h^w_{n,k} = \sigma_{n,k}.
\]
\end{prop}
\begin{proof}
  Let \( w' \) be the weight system of height \( 1 \) defined by
  \( w'(0;i,j) = q^{i-j} \). By \Cref{lem:q-binom-wt2},
  \( h^{w'}_{n,k} = q^{\binom{n-k}{2}} \qbinom{n}{k} \).
  Since \( w(0;i,j) = q^{-2i-1}w'(0;i,j) \), by \Cref{lem:wt-mult},
  \[
    h^{w}_{n,k} = q^{-(2k+1) - (2k+3) - \cdots - (2n-1)}h^{w'}_{n,k} =
    q^{k^2-n^2+\binom{n-k}{2}} \qbinom{n}{k} = \sigma_{n,k}. \qedhere
  \]
\end{proof}

Since \( \sigma_{n,k} \) is a polynomial in \( q \), we do not
consider its weight system of infinite height.

\subsection{$q$-Bessel polynomials}

The (monic) \emph{\( q \)-Bessel polynomials} (also known as
\emph{alternative \( q \)-Charlier polynomials}) are defined by
\begin{equation}
  \label{eq:alt-q-charlier}
  p_n(x;a;q) = \frac{1}{(-1)^{n}q^{-\binom{n}{2}}(-aq^n;q)_{n}}
  \qHyper21{q^{-n},-aq^n}{0}{q,qx}.
\end{equation}
Let \( \sigma^b_{n,k}(a;q) \) and \( \nu^b_{n,k}(a;q) \) denote the mixed
moments and the dual mixed moments of the \( q \)-Bessel polynomials:
\[
  x^n = \sum_{k=0}^{n} \sigma^b_{n,k}(a;q) p_k(x;a;q),
  \qquad
  p_n(x;a;q) = \sum_{k=0}^{n} \nu^b_{n,k}(a;q) x^k.
\]
These quantities have simple product formulas.

\begin{prop}\label{pro:1}
  We have
  \begin{align}
    \label{eq:mu^C}
    \sigma^b_{n,k}(a;q)
    &=\qbinom{n}{k} \frac{1}{(-aq^{2k+1};q)_{n-k}}
      =\qbinom{n}{k} \frac{(-aq;q)_{2k}}{(-aq;q)_{n+k}},\\
    \label{eq:nu^C}
    \nu^b_{n,k}(a;q)
    &= (-1)^{n-k} q^{\binom{n-k}2} \qbinom{n}{k} \frac{1}{(-aq^{n+k};q)_{n-k}}
      =(-1)^{n-k} q^{\binom{n-k}2} \qbinom{n}{k} \frac{(-a;q)_{n+k}}{(-a;q)_{2n}}.
  \end{align}
\end{prop}

\begin{proof}
  By \eqref{eq:alt-q-charlier}, we have
  \[
    \nu^b_{n,k}(a;q)
    = \frac{1}{(-1)^{n}q^{-\binom{n}{2}}(-aq^n;q)_{n}} \cdot
    \frac{(q^{-n};q)_{k} (-aq^{n};q)_{k}}{(q;q)_{k}} q^k,
  \]
  which is equivalent to \eqref{eq:nu^C}. To prove \eqref{eq:mu^C} let
  \( a_{n,k} \) and \( b_{n,k} \) be the right hand side of
  \eqref{eq:mu^C} and \eqref{eq:nu^C}, respectively. Then it suffices
  to show that
  \[
    \sum_{j=0}^{m} b_{m,j} a_{j,n} = \delta_{m,n}.
  \]
  Since the above sum is \( 0 \) if \( m<n \), we may assume
  \( m\ge n \). Then
  \begin{align*}
    \sum_{j=0}^{m} b_{m,j} a_{j,n}
    &= \frac{(-aq;q)_{2n}}{(-a;q)_{2m}}
      \sum_{j=n}^{m} (-1)^{m-j} q^{\binom{m-j}{2}}
      \qbinom{m}{j} \qbinom{j}{n} \frac{(-a;q)_{m+j}}{(-aq;q)_{n+j}} \\
    &= \frac{(-a;q)_{m+n}}{(-a;q)_{2m}} \qbinom{m}{n} (-1)^{m-n} q^{\binom{m-n}{2}}
      \qHyper21{q^{-m+n},-aq^{m+n}}{-aq^{2n+1}}{q,q}\\
    &= \frac{(-a;q)_{m+n}}{(-a;q)_{2m}} \qbinom{m}{n} (-1)^{m-n} q^{\binom{m-n}{2}}
      \frac{(q^{-m+n+1};q)_{m-n}}{(-aq^{2n+1};q)_{m-n}} = \delta_{m,n},
  \end{align*}
  where the following identity \cite[(1.5.2)]{GR} is used:
  \[
    \qHyper21{q^{-n},b}{c}{q,cq^n/b} = \frac{(c/b;q)_n}{(c;q)_n}.
  \]
\end{proof}

We will find several weight systems for \( \sigma^b_{n,k}(a;q) \). We
first give a weight system of height \( 1 \).

\begin{prop}\label{pro:charlier-wt-one-row}
Let \( w \) be the weight system of height \( 1 \) defined by
\[
    w(0;i,j) = \frac{q^{j}(1+aq^{i})}{(1+aq^{i+j})(1+aq^{i+j+1})}.
  \]
  Then the mixed moments of the \( q \)-Bessel polynomials are given by
  \[
    \sigma^b_{n,k}(a;q) = h^{w}_{n,k}.
  \]
\end{prop}
\begin{proof}
  Let us write \( h^{w}_{n,k} = h^{w}_{n,k}(a) \). By
  \Cref{lem:rec},
  \[
    h^{w}_{n,k}(a) =  w(0;k,k) h^{w'}_{n-1,k}(a) + h^{w'}_{n-1,k-1}(a).
  \]
  Since \( w'(0;i,j) = w(0;i+1,j) \), which is equal to
  \( w(0;i,j) \) with \( a \) replaced by \( aq \), we have
  \( h^{w'}_{n,k}(a) = h^{w}_{n,k}(aq) \). Thus
  \[
    h^{w}_{n,k}(a) =   \frac{q^{k}(1+aq^{k})}{(1+aq^{2k+1})(1+aq^{2k})}
    h^{w}_{n-1,k}(aq) + h^{w}_{n-1,k-1}(aq).
  \]
  
  By induction, it suffices to show that \( \sigma^b_{n,k}(a;q) \)
  also satisfies the same recurrence, where the initial conditions are
  \( h^{w}_{n,k}(a)=\sigma^b_{n,k}(a;q)=1 \) if \( n=k \), and
  \( h^{w}_{n,k}(a)=\sigma^b_{n,k}(a;q)=0 \) if \( n<k \). By
  \eqref{eq:mu^C}, the recurrence we need to establish is
\[
  \qbinom{n}{k} \frac{(-aq;q)_{2k}}{(-aq;q)_{n+k}}
  =  \frac{q^{k}(1+aq^{k})}{(1+aq^{2k+1})(1+aq^{2k})}
    \qbinom{n-1}{k} \frac{(-aq^2;q)_{2k}}{(-aq^2;q)_{n+k-1}}
+  \qbinom{n-1}{k-1} \frac{(-aq^2;q)_{2k-2}}{(-aq^2;q)_{n+k-2}},
\]
or equivalently,
\[
  (1+aq^{2k}) \qbinom{n}{k} = q^k(1+aq^{k})\qbinom{n-1}{k} + (1+aq^{n+k})
  \qbinom{n-1}{k-1}.
\]
This follows from the two recurrences in \eqref{eq:rec-q-bin}.
\end{proof}

To find a weight system of infinite height for
\( \sigma^b_{n,k}(a;q) \), we need some lemmas. The following lemma
allows us to find a weight system of infinite height for
\( \{\sigma_{n,k}\}_{n,k\ge0} \) if a certain recurrence is satisfied.

\begin{lem} [Construction of a weight system of infinite height]
  \label{lem:construction-infinite}
  Suppose that 
  \( \{\sigma_{n,k}\}_{n,k\ge0} \) satisfy
  the following recurrence:
  \[
    \sigma_{n,k} = \sum_{r=k}^n h^{w_1}_{n,r} \sigma_{r,k}
    \prod_{i=k}^{r-1} C_i.
  \]
  for some weight system \( w_1 \) of height \( 1 \).
  Then
  \[
    \sigma_{n,k} = h^w_{n,k},
  \]
  where \( w \) is the weight system of infinite height defined by
 \[
   w(t;i,j) = C_i^{t} w_1(0;i,j).
 \] 
\end{lem}
\begin{proof}
  Since \( \sigma_{n,n} = h^w_{n,n} = 1 \), it suffices to show that
  \( h^w_{n,k} \) satisfy the same recurrence, namely,
\[
     h^w_{n,k} = \sum_{r=k}^n h^{w_1}_{n,r} h^w_{r,k}
    \prod_{i=k}^{r-1} C_i.
\]
This can be proved similarly as in the proof of \Cref{prop:any-ht},
hence we omit the details.
\end{proof}

The mixed moments \( \sigma^b_{n,k}(a;q) \) satisfy a recurrence of
the form in \Cref{lem:construction-infinite}.

\begin{lem}\label{lem:q-bessel-rec}
  We have
 \[
   \sigma^b_{n,k}(a;q) = \sum_{r=k}^{n} \qbinom{n}{r} (-a)^{r-k}
   q^{r^2 - k^2} \sigma^b_{r,k}(a;q).
  \] 
\end{lem}
\begin{proof}
  By \eqref{eq:mu^C}, the identity we need to show is
  \[
    \qbinom{n}{k} \frac{(-aq;q)_{2k}}{(-aq;q)_{n+k}}
    = \sum_{r=k}^{n} \qbinom{n}{r} (-a)^{r-k} q^{r^2 - k^2}
    \qbinom{r}{k} \frac{(-aq;q)_{2k}}{(-aq;q)_{r+k}}.
  \]
  The above identity can be rewritten as
  \[
    \frac{1}{(-aq^{2k+1};q)_{n-k}} =
    \sum_{i=0}^{n-k} \qbinom{n-k}{i}
    (-a)^{i} q^{i^2+2ik} \frac{1}{(-aq^{2k+1};q)_{i}}.
  \]
  This is equivalent to the following form of the \( q \)-binomial
  theorem; see \cite[p.~77,~Exercise~101]{Andrews2004a}:
  \[
    \frac{1}{(aq;q)_n} = \sum_{k=0}^{n} a^kq^{k^2} \qbinom{n}{k} \frac{1}{(aq;q)_k}.
    \qedhere
  \]
\end{proof}

We are now ready to give a weight system of infinite height for
\( \sigma^b_{n,k}(a;q) \). Note that this weight system can be
constructed using \Cref{alg:1} with the monomial ordering defined by
\( q<a \). See \Cref{fig:q-Bessel} for the weight system in the
following proposition.

\begin{prop}\label{pro:charlier-wt}
  Let \( w \) be the weight system defined by
  \[
    w(t;i,j) = (-aq^{2i+1})^{t} q^{j}.
  \]
  The mixed moments  of the \( q \)-Bessel
  polynomials satisfy
 \[
    \sigma^b_{n,k}(a;q) = h^{w}_{n,k}.
  \] 
\end{prop}
\begin{proof}
  Let \( w_1 \) be the weight system of height \( 1 \) given by
  \( w_1(0;i,j) = q^{j} \) so that \( h^{w_1}_{n,k} = \qbinom{n}{k} \)
  by \Cref{lem:q-binom-wt}. We can rewrite the recurrence in
  \Cref{lem:q-bessel-rec} as
 \[
   \sigma^b_{n,k}(a;q) = \sum_{r=k}^{n} h^{w_1}_{n,r} 
   \sigma^b_{r,k}(a;q) \prod_{i=k}^{r-1} (-aq^{2i+1}).
  \] 
  Then by \Cref{lem:construction-infinite} we obtain the desired result.
\end{proof}

\begin{figure}
  \centering
  \begin{tikzpicture}[scale=1.5]
    \LHLLL{3}4 
    \begin{scope}[shift={(-.5,.15)}]
      \node at (1,0) {$1$};
      \node at (2,0) {$1$};
      \node at (2,1/2) {$q$};
      \node at (3,0) {$1$};
      \node at (3,1/3) {$q$};
      \node at (3,2/3) {$q^2$};
    \end{scope}
    \begin{scope}[shift={(-.5,1.15)}]
      \node at (1,0) {$-aq$};
      \node at (2,0) {$-aq^3$};
      \node at (2,1/2) {$-aq^4$};
      \node at (3,0) {$-aq^5$};
      \node at (3,1/3) {$-aq^6$};
      \node at (3,2/3) {$-aq^7$};
    \end{scope}
    \begin{scope}[shift={(-.5,2.15)}]
      \node at (1,0) {$a^2q^2$};
      \node at (2,0) {$a^2q^6$};
      \node at (2,1/2) {$a^2q^7$};
      \node at (3,0) {$a^2q^{10}$};
      \node at (3,1/3) {$a^2q^{11}$};
      \node at (3,2/3) {$a^2q^{12}$};
    \end{scope}
    \begin{scope}[shift={(-.5,3.15)}]
      \node at (1,0) {$-a^3q^3$};
      \node at (2,0) {$-a^3q^9$};
      \node at (2,1/2) {$-a^3q^{10}$};
      \node at (3,0) {$-a^3q^{15}$};
      \node at (3,1/3) {$-a^3q^{16}$};
      \node at (3,2/3) {$-a^3q^{17}$};
    \end{scope}
    \begin{scope}[shift={(-.5,4.15)}]
      \node at (1,0) {$a^4q^4$};
      \node at (2,0) {$a^4q^{12}$};
      \node at (3,0) {$a^4q^{20}$};
    \end{scope}
    \node at (0.5,4.7) {\( \vdots \)};
    \node at (1.5,4.7) {\( \vdots \)};
    \node at (2.5,4.7) {\( \vdots \)};
    \node at (3.5,0.5) {\( \cdots \)};
    \node at (3.5,1.5) {\( \cdots \)};
    \node at (3.5,2.5) {\( \cdots \)};
    \node at (3.5,3.5) {\( \cdots \)};
  \end{tikzpicture}
  \caption{The weight system for the mixed moments of \( q \)-Bessel
    polynomials.}
  \label{fig:q-Bessel}
\end{figure}

Since we have both a weight system \( w_1 \) of height \( 1 \) and a
weight system \( w \) of infinite height for \( \sigma^b_{n,k} \), we
can construct a weight system of any height using \Cref{prop:any-ht}.

\begin{cor}\label{cor:ht-l}
  For a positive integer \( \ell \), let \( w_\ell \) be the weight
  system of height \( \ell \) defined by
  \[
    w_\ell(t;i,j)
    = \begin{cases} 
         (-aq^{2i+1})^tq^j & \mbox{if \( t<\ell-1 \)},\\
         \frac{(-aq^{2i+1})^tq^j(1+aq^i)} {(1+aq^{i+j})(1+aq^{i+j+1})} & \mbox{if \( t=\ell-1 \)}.
       \end{cases}
  \]
  Then
  \[
    \sigma^b_{n,k}(a;q) = h^{w_\ell}_{n,k}.
  \]
\end{cor}
\begin{proof}
  The weight system \( w \) in \Cref{pro:charlier-wt} satisfies
  \[
    w(\ell-1;i,j) = (-aq^{2i+1})^{\ell-1} w(0;i,j).
  \]
  Thus the result follows from \Cref{pro:charlier-wt-one-row},
  \Cref{pro:charlier-wt}, and \Cref{prop:any-ht}.
\end{proof}

\subsection{Expanding weight systems}

In this subsection we prove some lemmas which allow us to expand one
row of a weight system to two rows. Here, a row means the part of a
weight system between \( y=\ell-1 \) and \( y=\ell \) for some
\( \ell\ge1 \).

\begin{figure}
  \centering
  \begin{tikzpicture}[scale=1.5]
    \LHLL{3}1
    \LHlabel{3}1
    \small
    \draw[white, thick] (0,1) -- (3,1);
    \begin{scope}[shift={(-.5,.15)}]
      \node at (1,0) {$1+b$};
      \node at (2,0) {$1+bq$};
      \node at (2,1/2) {$(1+bq)q$};
      \node at (3,0) {$1+bq^2$};
      \node at (3,1/3) {$(1+bq^2)q$};
      \node at (3,2/3) {$(1+bq^2)q^2$};
    \end{scope}
    \node at (3.5,0.5) {\( \cdots \)};
  \end{tikzpicture} \qquad
  \begin{tikzpicture}[scale=1.5]
    \LHLL{3}2
    \LHlabel{3}2
    \draw[white, thick] (0,2) -- (3,2);
    \begin{scope}[shift={(-.5,.15)}]
      \node at (1,0) {$1$};
      \node at (2,0) {$1$};
      \node at (2,1/2) {$q$};
      \node at (3,0) {$1$};
      \node at (3,1/3) {$q$};
      \node at (3,2/3) {$q^2$};
    \end{scope}
    \begin{scope}[shift={(-.5,1.15)}]
      \node at (1,0) {$b$};
      \node at (2,0) {$bq$};
      \node at (2,1/2) {$bq^2$};
      \node at (3,0) {$bq^2$};
      \node at (3,1/3) {$bq^3$};
      \node at (3,2/3) {$bq^4$};
    \end{scope}
    \node at (3.5,0.5) {\( \cdots \)};
    \node at (3.5,1.5) {\( \cdots \)};
  \end{tikzpicture}
  \caption{Two weight systems giving the same generating functions.}
  \label{fig:two-to-one}
\end{figure}

See \Cref{fig:two-to-one} for the weight systems in the following
lemma.

\begin{lem}[Expanding lemma 1]
  \label{lem:two-to-one} 
  Let \( w_1 \) and \( w_2 \) be the weight systems defined by
  \[
    w_1(t;i,j) =
    \begin{cases}
      (1+bq^{i}) q^{j} & \mbox{if \( t=0 \)},\\
      0 & \mbox{if \( t\ge1 \)},
    \end{cases}  \qquad 
    w_2(t;i,j) =
    \begin{cases}
      q^{j}& \mbox{if \( t=0 \)},\\
      b q^{i+j}& \mbox{if \( t=1 \)},\\
      0 & \mbox{if \( t\ge2 \).}
    \end{cases}
  \]
 Then
  \[
    \sum_{p:(k,1)\to(n,0)} w_1(p) = \sum_{p:(k,2)\to(n,0)} w_2(p).
  \]
\end{lem}

\begin{proof}
  Observe that \( w_1 \) is obtained from the weight system in
  \Cref{lem:q-binom-wt} by multiplying \( (1+bq^i) \) for each step
  between \( x=i-1 \) and \( x=i \) for \( i\ge1 \). Thus, by \Cref{lem:wt-mult} and
  \Cref{lem:q-binom-wt}, we have
  \begin{equation}\label{eq:28}
    \sum_{p:(k,1)\to(n,0)} w_1(p)
    = \qbinom{n}{k} (1+bq^{k})(1+bq^{k+1})\cdots (1+bq^{n-1}).
  \end{equation}

  By \Cref{lem:rec2} with \( \ell=1 \), we have
  \[
    \sum_{p:(k,2)\to(n,0)} w_2(p)
    =  \sum_{r=k}^{n}   \sum_{p:(k,1)\to(r,0)} w'(p) \sum_{p:(r,1)\to(n,0)} w''(p),
  \]
  where \( w' \) and \( w'' \) are the weight systems of height
  \( 1 \) defined by \( w'(0;i,j) = bq^{i+j} \) and
  \( w''(0;i,j) = q^j \). By \Cref{lem:wt-mult} and
  \Cref{lem:q-binom-wt} again, we have
  \[
   \sum_{p:(k,1)\to(r,0)} w'(p) = b^{r-k}q^{k+(k+1)+\cdots+(r-1)} \qbinom{r}{k}, \qquad 
   \sum_{p:(r,1)\to(n,0)} w''(p) = \qbinom{n}{r}.
  \]
  Thus,
  \begin{align*}
    \sum_{p:(k,2)\to(n,0)} w_2(p)
    &= \sum_{r=k}^{n} b^{r-k}q^{k+(k+1)+\cdots+(r-1)} \qbinom{r}{k}
      \qbinom{n}{r} \\
    &=  \qbinom{n}{k} \sum_{r=k}^{n} b^{r-k}
      q^{\binom{r-k}{2}+k(r-k)} \qbinom{n-k}{r-k}\\
    &=  \qbinom{n}{k} \sum_{r=0}^{n-k}  b^{r}
      q^{\binom{r}{2}+kr} \qbinom{n-k}{r}\\
    &=  \qbinom{n}{k} (1+bq^k)(1+bq^{k+1})\cdots (1+bq^{n-1}),
  \end{align*}
  where the \( q \)-binomial theorem is used in the last equality.
  By \eqref{eq:28} and the above formula, we obtain the lemma.
\end{proof}

For an integer \( k \), let \( \chi_o(k)=1 \) if \( k \) is odd and
\( \chi_o(k)=0 \) otherwise. Similarly, let \( \chi_e(k)=1 \) if
\( k \) is even and \( \chi_e(k)=0 \) otherwise.

\begin{lem}[Expanding lemma 2]\label{lem:two-to-one-gen}
  Let \( w_1 \) and \( w_2 \) be the weight systems defined by
  \begin{align*}
    w_1(t;i,j) &= a_{i,t}(1+bq^i) q^{j},\\
    w_2(t;i,j) &= a_{i,\flr{t/2}}(bq^i)^{\chi_o(t)} q^{j},
  \end{align*}
  where \( a_{i,j} \) is an arbitrary quantity that depends on \( i \)
  and \( j \). Then
  \[
    \sum_{p:(k,\infty)\to(n,0)} w_1(p)
    = \sum_{p:(k,\infty)\to(n,0)} w_2(p).
  \]
\end{lem}
\begin{proof}
  Consider a path \( p:(k,\infty)\to(n,0) \). Let
  \( t_1>t_2 > \cdots >t_m \) be the integers such that \( p \) has at
  least one east step in the region \( \{(x,y): t_s\le y<t_s+1\} \)
  for each \( s \). Then the restriction of \( p \) to this region is
  a path \( p_s:(k_{s-1},t_s+1)\to (k_s,t_s) \) starting with a south
  step, for some integers \( k=k_0<k_1 < \cdots < k_m=n \). Observe
  that \( w_1(p) = w_1(p_1) \cdots w_1(p_m) \). Let
  \( p'_s:(k_{s-1},1)\to (k_s,0) \) be the path obtained by
  translating \( p_s \) by \( t_s \) units downwards. By the
  definition of \( w_1 \), we have
\[
  w_1(p_s) = a_{k_{s-1},t_s}a_{k_{s-1}+1,t_s} \cdots a_{k_s-1,t_s} w'_1(p'_s),
\]
where \( w'_1 \) is the weight system of height \( 1 \)
defined by \( w'_1(0;i,j) = (1+bq^i)q^j \).
Therefore,
\begin{equation}\label{eq:5}
    \sum_{p:(k,\infty)\to(n,0)} w_1(p)
    = \sum_{m\ge0}  \sum_{(\vec t, \vec k)} \left( \prod_{s=1}^{m} \prod_{i=k_{s-1}}^{k_s-1} a_{i,t_s} \right)
    h^{w'_1}_{k_m,k_{m-1}} h^{w'_1}_{k_{m-1},k_{m-2}} \cdots h^{w'_1}_{k_{1},k_{0}},
\end{equation}
where the last sum is over all pairs \( (\vec t,\vec k) \) of tuples
\( \vec t = (t_1>t_2 > \cdots >t_m) \) and
\( \vec k = (k=k_0<k_1 < \cdots < k_m=n) \).

For the weight system \( w_2 \), we consider a different decomposition
of a path \( p:(k,\infty)\to(n,0) \). Let \( t_1>t_2 > \cdots >t_m \)
be the integers such that \( p \) has at least one east step in the
region \( \{(x,y): 2t_s\le y<2t_s+2\} \) for each \( s \). Then the
restriction of \( p \) to this region is a path
\( p_s:(k_{s-1},2t_s+2)\to (k_s,2t_s) \) starting with a south step,
for some integers \( k=k_0<k_1 < \cdots < k_m=n \), and we have
\( w_2(p) = w_2(p_1) \cdots w_2(p_m) \). Let
\( p'_s:(k_{s-1},2)\to (k_s,0) \) be the path obtained by translating
\( p_s \) by \( 2t_s \) units downwards. By the definition of
\( w_2 \), we have
\[
  w_2(p_s) = a_{k_{s-1},t_s}a_{k_{s-1}+1,t_s} \cdots a_{k_s-1,t_s} w'_2(p'_s),
\]
where \( w'_2 \) is the weight system of height \( 2 \) defined by
\[
  w'_2(t;i,j) =
    \begin{cases}
      q^{j}& \mbox{if \( t=0 \)},\\
      b q^{i+j}& \mbox{if \( t=1 \)},\\
      0 & \mbox{if \( t\ge2 \).}
    \end{cases}
\] 
Therefore, we have
\begin{equation}\label{eq:7}
    \sum_{p:(k,\infty)\to(n,0)} w_2(p)
    = \sum_{m\ge0}  \sum_{(\vec t, \vec k)} \left( \prod_{s=1}^{m} \prod_{i=k_{s-1}}^{k_s-1} a_{i,t_s} \right)
    h^{w'_2}_{k_m,k_{m-1}} h^{w'_2}_{k_{m-1},k_{m-2}} \cdots h^{w'_2}_{k_{1},k_{0}},
\end{equation}
where the last sum is the same as in \eqref{eq:5}.

By \Cref{lem:two-to-one}, we have
\( h^{w'_1}_{n,k} = h^{w'_2}_{n,k} \) for all \( n \) and \( k \).
Therefore, the right-hand sides of \eqref{eq:5} and \eqref{eq:7} are
equal, which completes the proof.
\end{proof}

\subsection{Little $q$-Jacobi}

Let \( p_n(x;a,b;q) \) be the monic \emph{little \( q \)-Jacobi polynomial}:
\begin{equation}\label{eq:def-lqj}
  p_n(x;a,b;q) =
  \frac{(aq;q)_n}{(-1)^n q^{-\binom{n}{2}}(abq^{n+1};q)_n}
  \qHyper21{q^{-n},abq^{n+1}}{aq}{q,qx}.
\end{equation}
We denote by \( \sigma^L_{n,k}(a,b;q) \) and \( \nu^L_{n,k}(a,b;q) \)
the mixed moments and the dual mixed moments of the little
\( q \)-Jacobi polynomials, respectively:
\[
  x^n = \sum_{k=0}^{n} \sigma^L_{n,k}(a,b;q) p_k(x;a,b;q),
  \qquad
  p_n(x;a,b;q) = \sum_{k=0}^{n} \nu^L_{n,k}(a,b;q) x^k.
\]

The following lemma will be used to find formulas for
\( \sigma^L_{n,k}(a,b;q) \) and \( \nu^L_{n,k}(a,b;q) \).

\begin{lem}\label{lem:DAD-1}
  Let \( \{z_i\}_{i\ge0} \) be a sequence of nonzero quantities. If
  the inverse of a matrix \( (a_{i,j})_{i,j=0}^\infty \) is
  \( (b_{i,j})_{i,j=0}^\infty \), then the inverse of the matrix
  \( (a_{i,j}z_i/z_j)_{i,j=0}^\infty \) is \( (b_{i,j}z_i/z_j)_{i,j=0}^\infty \).
\end{lem}
\begin{proof}
  Let \( D \) be the diagonal matrix with diagonal entries \( z_i \).
  Then
  \[
    \left( (a_{i,j}z_i/z_j)_{i,j=0}^\infty \right)^{-1} =
    \left( D (a_{i,j})_{i,j=0}^\infty D^{-1} \right)^{-1}
   = D (b_{i,j})_{i,j=0}^\infty D^{-1} = (b_{i,j}z_i/z_j)_{i,j=0}^\infty.
   \qedhere
  \]
\end{proof}

The following result has been proved in \cite[Lemma~2.5]{LHT} using
the \( q \)-Saalsch\"utz summation formula. We provide another proof
using \Cref{lem:DAD-1} because the same technique will be used later
for big \( q \)-Jacobi polynomials and Askey--Wilson polynomials.

\begin{prop} \label{pro:little-formula}
  We have
\begin{align}
  \label{eq:mu}
  \sigma^L_{n,k}(a,b;q) &=\qbinom{n}{k} \frac{(aq^{k+1};q)_{n-k}}{(abq^{2k+2};q)_{n-k}},\\
  \label{eq:nu}
  \nu^L_{n,k}(a,b;q) &= (-1)^{n-k} q^{\binom{n-k}2} \qbinom{n}{k} \frac{(aq^{k+1};q)_{n-k}}{(abq^{n+k+1};q)_{n-k}}.
\end{align}
Equivalently,
\begin{align}
  \label{eq:mu^L=mu^C}
  \sigma^L_{n,k}(a,b;q)
  &= (aq^{k+1};q)_{n-k} \sigma^b_{n,k}(-abq;q),\\
  \label{eq:nu^L=nu^C}
  \nu^L_{n,k}(a,b;q)
  &= (aq^{k+1};q)_{n-k} \nu^b_{n,k}(-abq;q),
\end{align}
where \( \sigma^b_{n,k}(a;q) \) and \( \nu^b_{n,k}(a;q) \) are the
mixed moments and the coefficients of the \( q \)-Bessel polynomials
in \eqref{eq:mu^C} and \eqref{eq:nu^C}.
\end{prop}
\begin{proof}
  By \eqref{eq:def-lqj} and \eqref{eq:nu^C}, we have \eqref{eq:nu} and
  \eqref{eq:nu^L=nu^C}. The equivalence of \eqref{eq:mu} and
  \eqref{eq:mu^L=mu^C} follows from \eqref{eq:mu^C}. Hence, we only
  need to prove \eqref{eq:mu^L=mu^C}. Since
  \( (\sigma^b_{n,k}(-abq;q))_{n,k} \) and
  \( (\nu^b_{n,k}(-abq;q))_{n,k} \) are inverses of each other, we
  obtain \eqref{eq:mu^L=mu^C} by \Cref{lem:DAD-1} and
  \eqref{eq:nu^L=nu^C}.
\end{proof}

\begin{figure}
  \centering
  \begin{tikzpicture}[scale=1.5]
    \LHLLL{3}4
    \begin{scope}[shift={(-.5,.15)}]
      \node at (1,0) {$1$};
      \node at (2,0) {$1$};
      \node at (2,1/2) {$q$};
      \node at (3,0) {$1$};
      \node at (3,1/3) {$q$};
      \node at (3,2/3) {$q^2$};
    \end{scope}
    \begin{scope}[shift={(-.5,1.15)}]
      \node at (1,0) {$-aq$};
      \node at (2,0) {$-aq^2$};
      \node at (2,1/2) {$-aq^3$};
      \node at (3,0) {$-aq^3$};
      \node at (3,1/3) {$-aq^4$};
      \node at (3,2/3) {$-aq^5$};
    \end{scope}
    \begin{scope}[shift={(-.5,2.15)}]
      \node at (1,0) {$abq^2$};
      \node at (2,0) {$abq^4$};
      \node at (2,1/2) {$abq^5$};
      \node at (3,0) {$abq^6$};
      \node at (3,1/3) {$abq^7$};
      \node at (3,2/3) {$abq^8$};
    \end{scope}
    \begin{scope}[shift={(-.5,3.15)}]
      \node at (1,0) {$-a^2bq^3$};
      \node at (2,0) {$-a^2bq^6$};
      \node at (2,1/2) {$-a^2bq^7$};
      \node at (3,0) {$-a^2bq^9$};
      \node at (3,1/3) {$-a^2bq^{10}$};
      \node at (3,2/3) {$-a^2bq^{11}$};
    \end{scope}
    \begin{scope}[shift={(-.5,4.15)}]
      \node at (1,0) {$a^2b^2q^4$};
      \node at (2,0) {$a^2b^2q^8$};
      \node at (3,0) {$a^2b^2q^{12}$};
    \end{scope}
    \node at (0.5,4.7) {\( \vdots \)};
    \node at (1.5,4.7) {\( \vdots \)};
    \node at (2.5,4.7) {\( \vdots \)};
    \node at (3.5,0.5) {\( \cdots \)};
    \node at (3.5,1.5) {\( \cdots \)};
    \node at (3.5,2.5) {\( \cdots \)};
    \node at (3.5,3.5) {\( \cdots \)};
  \end{tikzpicture}
  \caption{The weight system for the mixed moments of little
    \( q \)-Jacobi polynomials.}
  \label{fig:LHL-little}
\end{figure}

Now we give a weight system of height \( 1 \) for
\( \sigma^L_{n,k}(a,b;q) \).

\begin{prop}\label{pro:bessel-L-1}
Let \( w_1 \) be the weight system of height \( 1 \) defined by
\[
    w_1(0;i,j) = \frac{q^{j}(1-abq^{i+1})(1-aq^{i+1})}{(1-abq^{i+j+1})(1-abq^{i+j+2})}.
  \]
  Then the mixed moments of the little \( q \)-Jacobi polynomials satisfy
  \[
    \sigma^L_{n,k}(a,b;q) = h^{w_1}_{n,k}.
  \]
\end{prop}
\begin{proof}
  By \Cref{pro:charlier-wt-one-row}, we have
  \[
    \sigma^b_{n,k}(-abq;q) = h^{w'_1}_{n,k},
  \]
  where
  \[
    w'_1(0;i,j) = \frac{q^{j}(1-abq^{i+1})}{(1-abq^{i+j+1})(1-abq^{i+j+2})}.
  \]
  Observe that \( w_1(0;i,j) = (1-aq^{i+1})w'_1(0;i,j) \). Since every path from
  \( (k,\infty) \) to \( (n,0) \) has exactly one east step
  between \( x=i \) and \( x=i+1 \) for each \( i=k,k+1,\dots,n-1 \), we obtain
  \[
    h_{n,k}^{w_1} = (aq^{k+1};q)_{n-k} h_{n,k}^{w_1'} =
    (aq^{k+1};q)_{n-k} \sigma^b_{n,k}(-abq;q) = \sigma^L_{n,k}(a,b;q),
  \]
  where \eqref{eq:mu^L=mu^C} is used for the last equality.
\end{proof}

Using the expanding lemma (\Cref{lem:two-to-one}) we obtain a weight
system of infinite height for \( \sigma^L_{n,k}(a,b;q) \). We note
that this is also discoverable by \Cref{alg:1} using the monomial
ordering \( q<a<b \). See \Cref{fig:LHL-little} for the weight
system in the following proposition.

\begin{prop}\label{pro:wt^L}
    \cite{Corteel2020}
  Let \( w \) be the weight system defined by
    \begin{align*}
    w(t;i,j)
    &= (-a)^{\ceil{t/2}} (-b)^{\flr{t/2}} q^{(i+1)t+j}\\
    & = \left( ab q^{2i+2} \right)^m q^j \times
      \begin{cases}
        1  & \mbox{if \( t=2m \)},\\
        -aq^{i+1}  & \mbox{if \( t=2m+1 \)}.
      \end{cases}
  \end{align*}
 Then the mixed moments of the
  little \( q \)-Jacobi polynomials satisfy
  \[
    \sigma^L_{n,k}(a,b;q) = h^{w}_{n,k}.
  \]
\end{prop}

\begin{proof}
  By the same argument as in the proof of \Cref{pro:bessel-L-1},
  we obtain from \Cref{pro:charlier-wt} and \eqref{eq:mu^L=mu^C} that
  \[
    \sigma^L_{n,k}(a,b;q) = h^{w'}_{n,k},
  \]
 where 
  \[
    w'(t;i,j) = (1-aq^{i+1})(ab q^{2i+2})^t q^j.
  \]
  By \Cref{lem:two-to-one-gen}, we have
  \( h^{w'}_{n,k} = h^{w}_{n,k} \), where
  \[
    w(t;i,j)
    = (abq^{2i+2})^{\flr{t/2}} (-aq^{i+1})^{\chi_o(t)} q^{j}
    = (-a)^{\ceil{t/2}} (-b)^{\flr{t/2}} q^{(i+1)t+j},
  \]
  as desired.
\end{proof}

\begin{remark}
Using \Cref{cor:ht-l}, it is possible to obtain a weight system of any
height for \( \sigma^L_{n,k}(a,b;q) \). The same method also works for
big \( q \)-Jacobi polynomials and Askey--Wilson polynomials. We omit
the details.
\end{remark}

\subsection{Big $q$-Jacobi}

Let \( p_n(x;a,b,c;q) \) be the monic big \( q \)-Jacobi polynomial:
\begin{equation}\label{eq:big-q-jacobi-def}
  p_n(x;a,b,c;q) =
  \frac{(aq,cq;q)_n}{(abq^{n+1};q)_n}
  \qHyper32{q^{-n},abq^{n+1},x}{aq,cq}{q,q}.
\end{equation}
We denote by \( \sigma^B_{n,k}(a,b,c;q) \) and \( \nu^B_{n,k}(a,b,c;q) \)
the mixed moments and the dual mixed moments of the big \( q \)-Jacobi
polynomials, respectively:
\[
  x^n = \sum_{k=0}^{n} \sigma^B_{n,k}(a,b,c;q) p_k(x;a,b,c;q),
  \qquad
  p_n(x;a,b,c;q) = \sum_{k=0}^{n} \nu^B_{n,k}(a,b,c;q) x^k.
\]
We will first consider the factorial mixed moments
\( \ts^B_{n,k}(a,b,c;q) \) and factorial dual mixed moments
\( \tn^B_{n,k}(a,b,c;q) \) defined by
\[
  (x|\vq)^n = \sum_{k=0}^{n} \ts^B_{n,k}(a,b,c;q) p_k(x;a,b,c;q),
  \qquad
  p_n(x;a,b,c;q) = \sum_{k=0}^{n} \tn^B_{n,k}(a,b,c;q) (x|\vq)^k,
\]
where \( \vq=(1,q^{-1},q^{-2},\dots) \) so that
\[
  (x|\vq)^k=(x-1)(x-q^{-1})\cdots(x-q^{-k+1}).
\]
\begin{prop}\label{pro:big-formula}
We have
  \begin{align}
    \label{eq:ts-B}
    \ts^B_{n,k}(a,b,c;q) &=
                           (-1)^{n-k} q^{\binom{k}{2}-\binom{n}{2}} \qbinom{n}{k}
                           \frac{(aq^{k+1};q)_{n-k}(cq^{k+1};q)_{n-k}}{(abq^{2k+2};q)_{n-k}},\\
    \label{eq:tn-B}
    \tn^B_{n,k}(a,b,c;q) &=
                           q^{k(k-n)} \qbinom{n}{k}
                           \frac{(aq^{k+1},cq^{k+1};q)_{n-k}}{(abq^{n+k+1};q)_{n-k}}.
  \end{align}
  Equivalently,
\begin{align}
  \label{eq:ts^B=mu^L}
  \ts^B_{n,k}(a,b,c;q)
  &=(-1)^{n-k} q^{\binom{k}{2}-\binom{n}{2}} (cq^{k+1};q)_{n-k}
    \sigma_{n,k}^L(a,b;q),\\
  \label{eq:tn^B=nu^L}
  \tn^B_{n,k}(a,b,c;q)
  &=(-1)^{n-k} q^{\binom{k}{2}-\binom{n}{2}} (cq^{k+1};q)_{n-k}
    \nu_{n,k}^L(a,b;q),
\end{align}
where \( \sigma_{n,k}^L(a,b;q) \) and \( \nu_{n,k}^L(a,b;q) \) are the
mixed and dual mixed moments of the little \( q \)-Jacobi polynomials
given in \eqref{eq:mu} and \eqref{eq:nu}.
\end{prop}

\begin{proof}
  This can be proved similarly as in the proof of
  \Cref{pro:little-formula}. By \eqref{eq:big-q-jacobi-def} and
  \eqref{eq:nu^L=nu^C}, we obtain \eqref{eq:tn-B} and
  \eqref{eq:tn^B=nu^L}. The equivalence of \eqref{eq:ts-B} and
  \eqref{eq:ts^B=mu^L} follows from \eqref{eq:mu^L=mu^C}. Hence, we
  only need to prove \eqref{eq:ts^B=mu^L}. Since
  \( (\sigma^L_{n,k}(a,b;q))_{n,k} \) and
  \( (\nu^L_{n,k}(a,b;q))_{n,k} \) are inverses of each other, we
  obtain \eqref{eq:ts^B=mu^L} by \Cref{lem:DAD-1} and
  \eqref{eq:tn^B=nu^L}.
\end{proof}

We give a weight system of height \( 1 \) and a weight system of
infinite height for \( \ts^B_{n,k}(a,b,c;q) \) in the next two
propositions.

\begin{prop}\label{pro:L-B-1}
Let \( \widetilde{w}_1 \) be the weight system of height \( 1 \) defined by
\[
    \widetilde{w}_1(0;i,j) = \frac{-q^{j-i}(1-abq^{i+1})(1-aq^{i+1})(1-cq^{i+1})}{(1-abq^{i+j+1})(1-abq^{i+j+2})}.
  \]
  Then the factorial mixed moments of the big \( q \)-Jacobi polynomials satisfy
  \[
    \ts^B_{n,k}(a,b,c;q) = h^{\widetilde{w}_1}_{n,k}.
  \]
\end{prop}
\begin{proof}
  This can be proved similarly as in the proof of
  \Cref{pro:bessel-L-1} using that proposition and
  \eqref{eq:ts^B=mu^L}.
\end{proof}

See \Cref{fig:fac-big-q-lht} for the weight system in the following
proposition.

\begin{prop}\label{pro:tilde-wt^B}
  Let
  \begin{align*}
    \widetilde{w}(t;i,j)
    &= - (-a)^{\flr{(t+2)/4}} (-b)^{\flr{t/4}}(-c)^{\chi_o(t)} q^{(i+1)\ceil{t/2}-i+j}\\
    &= (abq^{2i+2})^m q^j \times
      \begin{cases}
         -q^{-i}  & \mbox{if \( t=4m \)},\\
         cq  & \mbox{if \( t=4m+1 \)},\\
         aq  & \mbox{if \( t=4m+2 \)},\\
         -acq^{i+2}  & \mbox{if \(t=4m+3 \)}.
       \end{cases}
  \end{align*}
  Then the factorial mixed moments of the big \( q \)-Jacobi polynomials satisfy
  \[
    \ts^B_{n,k}(a,b,c;q) = h^{\widetilde{w}}_{n,k}.
  \]
\end{prop}

\begin{figure}
  \centering
  \begin{tikzpicture}[scale=1.5]
    \LHLLL{3}4
    \begin{scope}[shift={(-.5,.15)}]
      \node at (1,0) {$-1$};
      \node at (2,0) {$-q^{-1}$};
      \node at (2,1/2) {$-1$};
      \node at (3,0) {$-q^{-2}$};
      \node at (3,1/3) {$-q^{-1}$};
      \node at (3,2/3) {$-1$};
    \end{scope}
    \begin{scope}[shift={(-.5,1.15)}]
      \node at (1,0) {$cq$};
      \node at (2,0) {$cq$};
      \node at (2,1/2) {$cq^2$};
      \node at (3,0) {$cq$};
      \node at (3,1/3) {$cq^2$};
      \node at (3,2/3) {$cq^3$};
    \end{scope}
    \begin{scope}[shift={(-.5,2.15)}]
      \node at (1,0) {$aq$};
      \node at (2,0) {$aq$};
      \node at (2,1/2) {$aq^2$};
      \node at (3,0) {$aq$};
      \node at (3,1/3) {$aq^2$};
      \node at (3,2/3) {$aq^3$};
    \end{scope}
    \begin{scope}[shift={(-.5,3.15)}]
      \node at (1,0) {$-acq^2$};
      \node at (2,0) {$-acq^3$};
      \node at (2,1/2) {$-acq^4$};
      \node at (3,0) {$-acq^4$};
      \node at (3,1/3) {$-acq^5$};
      \node at (3,2/3) {$-acq^6$};
    \end{scope}
    \begin{scope}[shift={(-.5,4.15)}]
      \node at (1,0) {$-abq^2$};
      \node at (2,0) {$-abq^3$};
      \node at (3,0) {$-abq^4$};
    \end{scope}
    \node at (0.5,4.7) {\( \vdots \)};
    \node at (1.5,4.7) {\( \vdots \)};
    \node at (2.5,4.7) {\( \vdots \)};
    \node at (3.5,0.5) {\( \cdots \)};
    \node at (3.5,1.5) {\( \cdots \)};
    \node at (3.5,2.5) {\( \cdots \)};
    \node at (3.5,3.5) {\( \cdots \)};
  \end{tikzpicture}
  \caption{The lecture hall graph model for the factorial mixed
    moments of big \( q \)-Jacobi polynomials.}
  \label{fig:fac-big-q-lht}
\end{figure}

\begin{proof}
  By \Cref{pro:wt^L} and \eqref{eq:ts^B=mu^L}
  we have
  \[
    \ts^B_{n,k}(a,b,c;q) = h^{w'}_{n,k},
  \]
  where
  \[
    w'(t;i,j) = - q^{-i} (1-cq^{i+1}) \cdot (-a)^{\ceil{t/2}}
    (-b)^{\flr{t/2}} q^{(i+1)t+j}.
  \]
  By \Cref{lem:two-to-one-gen}, we can replace the weight system
  \( w' \) by \( \widetilde{w} \), where
  \begin{align*}
    \widetilde{w}(t;i,j)
    &= - (-cq^{i+1})^{\chi_o(t)} (-a)^{\ceil{\flr{t/2}/2}} (-b)^{\flr{\flr{t/2}/2}} q^{(i+1)\flr{t/2}-i+j}\\
    &= - (-a)^{\flr{(t+2)/4}} (-b)^{\flr{t/4}}(-c)^{\chi_o(t)} q^{(i+1)\ceil{t/2}-i+j},
  \end{align*}
  as desired.
\end{proof}

By \Cref{pro:add-one-row}, we can construct a weight system for
\( \sigma^B_{n,k}(a,b,c;q) \) from the weight system
\( \widetilde{w} \) in \Cref{pro:tilde-wt^B} by adding one row coming
from the factorial basis \( \{(x|\vq)^n\}_{n\ge0} \) at the bottom.
Miraculously, the new row and the bottom row of \( \widetilde{w} \)
cancel each other and the resulting weight system is
\( \widetilde{w} \) with bottom row removed as described in the next
proposition. We note that this weight system is also discoverable by
\Cref{alg:1} using the monomial ordering \( q<c<a<b \). See
\Cref{fig:big-q-lht} for the weight system in the following
proposition.

\begin{prop}\label{pro:wt^B}
  Let
  \begin{align*}
    w(t;i,j)
    &= - (-a)^{\flr{(t+3)/4}} (-b)^{\flr{(t+1)/4}}(-c)^{\chi_e(t)} q^{(i+1)\flr{t/2}+j+1}\\
    &= (abq^{2i+2})^m q^j \times
      \begin{cases}
        cq      & \mbox{if \( t=4m \)},\\
        aq      & \mbox{if \( t=4m+1 \)},\\
        -acq^{i+2}  & \mbox{if \( t=4m+2 \)},\\
        -abq^{i+2}  & \mbox{if \(t=4m+3 \)}.
      \end{cases}
  \end{align*}
  Then the mixed moments of the big \( q \)-Jacobi polynomials satisfy
  \[
    \sigma^B_{n,k}(a,b,c;q) = h^{w}_{n,k}.
  \]
\end{prop}

\begin{figure}
  \centering
  \begin{tikzpicture}[scale=1.5]
    \LHLLL{3}4
    \begin{scope}[shift={(-.5,.15)}]
      \node at (1,0) {$cq$};
      \node at (2,0) {$cq$};
      \node at (2,1/2) {$cq^2$};
      \node at (3,0) {$cq$};
      \node at (3,1/3) {$cq^2$};
      \node at (3,2/3) {$cq^3$};
    \end{scope}
    \begin{scope}[shift={(-.5,1.15)}]
      \node at (1,0) {$aq$};
      \node at (2,0) {$aq$};
      \node at (2,1/2) {$aq^2$};
      \node at (3,0) {$aq$};
      \node at (3,1/3) {$aq^2$};
      \node at (3,2/3) {$aq^3$};
    \end{scope}
    \begin{scope}[shift={(-.5,2.15)}]
      \node at (1,0) {$-acq^2$};
      \node at (2,0) {$-acq^3$};
      \node at (2,1/2) {$-acq^4$};
      \node at (3,0) {$-acq^4$};
      \node at (3,1/3) {$-acq^5$};
      \node at (3,2/3) {$-acq^6$};
    \end{scope}
    \begin{scope}[shift={(-.5,3.15)}]
      \node at (1,0) {$-abq^2$};
      \node at (2,0) {$-abq^3$};
      \node at (2,1/2) {$-abq^4$};
      \node at (3,0) {$-abq^4$};
      \node at (3,1/3) {$-abq^5$};
      \node at (3,2/3) {$-abq^6$};
    \end{scope}
    \begin{scope}[shift={(-.5,4.15)}]
      \node at (1,0) {$abcq^3$};
      \node at (2,0) {$abcq^5$};
      \node at (3,0) {$abcq^7$};
    \end{scope}
    \node at (0.5,4.7) {\( \vdots \)};
    \node at (1.5,4.7) {\( \vdots \)};
    \node at (2.5,4.7) {\( \vdots \)};
    \node at (3.5,0.5) {\( \cdots \)};
    \node at (3.5,1.5) {\( \cdots \)};
    \node at (3.5,2.5) {\( \cdots \)};
    \node at (3.5,3.5) {\( \cdots \)};
  \end{tikzpicture}
  \caption{The weight system for the mixed moments of big \( q \)-Jacobi polynomials.}
  \label{fig:big-q-lht}
\end{figure}

\begin{proof}
  Let \( w' \) be the weight system defined by
  \[
    w'(t;i,j) =
    \begin{cases}
     \widetilde{w}(t-1;i,j) & \mbox{if \( t\ge1 \)},\\
     q^{-j} & \mbox{if \( t=0 \),}
    \end{cases}
  \]
  where \( \widetilde{w} \) is the weight system in
  \Cref{pro:tilde-wt^B}. In other words, \( w' \) is the weight system
  obtained from \( \widetilde{w} \) by adding the weight system
  \( w_1 \) of height \( 1 \) defined by \( w_1(0;i,j) = q^{-j} \) at
  the bottom as shown in \Cref{fig:fac-big-q-lht-2}. By
  \Cref{pro:add-one-row} and \Cref{pro:tilde-wt^B}, we have
  \(\sigma^B_{n,k}(a,b,c;q) = h^{w'}_{n,k} \). Thus it suffices to
  show that \( h^{w'}_{n,k} = h^{w}_{n,k} \).

  \begin{figure}
  \centering
  \begin{tikzpicture}[scale=1.5]
    \LHLLL{3}4
    \begin{scope}[shift={(-.5,.15)}]
      \node at (1,0) {$1$};
      \node at (2,0) {$1$};
      \node at (2,1/2) {$q^{-1}$};
      \node at (3,0) {$1$};
      \node at (3,1/3) {$q^{-1}$};
      \node at (3,2/3) {$q^{-2}$};
    \end{scope}
    \begin{scope}[shift={(-.5,1.15)}]
      \node at (1,0) {$-1$};
      \node at (2,0) {$-q^{-1}$};
      \node at (2,1/2) {$-1$};
      \node at (3,0) {$-q^{-2}$};
      \node at (3,1/3) {$-q^{-1}$};
      \node at (3,2/3) {$-1$};
    \end{scope}
    \begin{scope}[shift={(-.5,2.15)}]
      \node at (1,0) {$cq$};
      \node at (2,0) {$cq$};
      \node at (2,1/2) {$cq^2$};
      \node at (3,0) {$cq$};
      \node at (3,1/3) {$cq^2$};
      \node at (3,2/3) {$cq^3$};
    \end{scope}
    \begin{scope}[shift={(-.5,3.15)}]
      \node at (1,0) {$aq$};
      \node at (2,0) {$aq$};
      \node at (2,1/2) {$aq^2$};
      \node at (3,0) {$aq$};
      \node at (3,1/3) {$aq^2$};
      \node at (3,2/3) {$aq^3$};
    \end{scope}
    \begin{scope}[shift={(-.5,4.15)}]
      \node at (1,0) {$-acq^2$};
      \node at (2,0) {$-acq^3$};
      \node at (3,0) {$-acq^4$};
    \end{scope}
    \node at (0.5,4.7) {\( \vdots \)};
    \node at (1.5,4.7) {\( \vdots \)};
    \node at (2.5,4.7) {\( \vdots \)};
    \node at (3.5,0.5) {\( \cdots \)};
    \node at (3.5,1.5) {\( \cdots \)};
    \node at (3.5,2.5) {\( \cdots \)};
    \node at (3.5,3.5) {\( \cdots \)};
  \end{tikzpicture}
  \caption{The weight system \( w' \) in the proof of
    \Cref{pro:wt^B}.}
  \label{fig:fac-big-q-lht-2}
\end{figure}

Since \( \ceil{(t+1)/2}=\flr{t/2}+1 \), comparing the definitions of
\( w(t;i,j) \) and \( \widetilde{w}(t;i,j) \), we have
\( w(t;i,j) = \widetilde{w}(t+1;i,j) = w'(t+2;i,j) \) for all
\( t\ge0 \). Thus, by \Cref{lem:rec2} with \( \ell=2 \), we have
\[
  h^{w'}_{n,k} = \sum_{r=k}^n h^{w'_2}_{n,r} h^w_{r,k},
\]
where \( w'_2 \) is the weight system of height \( 2 \)
given by 
\[
  w'_2(t;i,j) =
  \begin{cases}
    w'(t;i,j) & \mbox{if \( t<2 \)},\\
    0 & \mbox{if \( t\ge 2 \)}.
  \end{cases}
\]
Therefore, to prove \( h^{w'}_{n,k} = h^{w}_{n,k} \), it suffices to
show that \( h^{w'_2}_{n,r} = \delta_{n,r} \).

Let \( \overline{w_1} \) be the weight system of height \( 1 \)
defined by \( \overline{w_1}(0;i,j) = w_1(0;i,i-j) \). Observe that
\( w'_2(0;i,j) = q^{-j} = w_1(0;i,j) \) and
\( w'_2(1;i,j) = -q^{j-i} = -w_1(0;i,i-j) = - \overline{w_1}(0;i,j)
\). By \Cref{lem:rec2} with \( \ell=1 \) and \Cref{lem:wt-mult}, we
have
\[
  h^{w'_2}_{n,r} = \sum_{m=r}^n h^{w_1}_{n,m} (-1)^{m-r} h^{\overline{w_1}}_{m,r}.
\]
Then, by \Cref{pro:w=overline w} and \Cref{lem:dual-wt}, we obtain
\[
  h^{w'_2}_{n,r} = \sum_{m=r}^n h^{w_1}_{n,m} (-1)^{m-r} e^{w_1}_{m,r} = \delta_{n,r},
\]
which completes the proof.
\end{proof}

\subsection{Askey--Wilson polynomials}

The monic \emph{Askey--Wilson polynomials}
\( p_n(x;a,b,c,d|q) \) are defined by
\begin{equation}\label{eq:AW-def}
  p_n(x;a,b,c,d|q) =
  \frac{(ab,ac,ad;q)_n}{2^na^n (abcdq^{n-1};q)_n}
  \qHyper43{q^{-n},abcdq^{n-1},ae^{i\theta}, ae^{-i\theta}}{ab,ac,ad}{q,q},
\end{equation}
where \( x=\cos \theta = (e^{i\theta}+e^{-i\theta})/2 \). Let
\( f=(f_0,f_1, \dots) \) be the sequence given by
\( f_j = (aq^{j}+a^{-1}q^{-j})/2 \). Then it is easy to check that
\[
  (ae^{i\theta};q)_k (ae^{-i\theta};q)_k = (-2a)^k q^{\binom{k}{2}} (x|f)^k.
\]

We denote by \( \sigma^{AW}_{n,k}(a,b,c,d;q) \) and
\( \nu^{AW}_{n,k}(a,b,c,d;q) \) the mixed moments and the dual mixed
moments of the Askey--Wilson polynomials, respectively:
\[
  x^n = \sum_{k=0}^{n} \sigma^{AW}_{n,k}(a,b,c,d;q) p_k(x;a,b,c,d|q),
  \qquad
  p_n(x;a,b,c,d|q) = \sum_{k=0}^{n} \nu^{AW}_{n,k}(a,b,c,d;q) x^k.
\]
As in the case of big \( q \)-Jacobi polynomials, we will first
consider the factorial mixed moments \( \ts^{AW}_{n,k}(a,b,c,d;q) \)
and factorial dual mixed moments \( \tn^{AW}_{n,k}(a,b,c,d;q) \)
defined by
\[
  (x|f)^n = \sum_{k=0}^{n} \ts^{AW}_{n,k}(a,b,c,d;q) p_k(x;a,b,c,d|q),
  \qquad
  p_n(x;a,b,c,d|q) = \sum_{k=0}^{n} \tn^{AW}_{n,k}(a,b,c,d;q) (x|f)^k.
\]

\begin{prop}\label{prop:ts_tn-AW}
  We have
  \begin{align}
    \label{eq:ts-AW}
    \ts^{AW}_{n,k}(a,b,c,d;q)
    &= (-2a)^{k-n} q^{\binom{k}{2}-\binom{n}{2}} \qbinom{n}{k}
      \frac{(abq^{k},acq^{k},adq^{k};q)_{n-k}}{(abcdq^{2k};q)_{n-k}},\\
    \label{eq:tn-AW}
    \tn^{AW}_{n,k}(a,b,c,d;q)
    &= (2a)^{k-n}q^{k(k-n)} \qbinom{n}{k}
      \frac{(abq^{k},acq^{k},adq^{k};q)_{n-k}}{(abcdq^{n+k-1};q)_{n-k}}.
  \end{align}
  Equivalently,
  \begin{align}
    \label{eq:ts-AW1}
    \ts^{AW}_{n,k}(a,b,c,d;q)
    &=(2a)^{k-n} (adq^{k};q)_{n-k} \ts_{n,k}^B(ab/q,cd/q,ac/q;q),\\
    \label{eq:tn-AW1}
    \tn^{AW}_{n,k}(a,b,c,d;q)
    &=(2a)^{k-n}  (adq^{k};q)_{n-k} \tn_{n,k}^B(ab/q,cd/q,ac/q;q),
  \end{align}
  where \( \ts_{n,k}^B(ab/q,cd/q,ac/q) \) and
  \( \tn_{n,k}^B(ab/q,cd/q,ac/q) \) are the factorial mixed and dual
  mixed moments of the big \( q \)-Jacobi polynomials in
  \eqref{eq:ts-B} and \eqref{eq:tn-B}.
\end{prop}

\begin{proof}
This can be proved in a similar way to \Cref{pro:big-formula}.
\end{proof}

As in the previous subsections, we give a weight system of height
\( 1 \) for \( \ts^{AW}_{n,k}(a,b,c,d;q) \).

\begin{prop}\label{prop:wt_1floor-AW}
  Let
  \[
    \widetilde{w}_1(0;i,j) =  \frac{q^{j-i}(1-abq^{i})(1-acq^{i})(1-adq^{i})(1-abcdq^{i-1})}
    {-2a(1-abcdq^{i+j-1})(1-abcdq^{i+j})}.
  \]
  Then the factorial mixed moments of the Askey--Wilson polynomials satisfy
  \[
    \ts^{AW}_{n,k}(a,b,c,d;q) = h^{\widetilde{w}_1}_{n,k}.
  \]
\end{prop}
\begin{proof}
  This can be proved in a similar way to \Cref{pro:L-B-1} using
  that proposition and \eqref{eq:ts-AW1}.
\end{proof}

\begin{figure}
  \centering
  \begin{tikzpicture}[scale=1.5]
    \LHLLL{4}8
    \small
    \begin{scope}[shift={(-.5,.15)}]
      \node at (1,0/1) {$- 1/(2a)$};
      \node at (1,1/1) {$ d/2$};
      \node at (1,2/1) {$ c/2$};
      \node at (1,3/1) {$- a c d/2$};
      \node at (1,4/1) {$ b/2$};
      \node at (1,5/1) {$- a b d/2$};
      \node at (1,6/1) {$- a b c/2$};
      \node at (1,7/1) {$ a^{2} b c d/2$};
      \node at (1,8/1) {$- b c d/2$};
      \node at (2,0/2) {$- 1/(2aq)$};
      \node at (2,1/2) {$- 1/(2a)$};
      \node at (2,2/2) {$ d/2$};
      \node at (2,3/2) {$ d q/2$};
      \node at (2,4/2) {$ c/2$};
      \node at (2,5/2) {$ c q/2$};
      \node at (2,6/2) {$- a c d q/2$};
      \node at (2,7/2) {$- a c d q^{2}/2$};
      \node at (2,8/2) {$ b/2$};
      \node at (2,9/2) {$ b q/2$};
      \node at (2,10/2) {$- a b d q/2$};
      \node at (2,11/2) {$- a b d q^{2}/2$};
      \node at (2,12/2) {$- a b c q/2$};
      \node at (2,13/2) {$- a b c q^{2}/2$};
      \node at (2,14/2) {$ a^{2} b c d q^{2}/2$};
      \node at (2,15/2) {$ a^{2} b c d q^{3}/2$};
      \node at (2,16/2) {$- b c d q/2$};
      \node at (3,0/3) {$- 1/(2aq^2)$};
      \node at (3,1/3) {$- 1/(2aq)$};
      \node at (3,2/3) {$- 1/(2a)$};
      \node at (3,3/3) {$ d/2$};
      \node at (3,4/3) {$ d q/2$};
      \node at (3,5/3) {$ d q^{2}/2$};
      \node at (3,6/3) {$ c/2$};
      \node at (3,7/3) {$ c q/2$};
      \node at (3,8/3) {$ c q^{2}/2$};
      \node at (3,9/3) {$- a c d q^{2}/2$};
      \node at (3,10/3) {$- a c d q^{3}/2$};
      \node at (3,11/3) {$- a c d q^{4}/2$};
      \node at (3,12/3) {$ b/2$};
      \node at (3,13/3) {$ b q/2$};
      \node at (3,14/3) {$ b q^{2}/2$};
      \node at (3,15/3) {$- a b d q^{2}/2$};
      \node at (3,16/3) {$- a b d q^{3}/2$};
      \node at (3,17/3) {$- a b d q^{4}/2$};
      \node at (3,18/3) {$- a b c q^{2}/2$};
      \node at (3,19/3) {$- a b c q^{3}/2$};
      \node at (3,20/3) {$- a b c q^{4}/2$};
      \node at (3,21/3) {$ a^{2} b c d q^{4}/2$};
      \node at (3,22/3) {$ a^{2} b c d q^{5}/2$};
      \node at (3,23/3) {$ a^{2} b c d q^{6}/2$};
      \node at (3,24/3) {$- b c d q^{2}/2$};
      \node at (4,0/4) {$- 1/(2aq^3)$};
      \node at (4,1/4) {$- 1/(2aq^2)$};
      \node at (4,2/4) {$- 1/(2aq)$};
      \node at (4,3/4) {$- 1/(2a)$};
      \node at (4,4/4) {$ d/2$};
      \node at (4,5/4) {$ d q/2$};
      \node at (4,6/4) {$ d q^{2}/2$};
      \node at (4,7/4) {$ d q^{3}/2$};
      \node at (4,8/4) {$ c/2$};
      \node at (4,9/4) {$ c q/2$};
      \node at (4,10/4) {$ c q^{2}/2$};
      \node at (4,11/4) {$ c q^{3}/2$};
      \node at (4,12/4) {$- a c d q^{3}/2$};
      \node at (4,13/4) {$- a c d q^{4}/2$};
      \node at (4,14/4) {$- a c d q^{5}/2$};
      \node at (4,15/4) {$- a c d q^{6}/2$};
      \node at (4,16/4) {$ b/2$};
      \node at (4,17/4) {$ b q/2$};
      \node at (4,18/4) {$ b q^{2}/2$};
      \node at (4,19/4) {$ b q^{3}/2$};
      \node at (4,20/4) {$- a b d q^{3}/2$};
      \node at (4,21/4) {$- a b d q^{4}/2$};
      \node at (4,22/4) {$- a b d q^{5}/2$};
      \node at (4,23/4) {$- a b d q^{6}/2$};
      \node at (4,24/4) {$- a b c q^{3}/2$};
      \node at (4,25/4) {$- a b c q^{4}/2$};
      \node at (4,26/4) {$- a b c q^{5}/2$};
      \node at (4,27/4) {$- a b c q^{6}/2$};
      \node at (4,28/4) {$ a^{2} b c d q^{6}/2$};
      \node at (4,29/4) {$ a^{2} b c d q^{7}/2$};
      \node at (4,30/4) {$ a^{2} b c d q^{8}/2$};
      \node at (4,31/4) {$ a^{2} b c d q^{9}/2$};
      \node at (4,32/4) {$- b c d q^{3}/2$};
    \end{scope}

    \node at (0.5,8.7) {\( \vdots \)};
    \node at (1.5,8.7) {\( \vdots \)};
    \node at (2.5,8.7) {\( \vdots \)};
    \node at (3.5,8.7) {\( \vdots \)};
    \node at (4.5,0.5) {\( \cdots \)};
    \node at (4.5,1.5) {\( \cdots \)};
    \node at (4.5,2.5) {\( \cdots \)};
    \node at (4.5,3.5) {\( \cdots \)};
    \node at (4.5,4.5) {\( \cdots \)};
    \node at (4.5,5.5) {\( \cdots \)};
    \node at (4.5,6.5) {\( \cdots \)};
    \node at (4.5,7.5) {\( \cdots \)};
  \end{tikzpicture}
  \caption{The lecture hall graph for the factorial moments of
    Askey--Wilson polynomials. }
  \label{fig:AW-lht-tilde}
\end{figure}

Now we give a weight system of infinite height for
\( \ts^{AW}_{n,k}(a,b,c,d;q) \). This is also discoverable by
\Cref{alg:1} using the monomial ordering \( q<d<c<a<b \). See
\Cref{fig:AW-lht-tilde} for the weight system in the following
proposition.

\begin{prop}\label{pro:AW-tw}
  Let
  \begin{align*}
    \widetilde{w}(t;i,j) &= -(2a)^{-1}q^{(i+1)\flr{(t+2)/4}-i+j}\\
    & \qquad \times (-ab/q)^{\flr{(t+4)/8}} (-cd/q)^{\flr{t/8}}
      (-ac/q)^{\chi_o(\flr{t/2})}(-adq^{i})^{\chi_o(t)}\\
    &=  \frac{1}{2} \left( abcd q^{2i} \right)^m q^j \times
         \begin{cases}
         - a^{-1} q^{-i} & \mbox{if \( t=8m \)},\\
         d & \mbox{if \( t=8m+1 \)},\\
         c & \mbox{if \( t=8m+2 \)},\\
         -acd  q^{i} & \mbox{if \( t=8m+3 \)},\\
         b & \mbox{if \( t=8m+4 \)},\\
         - abd q^{i} & \mbox{if \( t=8m+5 \)},\\
         - abc q^{i} & \mbox{if \( t=8m+6 \)},\\
         a^2bcd  q^{2i} & \mbox{if \( t=8m+7 \)}.
       \end{cases}
  \end{align*}
  Then the factorial mixed moments of the Askey--Wilson polynomials satisfy
  \[
    \ts^{AW}_{n,k}(a,b,c,d;q) = h^{\widetilde{w}}_{n,k}.
  \]
\end{prop}

\begin{proof}
  By \eqref{eq:ts-AW1} and \Cref{pro:tilde-wt^B}, we have
  \[
    \ts^{AW}_{n,k}(a,b,c,d;q) = h^{\widetilde{w}'}_{n,k},
  \]
  where \( w' \) is the weight system given by
 \[
    \widetilde{w}'(t;i,j) = (2a)^{-1}(1-adq^i) \cdot (-1)
    (-ab/q)^{\flr{(t+2)/4}} (-cd/q)^{\flr{t/4}}(-ac/q)^{\chi_o(t)}
    q^{(i+1)\ceil{t/2}-i+j}.
  \] 
  By \Cref{lem:two-to-one-gen},
  \( h^{\widetilde{w}'}_{n,k} = h^{\widetilde{w}}_{n,k} \), where
  \( \widetilde{w} \) is the weight system given by
  \begin{align*}
    &\widetilde{w}(t;i,j)\\
    &= -(2a)^{-1}(-adq^i)^{\chi_o(t)} 
    (-ab/q)^{\flr{(\flr{t/2}+2)/4}} (-cd/q)^{\flr{\flr{t/2}/4}}(-ac/q)^{\chi_o(\flr{t/2})}
    q^{(i+1)\ceil{\flr{t/2}/2}-i+j}\\
    &= -(2a)^{-1}(-adq^i)^{\chi_o(t)} 
    (-ab/q)^{\flr{(t+4)/8}} (-cd/q)^{\flr{t/8}}(-ac/q)^{\chi_o(\flr{t/2})}
    q^{(i+1)\flr{(t+2)/4}-i+j},
  \end{align*}
  as desired.
\end{proof}

Finally, by \Cref{pro:add-one-row} and \Cref{pro:AW-tw}, we obtain a
weight system for the original mixed moments of Askey--Wilson
polynomials. See \Cref{fig:AW-lht} for the weight system in the
following theorem.

\begin{thm}\label{thm:AW-wt-full}
  Let
  \[
    w(t;i,j) =
    \begin{cases}
      (aq^{j}+a^{-1}q^{-j})/2 & \mbox{if \( t=0 \)},\\
      \widetilde{w}(t-1;i,j) & \mbox{if \( t\ge1 \)},
    \end{cases}
  \]
  where \( \widetilde{w} \) is the weight system given in
  \Cref{pro:AW-tw}.  Then the mixed moments of
  the Askey--Wilson polynomials satisfy
  \[
    \sigma^{AW}_{n,k}(a,b,c,d;q) = h^{w}_{n,k}.
  \]
\end{thm}

\begin{figure}
  \centering
  \begin{tikzpicture}[scale=1.5]
    \LHLLL{4}{9}
    \small
\begin{scope}[shift={(-.5,1.15)}]
      \node at (1,0/1) {$- 1/(2a)$};
      \node at (1,1/1) {$ d/2$};
      \node at (1,2/1) {$ c/2$};
      \node at (1,3/1) {$- a c d/2$};
      \node at (1,4/1) {$ b/2$};
      \node at (1,5/1) {$- a b d/2$};
      \node at (1,6/1) {$- a b c/2$};
      \node at (1,7/1) {$ a^{2} b c d/2$};
      \node at (1,8/1) {$- b c d/2$};
      \node at (2,0/2) {$- 1/(2aq)$};
      \node at (2,1/2) {$- 1/(2a)$};
      \node at (2,2/2) {$ d/2$};
      \node at (2,3/2) {$ d q/2$};
      \node at (2,4/2) {$ c/2$};
      \node at (2,5/2) {$ c q/2$};
      \node at (2,6/2) {$- a c d q/2$};
      \node at (2,7/2) {$- a c d q^{2}/2$};
      \node at (2,8/2) {$ b/2$};
      \node at (2,9/2) {$ b q/2$};
      \node at (2,10/2) {$- a b d q/2$};
      \node at (2,11/2) {$- a b d q^{2}/2$};
      \node at (2,12/2) {$- a b c q/2$};
      \node at (2,13/2) {$- a b c q^{2}/2$};
      \node at (2,14/2) {$ a^{2} b c d q^{2}/2$};
      \node at (2,15/2) {$ a^{2} b c d q^{3}/2$};
      \node at (2,16/2) {$- b c d q/2$};
      \node at (3,0/3) {$- 1/(2aq^2)$};
      \node at (3,1/3) {$- 1/(2aq)$};
      \node at (3,2/3) {$- 1/(2a)$};
      \node at (3,3/3) {$ d/2$};
      \node at (3,4/3) {$ d q/2$};
      \node at (3,5/3) {$ d q^{2}/2$};
      \node at (3,6/3) {$ c/2$};
      \node at (3,7/3) {$ c q/2$};
      \node at (3,8/3) {$ c q^{2}/2$};
      \node at (3,9/3) {$- a c d q^{2}/2$};
      \node at (3,10/3) {$- a c d q^{3}/2$};
      \node at (3,11/3) {$- a c d q^{4}/2$};
      \node at (3,12/3) {$ b/2$};
      \node at (3,13/3) {$ b q/2$};
      \node at (3,14/3) {$ b q^{2}/2$};
      \node at (3,15/3) {$- a b d q^{2}/2$};
      \node at (3,16/3) {$- a b d q^{3}/2$};
      \node at (3,17/3) {$- a b d q^{4}/2$};
      \node at (3,18/3) {$- a b c q^{2}/2$};
      \node at (3,19/3) {$- a b c q^{3}/2$};
      \node at (3,20/3) {$- a b c q^{4}/2$};
      \node at (3,21/3) {$ a^{2} b c d q^{4}/2$};
      \node at (3,22/3) {$ a^{2} b c d q^{5}/2$};
      \node at (3,23/3) {$ a^{2} b c d q^{6}/2$};
      \node at (3,24/3) {$- b c d q^{2}/2$};
      \node at (4,0/4) {$- 1/(2aq^3)$};
      \node at (4,1/4) {$- 1/(2aq^2)$};
      \node at (4,2/4) {$- 1/(2aq)$};
      \node at (4,3/4) {$- 1/(2a)$};
      \node at (4,4/4) {$ d/2$};
      \node at (4,5/4) {$ d q/2$};
      \node at (4,6/4) {$ d q^{2}/2$};
      \node at (4,7/4) {$ d q^{3}/2$};
      \node at (4,8/4) {$ c/2$};
      \node at (4,9/4) {$ c q/2$};
      \node at (4,10/4) {$ c q^{2}/2$};
      \node at (4,11/4) {$ c q^{3}/2$};
      \node at (4,12/4) {$- a c d q^{3}/2$};
      \node at (4,13/4) {$- a c d q^{4}/2$};
      \node at (4,14/4) {$- a c d q^{5}/2$};
      \node at (4,15/4) {$- a c d q^{6}/2$};
      \node at (4,16/4) {$ b/2$};
      \node at (4,17/4) {$ b q/2$};
      \node at (4,18/4) {$ b q^{2}/2$};
      \node at (4,19/4) {$ b q^{3}/2$};
      \node at (4,20/4) {$- a b d q^{3}/2$};
      \node at (4,21/4) {$- a b d q^{4}/2$};
      \node at (4,22/4) {$- a b d q^{5}/2$};
      \node at (4,23/4) {$- a b d q^{6}/2$};
      \node at (4,24/4) {$- a b c q^{3}/2$};
      \node at (4,25/4) {$- a b c q^{4}/2$};
      \node at (4,26/4) {$- a b c q^{5}/2$};
      \node at (4,27/4) {$- a b c q^{6}/2$};
      \node at (4,28/4) {$ a^{2} b c d q^{6}/2$};
      \node at (4,29/4) {$ a^{2} b c d q^{7}/2$};
      \node at (4,30/4) {$ a^{2} b c d q^{8}/2$};
      \node at (4,31/4) {$ a^{2} b c d q^{9}/2$};
      \node at (4,32/4) {$- b c d q^{3}/2$};
    \end{scope}
        \begin{scope}[shift={(-.5,.15)}]
      \node at (1,0/1) {$(a+)/2$};
      \node at (2,0/2) {$(a+)/2$};
      \node at (2,1/2) {$(aq+)/2$};
      \node at (3,0/3) {$(a+)/2$};
      \node at (3,1/3) {$(aq+)/2$};
      \node at (3,2/3) {$(aq^2+)/2$};
      \node at (4,0/4) {$(a+)/2$};
      \node at (4,1/4) {$(aq+)/2$};
      \node at (4,2/4) {$(aq^2+)/2$};
      \node at (4,3/4) {$(aq^3+)/2$};
    \end{scope}
    \node at (0.5,9.7) {\( \vdots \)};
    \node at (1.5,9.7) {\( \vdots \)};
    \node at (2.5,9.7) {\( \vdots \)};
    \node at (3.5,9.7) {\( \vdots \)};
    \node at (4.5,0.5) {\( \cdots \)};
    \node at (4.5,1.5) {\( \cdots \)};
    \node at (4.5,2.5) {\( \cdots \)};
    \node at (4.5,3.5) {\( \cdots \)};
  \end{tikzpicture}
  \caption{The lecture hall graph for the moments of Askey--Wilson
    polynomials. Here, \( (x+) \) means \( x+x^{-1} \).}
  \label{fig:AW-lht}
\end{figure}

\section{Another bootstrapping method from continuous $q$-Hermite polynomials}
\label{sec:anoth-bootstr-meth}

In this section, we provide another bootstrapping method to find a
combinatorial model for mixed moments of Askey--Wilson polynomials
relative to the continuous \( q \)-Hermite polynomials, which are one
of the most well-studied families of orthogonal polynomials in the
\( q \)-Askey scheme. At the end of this section, we give a lecture
hall graph model for the mixed moments of the continuous
\( q \)-Hermite polynomials. As applications of the results in this
section, we prove \Cref{cor:1}, which establishes the total positivity
of the matrix of mixed moments of Askey--Wilson polynomials relative
to continuous \( q \)-Hermite polynomials, and we also give the first
combinatorial proof of the symmetry of \( a,b,c,d \) in the
Askey--Wilson polynomials in \Cref{sec:comb-prop-askey}.

Our results build on those of Kim and Stanton \cite{KS15}, who used a
bootstrapping method to compute the moments of Askey--Wilson
polynomials from continuous \( q \)-Hermite polynomials.

\subsection{Definitions}

The (monic) \emph{continuous $q$-Hermite polynomials}
\( H_n(x|q) \), \emph{continuous big $q$-Hermite polynomials}
\( H_n(x;a|q) \), \emph{Al--Salam--Chihara polynomials}
\( Q_n(x;a,b|q) \), and \emph{continuous dual $q$-Hahn polynomials}
\(p^{dH}_n(x; a, b, c|q)\) are defined by
\begin{align*}
  H_n(x|q)
  &= \frac{e^{in\theta}}{2^n} \qHyper20{q^{-n},0}{-}{q,q^n e^{-2i\theta}},\\
  H_n(x;a|q)
  &= \frac{1}{(2a)^n}
  \qHyper32{q^{-n},ae^{i\theta},ae^{-i\theta}}{0,0}{q,q}
  = \frac{e^{in\theta}}{2^n}
  \qHyper20{q^{-n},ae^{i\theta}}{-}{q,q^n e^{-2i\theta}}, \\
  Q_n(x;a,b|q)
  &= \frac{(ab;q)_{n}}{(2a)^n}
  \qHyper32{q^{-n},ae^{i\theta},ae^{-i\theta}}{ab,0}{q,q},\\
  p^{dH}_n(x; a, b, c|q)
  &= \frac{(ab, ac; q)_n}{(2a)^n}\qHyper32{q^{-n},a e^{i\theta}, a e^{-i\theta}}{ab, ac}{q,q}.
\end{align*}
Here, \( x=\cos\theta \). Note that these polynomials are special
cases of the Askey--Wilson polynomials \( p_n(x;a,b,c,d|q) \) defined
in \eqref{eq:AW-def}, namely,
\begin{align*}
 H_n(x|q) &= p_n(x;0,0,0,0|q),\\
 H_n(x;a|q) &= p_n(x;a,0,0,0|q),\\
 Q_n(x;a,b|q) &= p_n(x;a,b,0,0|q),\\
 p^{dH}_n(x; a, b, c|q) &= p_n(x;a,b,c,0|q).
\end{align*}

Let \( \sigma^{H}_{n,k}(q) \) be the mixed moments of the continuous
\( q \)-Hermite polynomials \( \{H_n(x|q)\}_{n\ge0} \):
\[
 x^n  = \sum_{k=0}^n \sigma^{H}_{n,k}(q) H_k(x|q).
\]
The Touchard--Riordan-like formula due to Josuat-Verg\`es
\cite[Proposition 5.1]{Josuat-Verges2011a} and Cigler and Zeng
\cite[Proposition 15]{Cigler2011a} (see also \cite[(12)]{KS15}) can be
stated as
\[
  \sigma^H_{n,k}(q)= \frac{1}{2^{n-k}} \sum_{r=k}^n \left(
    \binom{n}{\frac{n-r}2} - \binom{n}{\frac{n-r}2-1}\right)
  (-1)^{(r-k)/2} q^{\binom{(r-k)/2+1}2}
  \qbinom{\frac{r+k}2}{\frac{r-k}2},
\]
where we define \( \binom{i}{j}=\qbinom{i}{j}=0 \) unless both \( i \) and
\( j \) are integers.

We denote by \( \sigma^{bH,H}_{n,k}(a;q) \),
\( \sigma^{Q,H}_{n,k}(a,b;q) \), \( \sigma^{dH,H}_{n,k}(a,b,c;q) \),
and \( \sigma^{AW,H}_{n,k}(a,b,c,d;q) \) the mixed moments of
continuous big \( q \)-Hermite polynomials, Al--Salam--Chihara
polynomials, continuous dual \( q \)-Hahn polynomials, and
Askey--Wilson polynomials relative to the continuous \( q \)-Hermite
polynomials:
\begin{align}
  \label{eq:8}
  H_n(x|q)  &= \sum_{k=0}^n \sigma^{bH,H}_{n,k}(a;q) H_k(x;a|q),\\
  \notag
 H_n(x|q)  &= \sum_{k=0}^n \sigma^{Q,H}_{n,k}(a,b;q) Q_k(x;a,b|q),\\
  \notag
 H_n(x|q)  &= \sum_{k=0}^n \sigma^{dH,H}_{n,k}(a,b,c;q) p^{dH}_k(x;a,b,c|q),\\
  \notag
 H_n(x|q)  &= \sum_{k=0}^n \sigma^{AW,H}_{n,k}(a,b,c,d;q) p_k(x;a,b,c,d|q).
\end{align}

\subsection{Another bootstrapping method}

By re-normalizing the results in \cite[(8)-(11)]{KS15}, we have
\begin{align}
\label{eq:cc0}
H_n(x|q) &= \sum_{k=0}^n \qbinom{n}{k} \left( \frac{a}{2} \right)^{n-k} H_k(x;a|q) ,\\
\notag
H_n(x;a|q) &=\sum_{k=0}^n \qbinom nk \left( \frac{b}{2} \right)^{n-k} Q_k(x;a,b|q) ,\\
\notag
  Q_n(x;a,b|q) &= \sum_{k=0}^n \qbinom nk \left( \frac{c}{2} \right)^{n-k} (abq^k;q)_{n-k}
                 p^{dH}_k(x;a,b,c|q) ,\\
\notag
p^{dH}_n(x;a,b,c|q) &=\sum_{k=0}^n \qbinom{n}{k}
                 \left( \frac{d}{2} \right)^{n-k}
                 \frac{(abq^k,acq^k,bcq^k;q)_{n-k}}{(abcdq^{2k};q)_{n-k}}
                 p_k(x;a,b,c,d|q).
\end{align}
We will first find weight systems for the coefficients in the sums of
the above equations. Then, using \Cref{lem:inter-mixed2}, we can build
a weight system for \( \sigma^{AW,H}_{n,k}(a,b,c,d;q) \) by
successively stacking these weight systems.

See Figures~\ref{fig:w(1)}, \ref{fig:w(2)}, and \ref{fig:w(3)} for the
weight systems \( w^{(1)} \), \( w^{(2)} \), and \( w^{(3)} \),
respectively in the following lemma.

\begin{lem}\label{lem:hw(1-4)}
  Let \( w^{(1)} \), \( w^{(2)} \), \( w^{(3)} \), and \( w^{(4)} \)
  be the weight systems defined as follows:
\begin{align}
\label{eq:6}  w^{(1)}(t;i,j)
  &=  \begin{cases}
       aq^j/2 & \mbox{if \( t=0 \)},\\
       0 & \mbox{if \( t\ge1 \),}
      \end{cases} \\
\notag w^{(2)}(t;i,j)
  &=  \begin{cases}
       bq^j/2 & \mbox{if \( t=0 \)},\\
       0 & \mbox{if \( t\ge1 \),}
      \end{cases} \\
\notag  w^{(3)}(t;i,j)
  &=  \begin{cases}
       cq^j/2 & \mbox{if \( t=0 \)},\\
       -abcq^{i+j}/2 & \mbox{if \( t=1 \)},\\
       0 & \mbox{if \( t\ge2 \),}
      \end{cases} \\
\notag  w^{(4)}(t;i,j)
  &=  \frac{d}{2}(abcdq^{2i})^{\flr{t/8}}
  (-bcq^i)^{\chi_o(\flr{t/4})}
  (-acq^i)^{\chi_o(\flr{t/2})}
    (-abq^i)^{\chi_o(t)} q^j\\
  \notag
  &=  \frac{1}{2} \left( abcd q^{2i} \right)^m q^j \times
         \begin{cases}
         d & \mbox{if \( t=8m \)},\\
         -abd q^{i} & \mbox{if \( t=8m+1 \)},\\
         -acd q^{i} & \mbox{if \( t=8m+2 \)},\\
         a^2bcd q^{2i} & \mbox{if \( t=8m+3 \)},\\
         -bcd q^{i} & \mbox{if \( t=8m+4 \)},\\
         ab^2cd q^{2i} & \mbox{if \( t=8m+5 \)},\\
         abc^2d q^{2i} & \mbox{if \( t=8m+6 \)},\\
         -a^2b^2c^2d q^{3i} & \mbox{if \( t=8m+7 \)}.
       \end{cases}
\end{align}
 Then
\begin{align}
\label{eq:hw(1)} h^{w^{(1)}}_{n,k} &= \qbinom{n}{k} \left( \frac{a}{2} \right)^{n-k}, \\
\label{eq:hw(2)} h^{w^{(2)}}_{n,k} &= \qbinom nk \left( \frac{b}{2} \right)^{n-k}, \\
\label{eq:hw(3)} h^{w^{(3)}}_{n,k} &= \qbinom nk \left( \frac{c}{2} \right)^{n-k}(abq^k;q)_{n-k}, \\
\label{eq:hw(4)} h^{w^{(4)}}_{n,k} &= \qbinom{n}{k}
                 \left( \frac{d}{2} \right)^{n-k}
                 \frac{(abq^k,acq^k,bcq^k;q)_{n-k}}{(abcdq^{2k};q)_{n-k}}.
\end{align}
\end{lem}

\begin{proof}
  By \Cref{lem:q-binom-wt} and \Cref{lem:wt-mult}, we have
  \eqref{eq:hw(1)} and \eqref{eq:hw(2)}. By \Cref{lem:q-binom-wt} and
  \Cref{lem:two-to-one}, we have \eqref{eq:hw(3)}. To prove
  \eqref{eq:hw(4)}, observe that its right-hand side is
\[
 \qbinom{n}{k} \left( \frac{d}{2} \right)^{n-k}
 \frac{(abq^k,acq^k,bcq^k;q)_{n-k}}{(abcdq^{2k};q)_{n-k}}
 = \left( \frac{d}{2} \right)^{n-k}
 (abq^k,acq^k,bcq^k;q)_{n-k} \cdot \sigma^b_{n,k}(-abcd/q;q),
\]
where \( \sigma^b_{n,k}(a;q) \) is the mixed moment of \( q \)-Bessel
polynomials in \eqref{eq:mu^C}. Thus, by \Cref{pro:charlier-wt}, if we
define \( w_1(t;i,j)= (abcdq^{2i})^t q^j \), then
\[
  h^{w_1}_{n,k} = \sigma^b_{n,k}(-abcd/q;q).
\]
By \Cref{lem:wt-mult}, if we define
\( w_2(t;i,j)= (abcdq^{2i})^t q^j (1-abq^i)(1-acq^i)(1-bcq^i)d/2 \), then
\[
  h^{w_2}_{n,k} = \left( \frac{d}{2} \right)^{n-k}
  (abq^k,acq^k,bcq^k;q)_{n-k} \cdot \sigma^b_{n,k}(-abcd/q;q).
\]
Now by applying \Cref{lem:two-to-one-gen} to \( w_2(t;i,j) \) three
times with the factors \( (1-bcq^i) \), \( (1-acq^i) \), and
\( (1-abq^i) \) in this order, we obtain that
\( h^{w_2}_{n,k} = h^{w^{(4)}}_{n,k} \), which completes the proof of
\eqref{eq:hw(4)}.
\end{proof}

By \Cref{lem:inter-mixed2} and \Cref{lem:hw(1-4)}, we obtain the
following lemma.

\begin{lem}\label{lem:1}
  Let \( w^{(1)}, w^{(2)}, w^{(3)},  \) and \( w^{(4)} \) be the weight systems given in
  \Cref{lem:hw(1-4)}. Then
\begin{align*}
  \sigma^{bH,H}_{n,k}(a;q)
  &= h^{w^{(1)}}_{n,k}, \\
  \sigma^{Q,H}_{n,k}(a,b;q)
  &= h^{w^{(1)} \sqcup w^{(2)}}_{n,k}, \\
  \sigma^{dH,H}_{n,k}(a,b,c;q)
  &= h^{w^{(1)} \sqcup w^{(2)} \sqcup w^{(3)}}_{n,k}, \\
  \sigma^{AW,H}_{n,k}(a,b,c,d;q)
  &= h^{w^{(1)} \sqcup w^{(2)} \sqcup w^{(3)} \sqcup w^{(4)}}_{n,k}.
\end{align*}
\end{lem}

\begin{figure}
  \centering
\begin{tikzpicture}[scale=1.5]
    \LHLL{4}1 \LHlabel41
    \small
    \begin{scope}[shift={(-.5,.15)}]
      \node at (1,0/1) {$a/2$};
      \node at (2,0/2) {$a/2$};
      \node at (2,1/2) {$aq/2$};
      \node at (3,0/3) {$a/2$};
      \node at (3,1/3) {$aq/2$};
      \node at (3,2/3) {$aq^2/2$};
      \node at (4,0/4) {$a/2$};
      \node at (4,1/4) {$aq/2$};
      \node at (4,2/4) {$aq^2/2$};
      \node at (4,3/4) {$aq^3/2$};
    \end{scope}
    \node at (4.5,0.5) {\( \cdots \)};
  \end{tikzpicture}
  \caption{The weight system \( w^{(1)} \).}
  \label{fig:w(1)}
\end{figure}

\begin{figure}
\centering
\begin{tikzpicture}[scale=1.5]
    \LHLL{4}1 \LHlabel41
    \small
    \begin{scope}[shift={(-.5,.15)}]
      \node at (1,0/1) {$b/2$};
      \node at (2,0/2) {$b/2$};
      \node at (2,1/2) {$bq/2$};
      \node at (3,0/3) {$b/2$};
      \node at (3,1/3) {$bq/2$};
      \node at (3,2/3) {$bq^2/2$};
      \node at (4,0/4) {$b/2$};
      \node at (4,1/4) {$bq/2$};
      \node at (4,2/4) {$bq^2/2$};
      \node at (4,3/4) {$bq^3/2$};
    \end{scope}
    \node at (4.5,0.5) {\( \cdots \)};
  \end{tikzpicture}
  \caption{The weight system \( w^{(2)} \).}
  \label{fig:w(2)}
\end{figure}

\begin{figure}
\centering
\begin{tikzpicture}[scale=1.5]
    \LHLL{4}2 \LHlabel42
    \small
    \begin{scope}[shift={(-.5,.15)}]
      \node at (1,0/1) {$c/2$};
      \node at (2,0/2) {$c/2$};
      \node at (2,1/2) {$cq/2$};
      \node at (3,0/3) {$c/2$};
      \node at (3,1/3) {$cq/2$};
      \node at (3,2/3) {$cq^2/2$};
      \node at (4,0/4) {$c/2$};
      \node at (4,1/4) {$cq/2$};
      \node at (4,2/4) {$cq^2/2$};
      \node at (4,3/4) {$cq^3/2$};
    \end{scope}
    \begin{scope}[shift={(-.5,1.15)}]
      \node at (1,0/1) {$-abc/2$};
      \node at (2,0/2) {$-abcq/2$};
      \node at (2,1/2) {$-abcq^2/2$};
      \node at (3,0/3) {$-abcq^2/2$};
      \node at (3,1/3) {$-abcq^3/2$};
      \node at (3,2/3) {$-abcq^4/2$};
      \node at (4,0/4) {$-abcq^3/2$};
      \node at (4,1/4) {$-abcq^4/2$};
      \node at (4,2/4) {$-abcq^5/2$};
      \node at (4,3/4) {$-abcq^6/2$};
    \end{scope}
    \node at (4.5,0.5) {\( \cdots \)};
    \node at (4.5,1.5) {\( \cdots \)};
  \end{tikzpicture}
  \caption{The weight system \( w^{(3)} \).}
  \label{fig:w(3)}
\end{figure}

Rewriting
\( h^{w^{(1)} \sqcup w^{(2)} \sqcup w^{(3)} \sqcup w^{(4)}}_{n,k} \)
in \Cref{lem:1} we obtain the following combinatorial model for
\( \sigma^{AW,H}_{n,k}(a,b,c,d;q) \).

\begin{thm}\label{thm:1}
  Let \( w^{AW,H} \) be the weight system defined by 
\begin{align*}
   w^{AW,H}(t;i,j)
  &= (-a^2q^{i})^{\delta_{t,0}} \frac{d}{2}(abcdq^{2i})^{\flr{(t-4)/8}}
  (-bcq^i)^{\chi_e(\flr{t/4})}
  (-acq^i)^{\chi_o(\flr{t/2})}
    (-abq^i)^{\chi_o(t)} q^j\\
  &=  \frac{1}{2} \left( abcd q^{2i} \right)^m q^{j} \times
         \begin{cases}
          a & \mbox{if \( t=8m=0 \)},\\
          -a^{-1}q^{-i} & \mbox{if \( t=8m\ne0 \)},\\
         b & \mbox{if \( t=8m+1 \)},\\
         c & \mbox{if \( t=8m+2 \)},\\
         -abc q^{i} & \mbox{if \( t=8m+3 \)},\\
         d & \mbox{if \( t=8m+4 \)},\\
         -abd q^{i} & \mbox{if \( t=8m+5 \)},\\
         -acd q^{i} & \mbox{if \( t=8m+6 \)},\\
         a^2bcd q^{2i} & \mbox{if \( t=8m+7 \)}.
       \end{cases}
\end{align*}
Then we have
  \[
    h^{w^{AW,H}}_{n,k} = \sigma^{AW,H}_{n,k}(a,b,c,d;q).
  \]
\end{thm}

\subsection{A weight system for continuous $q$-Hermite polynomials}
\label{sec:continuous-q-Hermite}

In order to obtain a weight system for the original mixed moment
\( \sigma^{AW}_{n,k}(a,b,c,d;q) \) of Askey--Wilson polynomials from
\Cref{thm:1}, we can simply add a weight system for the mixed moment
\( \sigma^H_{n,k}(q) \) of continuous $q$-Hermite polynomials below
the weight system \( w^{AW,H} \). In this subsection we find a weight
system for \( \sigma^H_{n,k}(q) \). To do this we first consider
continuous big $q$-Hermite polynomials \( H_n(x;a|q) \).

Let \( \sigma_{n,k}^{bH}(a;q) \) be the mixed moment of continuous big
\( q \)-Hermite polynomials:
\[
   x^n  = \sum_{k=0}^n \sigma^{bH}_{n,k}(a;q) H_k(x;a|q).
\]
Since \( H_n(x;a|q)= p_n(x;a,0,0,0|q) \), setting \( b=c=d=0 \) in
\Cref{thm:AW-wt-full} gives the following proposition; see
\Cref{fig:CBH-lht} for the weight system.

\begin{prop}\label{pro:3}
  Let \( w^{bH} \) be the weight system of height \( 2 \) defined by
  \[
    w^{bH}(t;i,j) =
    \begin{cases}
      (aq^{j}+a^{-1}q^{-j})/2 & \mbox{if \( t=0 \)},\\
      -1/(2aq^{i-j})  & \mbox{if \( t=1 \)}.
    \end{cases}
  \]  
 Then
  \[
    \sigma_{n,k}^{bH}(a;q) = h^{w^{bH}}_{n,k}.
  \]
\end{prop}

\begin{figure}
  \centering
  \begin{tikzpicture}[scale=1.5]
    \LHLL{4}2 \LHlabel42
    \small
    \begin{scope}[shift={(-.5,1.15)}]
      \node at (1,0/1) {$- 1/(2a)$};
      \node at (2,0/2) {$- 1/(2aq)$};
      \node at (2,1/2) {$- 1/(2a)$};
      \node at (3,0/3) {$- 1/(2aq^2)$};
      \node at (3,1/3) {$- 1/(2aq)$};
      \node at (3,2/3) {$- 1/(2a)$};
      \node at (4,0/4) {$- 1/(2aq^3)$};
      \node at (4,1/4) {$- 1/(2aq^2)$};
      \node at (4,2/4) {$- 1/(2aq)$};
      \node at (4,3/4) {$- 1/(2a)$};
    \end{scope}
    \begin{scope}[shift={(-.5,.15)}]
      \node at (1,0/1) {$(a+)/2$};
      \node at (2,0/2) {$(a+)/2$};
      \node at (2,1/2) {$(aq+)/2$};
      \node at (3,0/3) {$(a+)/2$};
      \node at (3,1/3) {$(aq+)/2$};
      \node at (3,2/3) {$(aq^2+)/2$};
      \node at (4,0/4) {$(a+)/2$};
      \node at (4,1/4) {$(aq+)/2$};
      \node at (4,2/4) {$(aq^2+)/2$};
      \node at (4,3/4) {$(aq^3+)/2$};
    \end{scope}
    \node at (4.5,0.5) {\( \cdots \)};
    \node at (4.5,1.5) {\( \cdots \)};
  \end{tikzpicture}
  \caption{The lecture hall graph for the mixed moments of
    continuous big \( q \)-Hermite polynomials. Here, \( (x+) \) means
    \( x+x^{-1} \).}
  \label{fig:CBH-lht}
\end{figure}

\begin{figure}
  \centering
  \begin{tikzpicture}[scale=1.5]
    \LHLL{4}1 \LHlabel41
    \small
    \begin{scope}[shift={(-.5,0.15)}]
      \node at (1,0/1) {$- a/2$};
      \node at (2,0/2) {$- aq/2$};
      \node at (2,1/2) {$- a/2$};
      \node at (3,0/3) {$- aq^2/2$};
      \node at (3,1/3) {$- aq/2$};
      \node at (3,2/3) {$- a/2$};
      \node at (4,0/4) {$- aq^3/2$};
      \node at (4,1/4) {$- aq^2/2$};
      \node at (4,2/4) {$- aq/2$};
      \node at (4,3/4) {$- a/2$};
    \end{scope}
    \node at (4.5,0.5) {\( \cdots \)};
  \end{tikzpicture}
  \caption{The weight system \( w' \) in \Cref{lem:3}.}
  \label{fig:w'}
\end{figure}

To find a weight system for \( \sigma^H_{n,k}(q) \), we compare two
polynomials \( H_n(x|q) \) and \( H_n(x;a|q) \). See \Cref{fig:w'} for
the weight system in the following lemma.

\begin{lem}\label{lem:3}
  Let \( w' \) be the weight system of height \( 1 \) defined by
  \[
    w'(0;i,j) = -\frac{aq^{i-j}}{2}.
  \]
Then
\[
H_n(x;a|q) = \sum_{k=0}^n h^{w'}_{n,k} H_k(x|q).
\]
\end{lem}
\begin{proof}
  Let \( w^{(1)} \) be the weight system given in \eqref{eq:6}. By
  \eqref{eq:8} and \Cref{lem:1}, we have
\[
H_n(x|q) = \sum_{k=0}^n h^{w^{(1)}}_{n,k} H_k(x;a|q).
\]
By \Cref{pro:h-e-dual}, the above equation implies
\begin{equation}\label{eq:14}
H_n(x;a|q) = \sum_{k=0}^n (-1)^{n-k} e^{w^{(1)}}_{n,k} H_k(x|q).
\end{equation}
Since \( w^{(1)} \) is a weight system of height \( 1 \), by
\Cref{pro:w=overline w}, we have
\( e^{w^{(1)}}_{n,k} = h^{\overline{w^{(1)}}}_{n,k} \), where
\( \overline{w^{(1)}} \) is the weight system of height \( 1 \)
defined by
\( \overline{w^{(1)}} (0;i,j) = w^{(1)}(0;i,i-j) = aq^{i-j}/2 \).
Since \( w'(0;i,j) = - \overline{w^{(1)}}(0;i,j) \), \eqref{eq:14} is
equivalent to the identity in this proposition.
\end{proof}

By \Cref{lem:inter-mixed2}, a weight system for
\(\sigma_{n,k}^{H}(q) \) can be obtained by adding the weight system
\( w' \) in \Cref{lem:3} at the top of a weight system of
\(\sigma_{n,k}^{bH}(a;q) \). Therefore, by \Cref{pro:3},
\( w^{bH}\sqcup w' \) is a weight system for \(\sigma_{n,k}^{H}(q) \).
However, since \(\sigma_{n,k}^{H}(q) \) is independent of \( a \), we
can replace \( a \) by any number. In particular, if we set \( a=1 \)
in the weight system \( w^{bH}\sqcup w' \), we obtain the following
proposition; see \Cref{fig:continuous-q-Hermite-three-rows} for the
weight system.

\begin{prop}\label{pro:4}
  Let \( w^H \) be the weight system of height \( 3 \) defined by
  \[
    w^{H}(t;i,j) =
    \begin{cases}
      (q^{j}+q^{-j})/2 & \mbox{if \( t=0 \)},\\
      -q^{j-i}/2  & \mbox{if \( t=1 \)},\\
      -q^{i-j}/2  & \mbox{if \( t=2 \)}.
    \end{cases}
  \]  
 Then
  \[
    \sigma^H_{n,k}(q) = h^{w^H}_{n,k}.
  \]
\end{prop}

\begin{figure}
  \centering
  \begin{tikzpicture}[scale=1.5]
    \LHLL{4}3 \LHlabel43
    \small
    \begin{scope}[shift={(-.5,2.15)}]
      \node at (1,0/1) {$- 1/2$};
      \node at (2,0/2) {$- q/2$};
      \node at (2,1/2) {$- 1/2$};
      \node at (3,0/3) {$- q^2/2$};
      \node at (3,1/3) {$- q/2$};
      \node at (3,2/3) {$- 1/2$};
      \node at (4,0/4) {$- q^3/2$};
      \node at (4,1/4) {$- q^2/2$};
      \node at (4,2/4) {$- q/2$};
      \node at (4,3/4) {$- 1/2$};
    \end{scope}
    \begin{scope}[shift={(-.5,1.15)}]
      \node at (1,0/1) {$- 1/2$};
      \node at (2,0/2) {$- 1/(2q)$};
      \node at (2,1/2) {$- 1/2$};
      \node at (3,0/3) {$- 1/(2q^2)$};
      \node at (3,1/3) {$- 1/(2q)$};
      \node at (3,2/3) {$- 1/2$};
      \node at (4,0/4) {$- 1/(2q^3)$};
      \node at (4,1/4) {$- 1/(2q^2)$};
      \node at (4,2/4) {$- 1/(2q)$};
      \node at (4,3/4) {$- 1/2$};
    \end{scope}
    \begin{scope}[shift={(-.5,.15)}]
      \node at (1,0/1) {$(1+)/2$};
      \node at (2,0/2) {$(1+)/2$};
      \node at (2,1/2) {$(q+)/2$};
      \node at (3,0/3) {$(1+)/2$};
      \node at (3,1/3) {$(q+)/2$};
      \node at (3,2/3) {$(q^2+)/2$};
      \node at (4,0/4) {$(1+)/2$};
      \node at (4,1/4) {$(q+)/2$};
      \node at (4,2/4) {$(q^2+)/2$};
      \node at (4,3/4) {$(q^3+)/2$};
    \end{scope}
    \node at (4.5,0.5) {\( \cdots \)};
    \node at (4.5,1.5) {\( \cdots \)};
    \node at (4.5,2.5) {\( \cdots \)};
  \end{tikzpicture}
  \caption{The weight system \( w^H \) for the mixed moments of
    continuous \( q \)-Hermite polynomials. Here, \( (x+) \) means
    \( x+x^{-1} \).}
  \label{fig:continuous-q-Hermite-three-rows}
\end{figure}

\begin{figure}
  \centering
  \begin{tikzpicture}[scale=1.5]
    \LHLLL{4}9
    \small
    \begin{scope}[shift={(-.5,.15)}]
      \node at (1,0/1) {$a$};
      \node at (2,0/2) {$a$};
      \node at (2,1/2) {$aq$};
      \node at (3,0/3) {$a$};
      \node at (3,1/3) {$aq$};
      \node at (3,2/3) {$aq^2$};
      \node at (4,0/4) {$a$};
      \node at (4,1/4) {$aq$};
      \node at (4,2/4) {$aq^2$};
      \node at (4,3/4) {$aq^3$};
    \end{scope}
    \begin{scope}[shift={(-.5,.15)}]
      \node at (1,1/1) {$ b$};
      \node at (1,2/1) {$ c$};
      \node at (1,3/1) {$ a b c$};
      \node at (1,4/1) {$ d$};
      \node at (1,5/1) {$ a b d$};
      \node at (1,6/1) {$ a c d$};
      \node at (1,7/1) {$ a^{2} b c d$};
      \node at (1,8/1) {$ b c d$};
      \node at (1,9/1) {$a b^2 c d$};
      \node at (2,2/2) {$ b$};
      \node at (2,3/2) {$ b q$};
      \node at (2,4/2) {$ c$};
      \node at (2,5/2) {$ c q$};
      \node at (2,6/2) {$ a b c q$};
      \node at (2,7/2) {$ a b c q^{2}$};
      \node at (2,8/2) {$ d$};
      \node at (2,9/2) {$ d q$};
      \node at (2,10/2) {$ a b d q$};
      \node at (2,11/2) {$ a b d q^{2}$};
      \node at (2,12/2) {$ a c d q$};
      \node at (2,13/2) {$ a c d q^{2}$};
      \node at (2,14/2) {$ a^{2} b c d q^{2}$};
      \node at (2,15/2) {$ a^{2} b c d q^{3}$};
      \node at (2,16/2) {$ b c d q$};
      \node at (2,17/2) {$ b c d q^2$};
      \node at (2,18/2) {$a b^2 c d\, q^{2}$};
      \node at (3,3/3) {$ b$};
      \node at (3,4/3) {$ b q$};
      \node at (3,5/3) {$ b q^{2}$};
      \node at (3,6/3) {$ c$};
      \node at (3,7/3) {$ c q$};
      \node at (3,8/3) {$ c q^{2}$};
      \node at (3,9/3) {$ a b c q^{2}$};
      \node at (3,10/3) {$ a b c q^{3}$};
      \node at (3,11/3) {$ a b c q^{4}$};
      \node at (3,12/3) {$ d$};
      \node at (3,13/3) {$ d q$};
      \node at (3,14/3) {$ d q^{2}$};
      \node at (3,15/3) {$ a b d q^{2}$};
      \node at (3,16/3) {$ a b d q^{3}$};
      \node at (3,17/3) {$ a b d q^{4}$};
      \node at (3,18/3) {$ a c d q^{2}$};
      \node at (3,19/3) {$ a c d q^{3}$};
      \node at (3,20/3) {$ a c d q^{4}$};
      \node at (3,21/3) {$ a^{2} b c d q^{4}$};
      \node at (3,22/3) {$ a^{2} b c d q^{5}$};
      \node at (3,23/3) {$ a^{2} b c d q^{6}$};
      \node at (3,24/3) {$ b c d q^{2}$};
      \node at (3,25/3) {$ b c d q^{3}$};
      \node at (3,26/3) {$ b c d q^{4}$};
      \node at (3,27/3) {$a b^2 c d\, q^{4}$};
      \node at (4,4/4) {$ b$};
      \node at (4,5/4) {$ b q$};
      \node at (4,6/4) {$ b q^{2}$};
      \node at (4,7/4) {$ b q^{3}$};
      \node at (4,8/4) {$ c$};
      \node at (4,9/4) {$ c q$};
      \node at (4,10/4) {$ c q^{2}$};
      \node at (4,11/4) {$ c q^{3}$};
      \node at (4,12/4) {$ a b c q^{3}$};
      \node at (4,13/4) {$ a b c q^{4}$};
      \node at (4,14/4) {$ a b c q^{5}$};
      \node at (4,15/4) {$ a b c q^{6}$};
      \node at (4,16/4) {$ d$};
      \node at (4,17/4) {$ d q$};
      \node at (4,18/4) {$ d q^{2}$};
      \node at (4,19/4) {$ d q^{3}$};
      \node at (4,20/4) {$ a b d q^{3}$};
      \node at (4,21/4) {$ a b d q^{4}$};
      \node at (4,22/4) {$ a b d q^{5}$};
      \node at (4,23/4) {$ a b d q^{6}$};
      \node at (4,24/4) {$ a c d q^{3}$};
      \node at (4,25/4) {$ a c d q^{4}$};
      \node at (4,26/4) {$ a c d q^{5}$};
      \node at (4,27/4) {$ a c d q^{6}$};
      \node at (4,28/4) {$ a^{2} b c d q^{6}$};
      \node at (4,29/4) {$ a^{2} b c d q^{7}$};
      \node at (4,30/4) {$ a^{2} b c d q^{8}$};
      \node at (4,31/4) {$ a^{2} b c d q^{9}$};
      \node at (4,32/4) {$ b c d q^{3}$};
      \node at (4,33/4) {$ b c d q^{4}$};
      \node at (4,34/4) {$ b c d q^{5}$};
      \node at (4,35/4) {$ b c d q^{6}$};
      \node at (4,36/4) {$a b^2 c d\, q^{6}$};
    \end{scope}

    \node at (0.5,9.7) {\( \vdots \)};
    \node at (1.5,9.7) {\( \vdots \)};
    \node at (2.5,9.7) {\( \vdots \)};
    \node at (3.5,9.7) {\( \vdots \)};
    \node at (4.5,0.5) {\( \cdots \)};
    \node at (4.5,1.5) {\( \cdots \)};
    \node at (4.5,2.5) {\( \cdots \)};
    \node at (4.5,3.5) {\( \cdots \)};
    \node at (4.5,4.5) {\( \cdots \)};
    \node at (4.5,5.5) {\( \cdots \)};
    \node at (4.5,6.5) {\( \cdots \)};
    \node at (4.5,7.5) {\( \cdots \)};
    \node at (4.5,8.5) {\( \cdots \)};
  \end{tikzpicture}
  \caption{The lecture hall graph for the mixed moments of (rescaled)
    Askey--Wilson polynomials relative to continuous \( q \)-Hermite
    moments. Except for the first row, the rows \( 2,3,\dots,9 \) repeat
    modulo \( 8 \) with an extra factor of \( abcdq^{2i} \) for every
    \( 8 \) rows.}
  \label{fig:AW-lht-hermite}
\end{figure}

Using \Cref{thm:1} and \Cref{pro:4}, we obtain another weight
system for the original mixed moments
\( \sigma^{AW}_{n,k}(a,b,c,d;q) \) of Askey--Wilson polynomials.

\begin{thm}\label{thm:2}
  Let \( w^{AW} = w^{H} \sqcup w^{AW,H} \), where \( w^{AW,H} \) is
  the weight system in \Cref{thm:1} and \( w^H \) is the weight system
  in \Cref{pro:4}. Then we have
  \[
    h^{w^{AW}}_{n,k} = \sigma^{AW}_{n,k}(a,b,c,d;q).
  \]
\end{thm}

\section{Combinatorial properties of Askey--Wilson polynomials}
\label{sec:comb-prop-askey}

In this section, we study in more detail the weight system
\( w^{AW,H} \) in \Cref{thm:1} for the mixed moment
\( \sigma^{AW,H}_{n,k} \) of Askey--Wilson polynomials relative to
continuous \( q \)-Hermite polynomials. We then find a more efficient
combinatorial model for \( \sigma^{AW,H}_{n,k} \). Using our new
combinatorial model we give the first combinatorial proof of the
well-known fact that the mixed moments (and hence the coefficients as
well) of Askey--Wilson polynomials are symmetric in \( a,b,c,d \). As
applications we also find some interesting properties of Askey--Wilson
polynomials.

\subsection{Another combinatorial model for mixed moments}

Let \( \ts^{AW,H}_{n,k}(a,b,c,d;q) \) be the following rescaled mixed
moments of Askey--Wilson polynomials relative to continuous
\( q \)-Hermite polynomials:
\[
  \ts^{AW,H}_{n,k}(a,b,c,d;q) := \frac{2^{n-k}}{\vi^{n-k}}
  \sigma^{AW,H}_{n,k}(\vi a,\vi b,\vi c,\vi d;q),
\]
where \( \vi \) is the imaginary number \( \sqrt{-1} \). Using the
substitution \( (a,b,c,d)\mapsto (\vi a,\vi b,\vi c,\vi d) \), we can
restate \Cref{thm:1} as follows; see \Cref{fig:AW-lht-hermite} for the
weight system.

\begin{thm}\label{thm:4}
  Let \( \widetilde{w}^{AW,H} \) be the weight system defined by 
  \begin{equation}\label{eq:pos-wt}
       \widetilde{w}^{AW,H}(t;i,j)
  = \left( abcd q^{2i} \right)^m q^{j} \times
         \begin{cases}
          a & \mbox{if \( t=8m=0 \)},\\
          a^{-1}q^{-i} & \mbox{if \( t=8m\ne 0 \)},\\
         b & \mbox{if \( t=8m+1 \)},\\
         c & \mbox{if \( t=8m+2 \)},\\
         abc q^{i} & \mbox{if \( t=8m+3 \)},\\
         d & \mbox{if \( t=8m+4 \)},\\
         abd q^{i} & \mbox{if \( t=8m+5 \)},\\
         acd q^{i} & \mbox{if \( t=8m+6 \)},\\
         a^2bcd q^{2i} & \mbox{if \( t=8m+7 \)}.
       \end{cases}
  \end{equation}
Then we have
  \[
    h^{\widetilde{w}^{AW,H}}_{n,k} = \ts^{AW,H}_{n,k}(a,b,c,d;q).
  \]
\end{thm}

Observe that \( \widetilde{w}^{AW,H}(t;i,j) \) in \eqref{eq:pos-wt} is
a monomial in \( a,b,c,d \), and \( q \). Thus \Cref{thm:4} implies
that \( \ts^{AW,H}_{n,k}(a,b,c,d;q) \) is a formal power series in
\( a,b,c,d,q \) with nonnegative integer coefficients.
More generally, we have the following corollary.

\begin{cor}\label{cor:1}
  The matrix
  \[
    \left( \ts^{AW,H}_{n,k}(a,b,c,d;q) \right)_{n,k=0}^\infty
  \]
  is totally positive. More precisely, every minor of this matrix is
  a formal power series in \( a,b,c,d \), and \( q \) with nonnegative
  integer coefficients.
\end{cor}

We will find another combinatorial model for
\( \ts^{AW,H}_{n,k}(a,b,c,d;q) \) derived from \Cref{thm:4}. To do
this, we need some definitions. First, observe that every weight
\( \widetilde{w}^{AW,H}(t;i,j) \) in \Cref{thm:4} is of the form
\( a^{i_1}b^{i_2}c^{i_3}d^{i_4}q^j \) such that \( i_1+i_2+i_3+i_4 \)
is odd and \( |i_r-i_s|\le 1 \) for all \( 1\le r,s\le 4 \).

\begin{defn}\label{def:1}
Let \( \TT \) denote the set of quadruples \( (t_1,t_2,t_3,t_4) \) of
nonnegative integers satisfying the following conditions:
\begin{enumerate}
\item \( t_1+t_2+t_3+t_4 \) is odd and 
\item \( |t_i-t_j|\le 1 \) for all \( 1\le i,j\le 4 \).
\end{enumerate}
We define a total order \( \le \) on \( \TT \) by
\[
  (t_1,t_2,t_3,t_4) \le (s_1,s_2,s_3,s_4)
  \quad \mbox{if and only if} \quad 
  (t_4,t_3,t_2,-t_1) \le_{\mathrm{lex}} (s_4,s_3,s_2,-s_1),
\]
where \( \le_{\mathrm{lex}} \) is the lexicographic order. In other words,
\( (t_1,t_2,t_3,t_4) \le (s_1,s_2,s_3,s_4) \) if and only if
\( t_4<s_4 \) or (\( t_4=s_4 \) and \( t_3<s_3 \)) or
(\( t_4=s_4, t_3=s_3 \) and \( t_2<s_2 \)) or
(\( t_4=s_4, t_3=s_3, t_2=s_2 \) and \( -t_1\le -s_1 \)).
\end{defn}

For example, \( (3,3,4,3) \le (3,3,3,4) \) and
\( (4,3,3,3) \le (2,3,3,3) \). The motivation of the above definition
is the following lemma, which says that \( (\TT,\le) \) is essentially
the totally ordered set on the tuples \( (t_1,t_2,t_3,t_4) \) of the powers of
\( a,b,c,d \) appearing in the weights
\( \widetilde{w}^{AW,H}(t;i,j) \) ordered by their heights \( t \).

\begin{lem}\label{lem:2}
  For a nonnegative integer \( t \), define \( \kappa(t) \) to be the
  tuple \( (t_1,t_2,t_3,t_4) \) of integers satisfying
  \[
    a^{t_1} b^{t_2} c^{t_3} d^{t_4}
  = \left( abcd \right)^m \times
         \begin{cases}
          a & \mbox{if \( t=8m=0 \)},\\
          a^{-1} & \mbox{if \( t=8m\ne 0 \)},\\
         b & \mbox{if \( t=8m+1 \)},\\
         c & \mbox{if \( t=8m+2 \)},\\
         abc & \mbox{if \( t=8m+3 \)},\\
         d & \mbox{if \( t=8m+4 \)},\\
         abd & \mbox{if \( t=8m+5 \)},\\
         acd & \mbox{if \( t=8m+6 \)},\\
         a^2bcd & \mbox{if \( t=8m+7 \)}.
       \end{cases}
  \]
  Then \( \kappa: \ZZ_{\ge0} \to \TT \) is a bijection such that, for
  \( t,s\in \ZZ_{\ge0} \), we have \( t\le s \) if and only if
  \( \kappa(t) \le \kappa(s) \). Moreover, if
  \( \kappa(t) = (t_1,t_2,t_3,t_4) \), then
  for all \( i,j \in \ZZ_{\ge0} \) with \( j\le i \), we have
\[
  \widetilde{w}^{AW,H}(t;i,j)
  = a^{t_1} b^{t_2} c^{t_3} d^{t_4} q^{i(t_1+t_2+t_3+t_4-1)/2} \cdot q^j.
\]
\end{lem}

\begin{proof}
  Both statements are immediate from the definition of the totally ordered set
  \( \TT \) and \Cref{thm:4}.
\end{proof}

From now on, we will use the parameters \( a_1,a_2,a_3,a_4 \) in place
of \( a,b,c,d \), respectively. For any quadruple
\( T=(t_1,t_2,t_3,t_4) \) of nonnegative integers, we define
\begin{align*}
  \va^T  &:= a_1^{t_1}a_2^{t_2}a_3^{t_3}a_4^{t_4},\\
  |T|  &:= t_1 +t_2+t_3+t_4.
\end{align*}
For integers \( n\ge k\ge 0 \), let
\[
  \TT_{n,k} = \{ (T_{k},T_{k+1},\dots,T_{n-1})\in \TT^{n-k}: T_{k} \ge T_{k+1} \ge \cdots \ge T_{n-1} \}.
\]
Note that the set \( \TT_{n,k} \) depends only on the difference
\( n-k \). However, we keep both \( n \) and \( k \) in the notation
\( \TT_{n,k} \) to emphasize that we index the elements of a tuple
\( (T_{k},T_{k+1},\dots,T_{n-1})\in \TT_{n,k} \) from \( k \) to
\( n-1 \).

For \( \vT=(T_{k},\dots,T_{n-1}) \in \TT_{n,k} \), the \emph{sum}
\( s(\vT) \) and the \emph{multiplicity} \( m(\vT) \) of \( \vT \) are
defined by
\begin{align*}
  s(\vT) &:= T_k + \cdots + T_{n-1},\\
  m(\vT) &:= (m_1,\dots,m_r),
\end{align*}
where \( T_{k} + \cdots + T_{n-1} \) means the component-wise addition
and \( m_1,\dots,m_r \) are the positive integers such that
\( m_1 + \cdots + m_r = n-k \) and
\begin{equation}\label{eq:15}
  T_{k} = \cdots = T_{k+m_1-1}
  > T_{k+m_1} = \cdots = T_{k+m_1+m_2-1} > \cdots 
  > T_{k+m_1 + \cdots + m_{r-1}} = \cdots = T_{n-1}.
\end{equation}
We will use the following notation:
\begin{align*}
  \norm{\vT}_{n,k}  &:= \sum_{i=k}^{n-1}  \frac{i(|T_i|-1)}{2}, \\
  \qbinom{n}{k,m(\vT)} &:= \qbinom{n}{k,m_1,\dots,m_r}= \frac{[n]_q!}{[k]_q! [m_1]_q! \cdots [m_r]_q!},
\end{align*}
where \( m(\vT) = (m_1,\dots,m_r) \).

We are now ready to state another combinatorial formula for
\( \ts^{AW,H}_{n,k}(a_1,a_2,a_3,a_4;q) \).

\begin{thm}\label{thm:5}
  We have
  \[
    \ts^{AW,H}_{n,k}(a_1,a_2,a_3,a_4;q)
    = \sum_{\vT \in \TT_{n,k}} \va^{s(\vT)}
    q^{\norm{\vT}_{n,k}} \qbinom{n}{k,m(\vT)} .
  \]
\end{thm}

\begin{proof}
Let \( w= \widetilde{w}^{AW,H} \) be the weight system in
\Cref{thm:4} so that
\[
  \ts^{AW,H}_{n,k}(a,b,c,d;q) = \sum_{p:(k,\infty) \to (n,0)} w(p).
\]
Consider a path \( p:(k,\infty) \to (n,0) \). For \( k\le i\le n-1 \),
let \( w(t^{(i)};i,j^{(i)}) \) be the weight of the east step of
\( p \) between \( x=i \) and \( x=i+1 \). Then \( p \) is determined
by the numbers \( t^{(i)} \)'s and \( j^{(i)} \)'s. Let
\( T_i = \kappa(t^{(i)})= (t^{(i)}_1, t^{(i)}_2, t^{(i)}_3, t^{(i)}_4)
\) be the tuple defined in \Cref{lem:2}, which satisfies
\[
  \widetilde{w}^{AW,H}(t^{(i)};i,j^{(i)}) = a^{t^{(i)}_1} b^{t^{(i)}_2} c^{t^{(i)}_3} d^{t^{(i)}_4}
  q^{i(t^{(i)}_1+t^{(i)}_2+t^{(i)}_3+t^{(i)}_4-1)/2} \cdot q^{j^{(i)}}.
\]
Since \( t^{(k)} \ge \cdots \ge t^{(n-1)} \), we have
\( \vT=(T_k,\dots,T_{n-1})\in \TT_{n,k} \). Let
\( m(\vT) = (m_1,\dots,m_r) \).
Then \eqref{eq:15} is equivalent to
\[
  t^{(c_0)} = \cdots = t^{(c_1-1)}
  > t^{(c_1)} = \cdots = t^{(c_2-1)} > \cdots 
  > t^{(c_{r-1})} = \cdots = t^{(c_r-1)},
\]
where \( c_\ell= k+m_1 + \cdots + m_\ell\). Since the
\( y \)-coordinates of the starting points of the east steps must be
weakly decreasing, this implies that
\begin{equation}\label{eq:17}
  c_\ell \ge j^{(c_\ell)} \ge j^{(c_\ell+1)} \ge \cdots \ge j^{(c_{\ell+1}-1)}\ge 0.
\end{equation}

Note that
\[
  \sum q^{j^{(c_\ell)} + \cdots + j^{(c_{\ell+1}-1)}} = \qbinom{c_{\ell+1}}{m_{\ell+1}},
\]
where the sum is over all tuples
\( (j^{(c_\ell)},\dots,j^{(c_{\ell+1}-1)}) \) satisfying \eqref{eq:17}.
Hence, the sum of \( w(p) \) for all paths
\( p:(k,\infty) \to (n,0) \) corresponding to a fixed tuple
\( \vT \in \TT_{n,k} \) with
\( m(\vT) = (m_1,\dots,m_r) \) is equal to
\begin{equation}\label{eq:18}
 \va^{s(\vT)}
  \qbinom{c_1}{m_1} \cdots \qbinom{c_r}{m_r}
  q^{\norm{\vT}_{n,k}}.
\end{equation}
Since
\( \qbinom{c_1}{m_1} \cdots \qbinom{c_r}{m_r} =
\qbinom{n}{k,m_1,\dots,m_r} = \qbinom{n}{k,m(\vT)} \), summing
\eqref{eq:18} over all \( \vT \in \TT_{n,k} \) gives the theorem.
\end{proof}

Using \Cref{thm:5}, we obtain the following ``shifting'' property of
\( \ts^{AW,H}_{n,k}(a_1,a_2,a_3,a_4;q) \). Here,
\( [\va^{\vec s}] f(a_1,a_2,a_3,a_4) \) means the coefficient of
\( \va^{\vec s} \) in \( f(a_1,a_2,a_3,a_4) \).

\begin{cor}
  Let \( \vec s = (s_1,s_2,s_3,s_4) \).
  Then
  \[
    [\va^{\vec s}]  \ts^{AW,H}_{n+j,k+j}(a_1,a_2,a_3,a_4;q)
    = q^{j(|\vec s|-(n-k))/2} \frac{(q^{n+1};q)_{j}}{(q^{k+1};q)_j}
    [\va^{\vec s}] \ts^{AW,H}_{n,k}(a_1,a_2,a_3,a_4;q).
  \]
\end{cor}

\begin{proof}
  By \Cref{thm:5},
  \[
    [\va^{\vec s}]  \ts^{AW,H}_{n+j,k+j}(a_1,a_2,a_3,a_4;q)
    = \sum_{\substack{\vT \in \TT_{n+j,k+j}\\ s(\vT) = \vec s}} 
    q^{\norm{\vT}_{n+j,k+j}} \qbinom{n+j}{k+j,m(\vT)} .
  \]
  Since the two sets \( \TT_{n+j,k+j} \) and \( \TT_{n,k} \) are
  identical, we obtain the desired formula from
  \[
    q^{\norm{\vT}_{n+j,k+j}} \qbinom{n+j}{k+j,m(\vT)} =
    q^{j(|\vec s|-(n-k))/2} \frac{(q^{n+1};q)_{j}}{(q^{k+1};q)_j}
    \cdot q^{\norm{\vT}_{n,k}} \qbinom{n}{k,m(\vT)},
  \]
  which is straightforward to verify by reindexing
  $\norm{\vT}_{n+j,k+j}=\sum_{i=k+j}^{n+j-1}\tfrac{i(|T_i|-1)}{2}$
  via $i\mapsto i'+j$.
\end{proof}

\subsection{The symmetry of $a,b,c,d$ in mixed moments}

As another application of \Cref{thm:5}, we give a combinatorial proof
of the symmetry of the parameters of \( a_1,a_2,a_3,a_4 \) in
\( \ts^{AW,H}_{n,k}(a_1,a_2,a_3,a_4;q) \). To this end we introduce
some definitions.

For a function
\( f(a_1,a_2,a_3,a_4) \) in the variables \( a_1,a_2,a_3,a_4 \) and a
permutation \( \tau \in \sym_4 \), let
\[
  \tau \cdot f(a_1,a_2,a_3,a_4) =
  f(a_{\tau(1)},a_{\tau(2)},a_{\tau(3)},a_{\tau(4)}).
\]
Our goal is to show that \( \ts^{AW,H}_{n,k}(a_1,a_2,a_3,a_4;q) \) is
symmetric in \( a_1,a_2,a_3,a_4 \), that is, for any permutation
\( \tau \in \sym_4 \),
\begin{equation}\label{eq:21}
  \tau \cdot \ts^{AW,H}_{n,k}(a_1,a_2,a_3,a_4;q) =
  \ts^{AW,H}_{n,k}(a_1,a_2,a_3,a_4;q) .
\end{equation}
Since the simple transpositions generate \( \sym_4 \), it suffices to
prove \eqref{eq:21} for the cases \( \tau = (1,2) \),
\( \tau= (2,3) \), and \( \tau = (3,4) \).

An \emph{interval} of \( \TT \) is a subset \( I \) of \( \TT \) such
that if \( T_1,T_2\in I \) and \( T_1 \le T \le T_2 \), then
\( T\in I \). For an interval \( I \) of \( \TT \), we define
\[
  \TT_{n,k}(I) := \TT_{n,k} \cap I^{n-k} = \{
  (T_{k},T_{k+1},\dots,T_{n-1})\in I^{n-k}: T_{k} \ge T_{k+1} \ge
  \cdots \ge T_{n-1} \}
\]
and
\begin{equation}\label{eq:13}
    \sigma_{n,k}(I)
    = \sum_{\vT \in \TT_{n,k}(I)} \va^{s(\vT)}
    q^{\norm{\vT}_{n,k}} \qbinom{n}{k,m(\vT)} .
\end{equation}
For two intervals \( I_1 \) and \( I_2 \) of \( \TT \) we write
\( I_1 < I_2 \) if \( T_1<T_2 \) for all \( T_1\in I_1 \) and
\( T_2\in I_2 \). An \emph{increasing interval partition} of \( \TT \)
is an infinite sequence \( (I_0,I_1,\dots) \) of nonempty intervals
\( I_0 <I_1 < \cdots \) in \( \TT \) such that
\( I_0 \cup I_1 \cup \cdots = \TT \). Note that in this case the
intervals \( I_j \)'s are pairwise disjoint and indeed form a
partition of \( \TT \).

\begin{lem}\label{lem:9}
  Let \( (I_0,I_1,\dots) \) be an increasing interval partition of
  \( \TT \). Then
  \[
    \ts^{AW,H}_{n,k}(a_1,a_2,a_3,a_4;q) = \sum_{n=n_0\ge n_1\ge \cdots
      \ge k} \prod_{j\ge0} \sigma_{n_j,n_{j+1}}(I_j),
  \]
  where the sum is over all infinite sequences
  \( (n_0,n_1, \dots) \) of integers such that
  \( n=n_0\ge n_1\ge \cdots \ge k \) and \( n_j = k \) for all
  sufficiently large \( j \).
\end{lem}

\begin{proof}
  Let \( X(n_0,n_1,\dots) \) be the set of tuples
  \( \vT=(T_k,\dots,T_{n-1})\in \TT_{n,k} \) such that the number of
  elements of \( \vT \) in \( I_j \) is \( n_j-n_{j+1} \) for all \( j\ge0 \).
  Then by \Cref{thm:5},
  \begin{equation}\label{eq:20}
    \ts^{AW,H}_{n,k}(a_1,a_2,a_3,a_4;q) = \sum_{n=n_0\ge n_1\ge \cdots
      \ge k} \sum_{\vT \in X(n_0,n_1,\dots)} \va^{s(\vT)}
    q^{\norm{\vT}_{n,k}} \qbinom{n}{k,m(\vT)} .
  \end{equation}
  Consider \( \vT = (T_k,\dots,T_{n-1})\in X(n_0,n_1,\dots) \).
  For \( j\ge0 \), let 
  \[
    \vT_j = (T_{n_{j+1}}, T_{n_{j+1}+1},\dots,T_{n_j-1}).
  \]
  Then
  \[
    \va^{s(\vT)} q^{\norm{\vT}_{n,k}} = \prod_{j\ge0} \va^{s(\vT_j)}
    q^{\norm{\vT_j}_{n_{j},n_{j+1}}}
  \]
  and
  \[
   \prod_{j\ge0}  \qbinom{n_{j}}{n_{j+1},m(\vT_j)}
   = \prod_{j\ge0}  \frac{[n_j]_q!}{[n_{j+1}]_q!} \prod_{m\in m(\vT_j)} \frac{1}{[m]_q!}
   = \frac{[n]_q!}{[k]_q!} \prod_{m\in m(\vT)} \frac{1}{[m]_q!}
   = \qbinom{n}{k,m(\vT)}.
  \]
  Therefore
  \[
   \va^{s(\vT)} q^{\norm{\vT}_{n,k}} \qbinom{n}{k,m(\vT)}  
   =  \prod_{j\ge0} \va^{s(\vT_j)} q^{\norm{\vT_j}_{n_{j},n_{j+1}}} \qbinom{n_{j}}{n_{j+1},m(\vT_j)}.
  \]
  This shows that
  \begin{equation}\label{eq:19}
    \sum_{\vT \in X(n_0,n_1,\dots)} \va^{s(\vT)}
    q^{\norm{\vT}_{n,k}} \qbinom{n}{k,m(\vT)} = \prod_{j\ge0} \sigma_{n_{j},n_{j+1}}(I_j).
  \end{equation}
  By \eqref{eq:20} and \eqref{eq:19} we obtain the lemma.
\end{proof}

By \Cref{lem:9}, in order to prove that
\( \ts^{AW,H}_{n,k}(a_1,a_2,a_3,a_4;q) \) is invariant under a
permutation \( \tau\in \sym_4 \), it suffices to show that there is an
increasing interval partition \( (I_0, I_1,\dots) \) of \( \TT \) such
that \( \tau \cdot \sigma_{r,s}(I_j) = \sigma_{r,s}(I_j) \) for any
integers \( r\ge s\ge 0\) and \( j\ge0 \). We will find such an
increasing interval partition of \( \TT \) for the three cases that
\( \tau=(2,3) \), \( \tau=(3,4) \), and \( \tau=(1,2) \) in this
order.

First, we consider the simplest case \( \tau=(2,3) \).

\begin{lem}\label{lem:(23)}
  Let \( \vec I = (I_0,I^+_0,I^{++}_0,I^{--}_1,I^-_1,I_1,I^+_1,I^{++}_1,I^{--}_2,I^-_2,I_2,I^+_2,I^{++}_2,\dots) \)
  be the increasing interval partition of \( \TT \) given by
\begin{align*}
I_t  &:=  \{(t_1,t_2,t_3,t_4)\in \TT: t_4=t_3=t_2=t \},\\
I^+_t  &:=  \{(t_1,t_2,t_3,t_4)\in \TT: t_4=t, \{t_2,t_3\} = \{t,t+1\} \},\\
I^{++}_t  &:=  \{(t_1,t_2,t_3,t_4)\in \TT: t_4=t,\ t_2=t_3=t+1 \},\\
I^-_t  &:=  \{(t_1,t_2,t_3,t_4)\in \TT: t_4=t, \{t_2,t_3\} = \{t,t-1\} \},\\
I^{--}_t  &:=  \{(t_1,t_2,t_3,t_4)\in \TT: t_4=t,\ t_2=t_3=t-1 \}.
\end{align*}
Let \( \tau = (2,3) \). Then for all integers \( n\ge k\ge 0 \) and
\( I\in \vec I \), we have
\[
  \tau \cdot \sigma_{n,k}(I) = \sigma_{n,k}(I).
\]
\end{lem}

\begin{proof}
  By the definition \eqref{eq:13} of \( \sigma_{n,k}(I) \), it
  suffices to find a bijection
  \( \theta:\TT_{n,k}(I)\to \TT_{n,k}(I) \) such that if
  \( \theta(\vT) = \vS \) then
  \( \tau\cdot \va^{s(\vT)} = \va^{s(\vS)} \),
  \( \norm{\vT}_{n,k} = \norm{\vS}_{n,k} \), and
  \( m(\vT) = m(\vS) \). To do this consider a tuple
  \( \vT= (T_{k},\dots,T_{n-1}) \in \TT_{n,k}(I) \). For
  \( k\le i\le n-1 \), let
  \( T_i = (t^{(i)}_1, t^{(i)}_2, t^{(i)}_3, t^{(i)}_4) \) and define
  \( T'_i = (t^{(i)}_1, t^{(i)}_3, t^{(i)}_2, t^{(i)}_4) \). Note that
  \( |T_i| = |T'_i| \) for all \( i \). We define \( \theta(\vT) \) to
  be the weakly decreasing rearrangement \( \vS \) of
  \( (T'_{k},\dots,T'_{n-1}) \) in the total order \( \le \) on
  \( \TT \). By the construction,
  \( \theta:\TT_{n,k}(I)\to \TT_{n,k}(I) \) is a bijection such that
  if \( \theta(\vT) = \vS \) then
  \( \tau\cdot \va^{s(\vT)} = \va^{s(\vS)} \) and
  \( m(\vT) = m(\vS) \). Thus it remains to show that
  \( \norm{\vT}_{n,k} = \norm{\vS}_{n,k} \).

  If \( I \) is one of \( I_t \), \( I^{++}_t \), or \( I^{--}_t \),
  then every \( T\in I \) satisfies \( t_2=t_3 \), so \( T'=T \) and
  hence \( \vT=\vS \); in particular \( \norm{\vT}_{n,k} =
  \norm{\vS}_{n,k} \). If \( I=I^+_t \)
  (resp.~\( I=I^-_t \)), then every element \( T\in I \) is either
  \( (t,t,t+1,t) \) or \( (t,t+1,t,t) \)
  (resp.~\( (t,t,t-1,t) \) or \( (t,t-1,t,t) \)). Hence we
  always have \( |T|=|S| \) and therefore
  \[
    \norm{\vS}_{n,k} = \sum_{i=k}^{n-1} \frac{i(|S_i|-1)}{2} =
    \sum_{i=k}^{n-1} \frac{i(|T_i|-1)}{2} = \norm{\vT}_{n,k},
  \]
  which completes the proof.
\end{proof}

Before finding an increasing interval partition of \( \TT \) for
\( \tau=(3,4) \), we need some definitions. Let \( \sym(0^r,1^s) \) be
the set of words \( \pi = \pi_1 \cdots \pi_{r+s} \) consisting of
\( r \) \( 0 \)'s and \( s \) \( 1 \)'s. For such a word
\( \pi = \pi_1 \cdots \pi_{r+s} \), an \emph{inversion} is a pair
\( (i,j) \) of integers \( 1\le i<j\le r+s \) such that
\( \pi_i>\pi_j \). Let \( \inv(\pi) \) denote the number of inversions
of \( \pi \). It is well known \cite[1.7.1~Proposition]{EC1} that
  \begin{equation}\label{eq:23}
    \sum_{\pi\in \sym(0^r, 1^s)} q^{\inv(\pi)} = \qbinom{r+s}{r}.
  \end{equation}

Now we are ready to consider the case \( \tau=(3,4) \).

\begin{lem}\label{lem:(34)}
Let \( \vec I = (I_0,I_0^+,I_1,I_1^+,\dots) \) be the increasing interval partition of \( \TT \)
given by
\begin{align*}
I_t  &:=  \{(t_1,t_2,t_3,t_4)\in \TT: t_3=t_4=t\},\\
I^+_t  &:=  \{(t_1,t_2,t_3,t_4)\in \TT: \{t_3,t_4\} = \{t,t+1\} \}.
\end{align*}
Let \( \tau = (3,4) \). Then for all integers \( n\ge k\ge 0 \) and
\( I\in \vec I \), we have
\[
  \tau \cdot \sigma_{n,k}(I) = \sigma_{n,k}(I).
\]
\end{lem}

\begin{proof}
  The case \( I=I_t \) can be proved similarly as in the proof of
  \Cref{lem:(23)}. For the case \( I=I^+_t \), we investigate the summand
  in
\[
  \sigma_{n,k}(I^+_t) = \sum_{\vT \in \TT_{n,k}(I^+_t)} \va^{s(\vT)}
  q^{\norm{\vT}_{n,k}} \qbinom{n}{k,m(\vT)}.
\]
Note that \( I^+_t \) has exactly four elements \( T_1,T_2,T_3,T_4 \),
where
  \[
    T_1= (t+1,t+1,t,t+1) > T_2 = (t,t,t,t+1) > T_3 = (t+1,t+1,t+1,t) > T_4 = (t,t,t+1,t).
  \]
  Therefore, every element \( \vT \in \TT_{n,k}(I^+_t) \) is of the
  form \( \vT= (T_1^{m_1}, T_2^{m_2} , T_3^{m_3}, T_4^{m_4}) \), where
  \( T_i^{m_i} \) means the sequence \( T_i,\dots,T_i \) of \( m_i \)
  occurrences of \( T_i \). Then
  \begin{align*}
    \va^{s(\vT)} &= (a_1a_2a_3a_4)^{(n-k)t} (a_1a_2a_4)^{m_1} a_4^{m_2} (a_1a_2a_3)^{m_3} a_3^{m_4},\\
    \norm{\vT}_{n,k} &=  \sum_{i=k}^{n-1} i(2t) +  \sum_{i=k}^{k+m_1-1}i 
                       + \sum_{i=k+m_1+m_2}^{k+m_1+m_2+m_3-1}i,\\
    \qbinom{n}{k,m(\vT)} &= \qbinom{n}{k,m_1,m_2,m_3,m_4}.
  \end{align*}

  Letting \( M_1 = m_1+m_2 \), \( M_2 = m_3+m_4 \), and
  \( N = m_1+m_3 \), we can rewrite the above equations as follows:
  \begin{align*}
    \va^{s(\vT)} &= (a_1a_2a_3a_4)^{(n-k)t} a_1^{N}a_2^{N} a_3^{M_2} a_4^{M_1},\\
    \norm{\vT}_{n,k} &=  2t \left( \binom{n}{2} - \binom{k}{2} \right)
      + kN + \binom{N}{2} + m_2 m_3,\\
    \qbinom{n}{k,m(\vT)} &= \qbinom{n}{k,M_1,M_2} \qbinom{M_1}{m_1} \qbinom{M_2}{m_3}.
  \end{align*}
  Therefore,
  \begin{equation}\label{eq:16}
    \sigma_{n,k}(I^+_t)
    = \sum_{\substack{M_1+M_2=n-k\\ 0\le N\le n-k}}f(M_1,M_2,N) g(M_1,M_2,N),
  \end{equation}
  where
  \begin{align}
\notag    f(M_1,M_2,N)
    &=  (a_1a_2a_3a_4)^{(n-k)t} a_1^{N}a_2^{N} a_3^{M_2} a_4^{M_1}
      q^{2t \left( \binom{n}{2} - \binom{k}{2} \right) +kN + \binom{N}{2}}
      \qbinom{n}{k,M_1,M_2},\\
\label{eq:24}    g(M_1,M_2,N)
    &= \sum_{(m_1,m_2,m_3,m_4)\in X(M_1,M_2,N)} q^{m_2m_3} \qbinom{M_1}{m_1}
      \qbinom{M_2}{m_3},
  \end{align}
  and \( X(M_1,M_2,N) \) is the set of tuples \( (m_1,m_2,m_3,m_4) \)
  of nonnegative integers such that \( m_1+m_2 = M_1 \),
  \( m_3+m_4 = M_2 \), and \( m_1+m_3 = N \).

  Since \( \tau \cdot f(M_1,M_2,N)= f(M_2,M_1,N) \) and
  \( \tau \cdot g(M_1,M_2,N)= g(M_1,M_2,N) \), applying \( \tau \) to
  \eqref{eq:16} yields
  \begin{equation}\label{eq:22}
    \tau \cdot\sigma_{n,k}(I^+_t)
    = \sum_{\substack{M_1+M_2=n-k\\ 0\le N\le n-k}} f(M_2,M_1,N)  g(M_1,M_2,N).
  \end{equation}
  By \eqref{eq:16} and \eqref{eq:22}, in order to show that
  \( \tau \cdot \sigma_{n,k}(I^+_t) = \sigma_{n,k}(I^+_t) \), it
  suffices to verify that \( g(M_1,M_2,N) = g(M_2,M_1,N) \). More
  generally, we will prove that \( g(M_1,M_2,N) \) is independent of
  \( M_1 \) and \( M_2 \) by showing the following claim:
  \begin{equation}\label{eq:25}
    g(M_1,M_2,N) = \sum_{\pi\in \sym(0^{N}, 1^{n-k-N})} q^{\inv(\pi)}
    = \qbinom{n-k}{N}.
  \end{equation}

  To see this, consider
  \( \pi = \pi_1 \cdots \pi_{n-k}\in \sym(0^{N}, 1^{n-k-N}) \). Let
  \( m_1 \) (resp.~\( m_2 \)) be the number of \( 0 \)'s
  (resp.~\( 1 \)) in \( \pi'=\pi_1 \cdots \pi_{M_1}\) and let
  \( m_3 \) (resp.~\( m_4 \)) be the number of \( 0 \)'s
  (resp.~\( 1 \)) in \( \pi'' = \pi_{M_1+1} \cdots \pi_{n-k}\). Then
  \( (m_1,m_2,m_3,m_4)\in X(M_1,M_2,N) \) and
  \( \inv(\pi) = m_2m_3+\inv(\pi')+\inv(\pi'') \). Hence, the
  right-hand side of \eqref{eq:25} is equal to
  \[
    \sum_{(m_1,m_2,m_3,m_4)\in X(M_1,M_2,N)} q^{m_2 m_3}
    \sum_{\pi'\in \sym(0^{m_1}, 1^{m_2})} q^{\inv(\pi')} 
    \sum_{\pi''\in \sym(0^{m_3}, 1^{m_4})} q^{\inv(\pi'')} .
  \]
  By \eqref{eq:23}, this is equal to the right-hand side of
  \eqref{eq:24}. Therefore the claim \eqref{eq:25} holds and the proof
  is completed.
\end{proof}

Finally we consider the case \( \tau=(1,2) \), which is similar to the
case \( \tau=(3,4) \).

\begin{lem}\label{lem:(12)}
Let \( \vec I = (I_0,I_0^+,I_1,I_1^+,\dots) \) be the increasing interval partition of \( \TT \)
given by
\begin{align*}
I_t  &:=  \{(t_1,t_2,t_3,t_4)\in \TT: t_3=t_4=t\},\\
I^+_t  &:=  \{(t_1,t_2,t_3,t_4)\in \TT: \{t_3,t_4\} = \{t,t+1\} \}.
\end{align*}
Let \( \tau = (1,2) \). Then for all integers \( n\ge k\ge 0 \) and
\( I\in \vec I \), we have
\[
  \tau \cdot \sigma_{n,k}(I) = \sigma_{n,k}(I).
\]
\end{lem}

\begin{proof}
  Since every tuple \( (t_1,t_2,t_3,t_4)\in I^+_t \) satisfies
  \( t_1=t_2 \), the case \( I=I^+_t \) can be proved similarly as in
  the proof of \Cref{lem:(23)}. To prove the case \( I=I_t \) we use a
  similar approach as in the proof of \Cref{lem:(34)}. For \( t\ge 1 \),
  \( I_t \) has exactly four elements \( T_1,T_2,T_3,T_4 \), where
  \[
    T_1= (t,t+1,t,t) > T_2 = (t-1,t,t,t) > T_3 = (t+1,t,t,t) > T_4 = (t,t-1,t,t).
  \]
  For \( t=0 \), the tuples \( T_2 \) and \( T_4 \) have a negative
  coordinate and hence do not lie in \( \TT \), so
  \( I_0=\{T_1,T_3\} \); the analysis below remains valid in this case
  with the convention \( m_2=m_4=0 \).
  Therefore, every element \( \vT \in \TT_{n,k}(I_t) \) is of the form
  \( \vT= (T_1^{m_1}, T_2^{m_2} , T_3^{m_3}, T_4^{m_4}) \) and
  \begin{align*}
    \va^{s(\vT)} &= (a_1a_2a_3a_4)^{(n-k)t} a_1^{m_3-m_2} a_2^{m_1-m_4},\\
    \qbinom{n}{k,m(\vT)} &= \qbinom{n}{k,m_1,m_2,m_3,m_4},\\
    \norm{\vT}_{n,k} &=  \sum_{i=k}^{n-1} i(2t-1) +  \sum_{i=k}^{k+m_1-1} i
                       + \sum_{i=k+m_1+m_2}^{k+m_1+m_2+m_3-1} i.
  \end{align*}

  Letting \( M_1 = m_1+m_2 \), \( M_2 = m_3+m_4 \), and
  \( N = m_1+m_3 \), we can rewrite the above equations as follows:
  \begin{align*}
    \va^{s(\vT)} &= (a_1a_2a_3a_4)^{(n-k)t} a_1^{N-M_1}a_2^{N-M_2},\\
    \qbinom{n}{k,m(\vT)} &= \qbinom{n}{k,M_1,M_2} \qbinom{M_1}{m_1} \qbinom{M_2}{m_3},\\
    \norm{\vT}_{n,k} &=  (2t-1) \left( \binom{n}{2} - \binom{k}{2} \right)
      + kN + \binom{N}{2} + m_2 m_3.
  \end{align*}
  Then by the same argument in the proof of \Cref{lem:(34)}, we obtain
  \( \tau \cdot \sigma_{n,k}(I_t) = \sigma_{n,k}(I_t) \).
\end{proof}

By Lemmas~\ref{lem:(23)}, \ref{lem:(34)}, and \ref{lem:(12)}, we establish
the symmetry of \( a,b,c,d \) in the mixed moments of Askey--Wilson
polynomials. Let \( \nu^{AW,H}_{n,k}(a,b,c,d;q) \) be the coefficient
in the expansion
\[
  p_n(x;a,b,c,d|q)  = \sum_{k=0}^n \nu^{AW,H}_{n,k}(a,b,c,d;q)\, H_k(x|q).
\]

\begin{thm}
  The rescaled mixed moment \( \ts^{AW,H}_{n,k}(a_1,a_2,a_3,a_4;q) \)
  is symmetric in \( a_1,a_2,a_3,a_4 \). Equivalently, the mixed
  moment \( \sigma^{AW,H}_{n,k}(a,b,c,d;q) \) (and also the
  coefficient \( \nu^{AW,H}_{n,k}(a,b,c,d;q) \)) relative to
  continuous \( q \)-Hermite polynomials is symmetric in
  \( a,b,c,d \).
\end{thm}

\section*{Acknowledgments}

The authors are grateful to Donghyun Kim for fruitful discussions. They
also thank the anonymous referees for their careful reading of the
manuscript and for helpful comments. SC is partially supported by NSF
grant DMS-2054482 and ANR grants ANR-19-CE48-0011 and ANR-18-CE40-0033.
BJ and JPK are pleased to acknowledge support from ERC Advanced Grant
740900 (LogCorRM). JSK is supported by NRF grant RS-2025-00557835.

\bibliographystyle{abbrv}

\begin{thebibliography}{10}

\bibitem{Aigner2018}
M.~Aigner and G.~M. Ziegler.
\newblock {\em {Proofs from {T}}he {B}ook}.
\newblock Springer, Berlin, sixth edition, 2018.
\newblock See corrected reprint of the 1998 original [MR1723092], Including
  illustrations by Karl H. Hofmann.

\bibitem{Andrews2004a}
G.~E. Andrews and K.~Eriksson.
\newblock {\em {Integer partitions}}.
\newblock Cambridge University Press, Cambridge, 2004.

\bibitem{BME1}
M.~Bousquet-M{\'e}lou and K.~Eriksson.
\newblock Lecture hall partitions.
\newblock {\em The Ramanujan Journal}, 1(1):101--111, 1997.

\bibitem{Cigler2011a}
J.~Cigler and J.~Zeng.
\newblock {A curious {$q$}}-analogue of {H}ermite polynomials.
\newblock {\em J. Combin. Theory Ser. A}, 118(1):9--26, 2011.

\bibitem{Corteel2021a}
S.~Corteel, D.~Keating, and M.~Nicoletti.
\newblock {Arctic curves phenomena for bounded lecture hall tableaux}.
\newblock {\em Comm. Math. Phys.}, 382(3):1449--1493, 2021.

\bibitem{Corteel2020}
S.~Corteel and J.~S. Kim.
\newblock {Enumeration of bounded lecture hall tableaux}.
\newblock {\em S{\'e}min. Lothar. Comb.}, 81:b81f, 28, 2020.

\bibitem{LHT}
S.~Corteel and J.~S. Kim.
\newblock Lecture hall tableaux.
\newblock {\em Adv. Math.}, 371:107266, 35, 2020.

\bibitem{CKS}
S.~Corteel, J.~S. Kim, and D.~Stanton.
\newblock Moments of orthogonal polynomials and combinatorics.
\newblock In {\em Recent Trends in Combinatorics}, pages 545--578. Springer,
  2016.

\bibitem{trunc_LHP}
S.~Corteel and C.~D. Savage.
\newblock Lecture hall theorems, {$q$}-series and truncated objects.
\newblock {\em J. Combin. Theory Ser.~A}, 108(2):217--245, 2004.

\bibitem{Corteel2012}
S.~Corteel, R.~Stanley, D.~Stanton, and L.~Williams.
\newblock {Formulae for {A}}skey-{W}ilson moments and enumeration of staircase
  tableaux.
\newblock {\em Trans. Amer. Math. Soc.}, 364(11):6009--6037, 2012.

\bibitem{Corteel2011}
S.~Corteel and L.~K. Williams.
\newblock Tableaux combinatorics for the asymmetric exclusion process and
  {A}skey-{W}ilson polynomials.
\newblock {\em Duke Math. J.}, 159(3):385--415, 2011.

\bibitem{Flajolet1980}
P.~Flajolet.
\newblock Combinatorial aspects of continued fractions.
\newblock {\em Discrete Math.}, 32(2):125--161, 1980.

\bibitem{Fomin2000}
S.~Fomin and A.~Zelevinsky.
\newblock {Total positivity: tests and parametrizations}.
\newblock {\em Math. Intelligencer}, 22(1):23--33, 2000.

\bibitem{GR}
G.~Gasper and M.~Rahman.
\newblock {\em Basic hypergeometric series}, volume~96 of {\em Encyclopedia of
  Mathematics and its Applications}.
\newblock Cambridge University Press, Cambridge, second edition, 2004.

\bibitem{LGV}
I.~M. Gessel and X.~G. Viennot.
\newblock Determinants, paths, and plane partitions.
\newblock {\it preprint}, 1989.

\bibitem{Josuat-Verges2011a}
M.~Josuat-Verg\`es.
\newblock {Rook placements in {Y}}oung diagrams and permutation enumeration.
\newblock {\em Adv. in Appl. Math.}, 47(1):1--22, 2011.

\bibitem{KS15}
J.~S. Kim and D.~Stanton.
\newblock Bootstrapping and {A}skey-{W}ilson polynomials.
\newblock {\em J. Math. Anal. Appl.}, 421(1):501--520, 2015.

\bibitem{KLS}
R.~Koekoek, P.~A. Lesky, and R.~F. Swarttouw.
\newblock {\em Hypergeometric orthogonal polynomials and their
  {$q$}-analogues}.
\newblock Springer Monographs in Mathematics. Springer-Verlag, Berlin, 2010.

\bibitem{Nakagawa2001}
J.~Nakagawa, M.~Noumi, M.~Shirakawa, and Y.~Yamada.
\newblock {Tableau representation for {M}}acdonald's ninth variation of {S}chur
  functions.
\newblock In {\em Physics and combinatorics, 2000 ({N}agoya)}, pages 180--195.
  World Sci. Publ., River Edge, NJ, 2001.

\bibitem{LHPSavage}
C.~D. Savage.
\newblock The mathematics of lecture hall partitions.
\newblock {\em J. Combin. Theory Ser.~A}, 144:443--475, 2016.

\bibitem{EC1}
R.~P. Stanley.
\newblock {\em Enumerative Combinatorics. {V}ol. 1, second ed.}
\newblock Cambridge University Press, New York/Cambridge, 2011.

\bibitem{ViennotLN}
G.~Viennot.
\newblock Une th\'eorie combinatoire des polyn\^omes orthogonaux g\'en\'eraux.
\newblock Lecture Notes, UQAM, 1983.

\bibitem{ViennotOP}
G.~Viennot.
\newblock A combinatorial theory for general orthogonal polynomials with
  extensions and applications.
\newblock In {\em Orthogonal polynomials and applications ({B}ar-le-{D}uc,
  1984)}, volume 1171 of {\em Lecture Notes in Math.}, pages 139--157.
  Springer, Berlin, 1985.

\bibitem{Zeng2021}
J.~Zeng.
\newblock {Combinatorics of orthogonal polynomials and their moments}.
\newblock In {\em Lectures on orthogonal polynomials and special functions},
  volume 464 of {\em London Math. Soc. Lecture Note Ser.}, pages 280--334.
  Cambridge Univ. Press, Cambridge, 2021.

\end{thebibliography}

\end{document}